\documentclass[reqno]{amsart}

\usepackage{amsmath,amssymb,amsfonts,mathtools}
\usepackage{kbordermatrix}
\usepackage[english]{babel}
\usepackage{bbold}
\usepackage{bigints}
\usepackage{mathrsfs}
\usepackage{amsthm}

\newtheorem{theorem}{Theorem}[section]
\newtheorem{proposition}[theorem]{Proposition}
\newtheorem{lemma}[theorem]{Lemma}

\newtheorem{corollary}[theorem]{Corollary}
\theoremstyle{definition}
\newtheorem{definition}[theorem]{Definition}
\theoremstyle{definition}
\newtheorem{remark}[theorem]{Remark}
\theoremstyle{remark}

\theoremstyle{remark}

\theoremstyle{remark}

\newcommand{\OO}{\textrm{O}}
\newcommand{\oo}{\mathbf{0}}
\newcommand{\ii}{\mathbf{1}}
\newcommand{\I}{\textrm{I}}

\newcommand{\be}{\beta}
\newcommand{\bd}{\beta_\Diamond}
\newcommand{\ba}{\beta_\ast}

\newcommand{\V}{\mathbb V}

\newcommand{\cbb}{\mathbb c}
\newcommand{\sbb}{\mathbb s}

\newcommand{\mce}{\mathcal E}
\newcommand{\mcd}{\mathcal D}

\newcommand{\mco}{\mathcal O}

\newcommand{\beq}{\begin{equation}}
\newcommand{\eeq}{\end{equation}}
\newcommand{\bdima}{\begin{displaymath}}
\newcommand{\edima}{\end{displaymath}}

\def\R{\mathbb R}

\begin{document}

\title[On the stability of swarming systems]{On the Stability of Rotating States for Second-Order Self-Propelled Swarms} 

\author[C. Kolon]{Carl Kolon}
\email{carl@kolon.org}
\address{United States Navy}

\author[C. Medynets]{Constantine Medynets}
\email{medynets@usna.edu}
\address{Mathematics Department, United States Naval Academy, Annapolis, MD 21402}

\author[I. Popovici]{Irina Popovici}
\email{popovici@usna.edu}
\address{Mathematics Department, United States Naval Academy, Annapolis, MD 21402}

\thanks{The research of C.M. and I.P. was supported by Office of Naval Research Grant \# N0001421WX00045}





\begin{abstract} 
In this paper, we study the dynamics of a system of $n$ coupled, self-propelled particles: $\ddot r_k = (\alpha-\beta |\dot r_k|^2)\dot r_k - \frac{\gamma}{n}\sum_{m=1}^n(r_k-r_m)$, $r_k\in \mathbb R^2.$ Numerical experiments indicate that, for a large set of initial conditions, after an initial drift, the center of mass  converges to a stationary point, with each particle eventually rotating around it  with constant angular velocity. The distribution of particles on the circle need not be uniform. These limit configurations, where all particles rotate in the same direction, are termed {\it rotating states} . We prove that rotating states are stable and that every solution that starts sufficiently close, asymptotically approaches a rotating state, exponentially fast if $n$ is odd, or at a rate
that may be exponential or $\frac{1}{\sqrt t} $ if $n$ is even. The proof uses a new approximation
technique for the flow on the center manifold in the presence of non-isolated fixed points.
\end{abstract}

\maketitle

\section{Introduction}  The onset of synchronization or, more generally, emergence of collective patterns in multi-agent systems is a well-documented phenomenon spanning many applied and theoretical disciplines \cite{Kuramoto, KeefleStrogatzNatureCommunication, Pikovsky:Book}.
This note is devoted to the study of second-order systems of $n$ identical, self propelled particles that swarm in the plane, interacting through rotationally symmetric, quadratic  potentials $U(x)=\gamma|x|^2/2$. The dynamics is governed by
$\displaystyle
\ddot r_k=(\alpha -\beta |\dot r_k|^2)\dot{r}_k - \frac{1}{n}\sum_{m\neq k }\nabla U(r_k-r_m),$
where $\alpha, \beta, \gamma $ are positive constants. By choosing appropriate units of time and distance,  the equations can be normalized to:
\begin{equation}\label{eqnMainModel}
\begin{split}
\ddot{r}_k&=(1-|\dot r_k|^2)\dot{r}_k-\frac{ \mathrm{A}^2}{n}\sum_{m=1}^n (r_k-r_m),
\end{split}
\end{equation}
where  $\mathrm{A}^2 $ is the coupling strength. Let $\mathrm a= \frac{1}{A}.$ Equivalently,
\begin{equation}\label{eqnaMainModel}
\begin{split}
\ddot{r}_k&=(1-|\dot r_k|^2)\dot{r}_k-\frac{1}{\mathrm{a}^2} (r_k-R)
\end{split}
\end{equation}
where $R = \frac{1}{n}\sum_{m=1}^nr_m $ is the center of mass of the system. Here, $r_k=(x_k,y_k)\in \R^2,$  or interchangeably $r_k=x_k+iy_k \in \mathbb C$, represents the two-dimensional position vector of the $k$-th particle and $\dot r_k$ denotes the time derivative $\frac{dr_k}{dt}$. We refer the reader to \cite{Mach:2006,Szwaykowska:2016,Ebeling:2001} for various applications.

Mathematical models of multi-agents whose dynamics is very simple if acting in isolation, but develop complex patterns or structures when coupled, go as far back as 1951  \cite{Turing:1952}. 
These models highlight a paradoxical aspect in the evolution of multi-agent dynamical systems: otherwise ``dead'' cells (i.e. convergent to a stable fixed point, \cite{Smale:1976}) turn alive (pulsating or oscillating) due to their interactions.

If  the agents in Equation (\ref{eqnMainModel}) are isolated from each other ($\mathrm{A}=0$), they maintain their initial directions of motions, along $\dot r_k(0)$ , while their speeds converge to one, leading to agents dispersing.
Coupling  fundamentally changes the dynamics:
agents remain within  bounded distance from the center of mass
( \cite{MedynetsPopovici:2021}), with  the center of mass itself having  bounded range or escaping to infinity.
Figure \ref{FigureRingState} illustrates oscillatory  patterns in swarms whose center of mass has bounded motion.

\begin{figure}[h!]
\centering
\includegraphics[width=.6\textwidth]{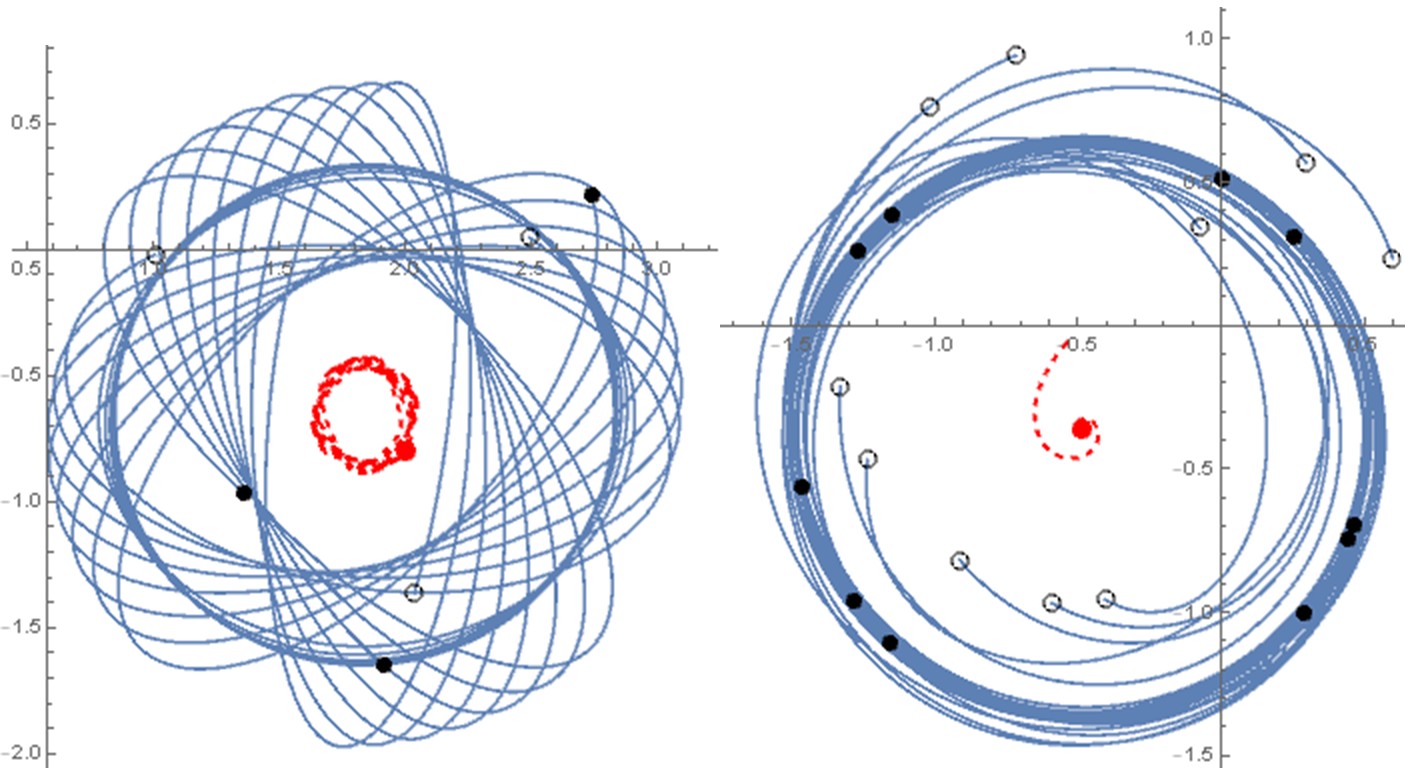}
\caption{\label{FigureRingState}
Limit patterns for (1), with $\mathrm A=1,  n=3$ (left) and $n=10$ (right).
Left: All agents and the center of mass exhibit sustained oscillations.
Agent positions at $t = 0$ (hollow dots) and at $t = 50$ (solid dots) are shown, with the center of mass marked by a red dot.
Right: The center of mass $R(t)$ (dashed) becomes nearly stationary. Agents stabilize into uniform circular motion
centered at $R(t)$ with relative phases ${\theta_k}$ satisfying $\sum_{k=1}^{n} e^{i\theta_k} = 0$. }
\end{figure}

Numerical simulations have captured two seemingly stable limit patterns for (\ref{eqnMainModel}): circular configurations in which the agents eventually rotate around a stationary center of mass,  $\mathrm r_k(t) =\frac{1}{\mathrm A} e^{ \pm i \mathrm A t} e^{ i\theta_{0,k}}+\mathrm R_0$    (Figure \ref{FigureRingState}, Right),  and rectilinear  motion with unit speed, $r_k(t)=v_0 t+d$, 
the latter proven stable in \cite{Popovici}.

\begin{definition}
A solution 
for (\ref{eqnMainModel}) is called a rotating state if there exist angles $\theta_{0,k}$  with $\displaystyle \sum_{k=1}^n e^{i\theta_{0,k}}=0$ and a constant $ R_0 \in \mathbb C$ such that
$ \mathbb r_k (t) = \frac{1}{\mathrm A} e^{i\mathrm A t} e^{i\theta_{0,k}} +R_0$ for all $t$ and all $k.$
We call a rotating state {\it degenerate}, possible only when $n$ is even, if there exist  $\theta_*\in [0,2\pi)$ and $\phi_j \in \{ 0,\pi\}$ such that $\theta_j = \theta_* + \phi_j$; that is, the particles form two equal-sized diametrically opposite groups rotating about $R_0$. 

A solution
$ \mathbb r_k (t) = \frac{1}{\mathrm A}e^{-it\mathrm A }e^{i\theta_{0,k}} + R_0 $ for all $t$ and all $k$ is called a clockwise-rotating state.
\end{definition}

Observe that the set of solutions of (\ref{eqnMainModel}) is
 translation invariant, rotation invariant, and
  invariant with respect to the reflections about the coordinate axes, that is, with respect to the transformations $y\mapsto -y$ and $x\mapsto -x$.
  Thus, if $(\mathrm r_1, \dots \mathrm r_n)$  is a solution of the system, then $(e^{i\theta}\mathrm r_1, \dots e^{i\theta}\mathrm r_n) $ and $(\mathrm r_1+c, \dots \mathrm r_n+c)$ are also solutions for any $\theta\in\mathbb R$ and any $c\in \mathbb C $.
In this paper we only discuss the case of counterclockwise motion;
clockwise motion follows from the system's symmetry with respect to reflection over the $x$-axis.

Some authors, but not all, use the term mill or ring for a rotating state; most authors define a mill as a rotating state whose angles $\theta_{0,k}$ are uniformly distributed in $[0, 2\pi]$, a.k.a having rotational symmetry.

 We emphasize that we make no symmetry or uniform-distribution assumptions. A
 substantial part of the analysis addresses configurations where agents split into two clusters modeled by systems with non-isolated fixed points. A contribution of our paper is a novel approximation technique for center manifolds  in this setting.

The paper gives a  rigorous proof of the stability of rotating states of (\ref{eqnMainModel}), and of the rates of convergence: exponential for $n$ odd, and exponential or of order $\frac{1}{\sqrt t}$ for $n$ even.

\begin{theorem}\label{TheoremMainResultIntro} Every rotating state solution of Equation (\ref{eqnMainModel}) is stable:
for any $\epsilon >0$ there exists $\delta >0$ such that given any initial center of mass location $\mathrm R_0$ and polar angles $\theta_{0,k}$ with
 $\sum _{k=1}^n e^{i\theta_{0,k}}=0$, if a solution $\{\mathrm r_k(t)\}$ of  (\ref{eqnMainModel}) satisfies
 $|\mathrm r_k (0)-\frac{1}{\mathrm A} e^{i\theta_{0,k}}|< \delta $ and  $|\dot{\mathrm r}_k (0)-i e^{i\theta_{0,k}}|< \delta $,
 then $$ |\mathrm r_k(t)-(\mathrm R_0+\frac{1}{\mathrm A}e^{i\theta_{0,k}}e^{i\mathrm A t})|<\epsilon\mbox{ and }  | \dot r_k(t)-ie^{i\theta_{0,k}}e^{i\mathrm A t}|<\epsilon\mbox{ for all }t\geq 0.$$
Furthermore, every solution that starts near a rotating state converges to a nearby rotating state:
 there exist $\mathrm R_{\infty}$ and  $\theta_{\infty , k}$ with $\sum _{k=1}^n e^{i\theta_{\infty,k}}=0$ such that
 $$\lim _{t \to \infty} \left [  \mathrm r_k(t)-(\mathrm R_{\infty} + \frac{1}{\mathrm A} e^{i\theta_{\infty,k}}e^{i\mathrm A t}) ) \right ] =0 \mbox{ and }
\lim _{t \to \infty} \left [ \dot{\mathrm r}_k(t)-ie^{i\theta_{\infty,k}}e^{i \mathrm A t} \right ] =0. $$
\end{theorem}

System (\ref{eqnMainModel})  belongs to the class of second-order gradient systems $\ddot r = f(\dot r) - \nabla U(r)$ (\cite[Ch. 7]{Haraux:book} and \cite{Haraux:1998}).  In those works it is shown that under some conditions on $f$ and the potential function $U$, most notably $f(\dot r)\cdot \dot r<0$ if $\dot r\neq 0$, every bounded solution converges to a configuration that solves $\nabla U(r)=0$.  Since the term $(1-|\dot r_k|^2)\dot{r}_k\cdot \dot{r}_k$ in (\ref{eqnMainModel}) is not negative-definite,  these results do not apply.

An in-depth discussion  of the existence of a global attractor and the synchronization
of  dissipative systems with {\it strong} coupling can be found in \cite{Hale:1997}. In Section 3 of \cite{Hale:1997}, it is proven that  1-dimensional Duffing oscillators with
sufficiently strong, nearest-neighbor, diffusive coupling synchronize.

If the motion in (\ref{eqnMainModel}) is constrained to a line, switching to velocity-acceleration coordinates, yields
 a system of coupled Van der Pol equations.  Such systems have been shown to admit nontrivial periodic cycles \cite{Rybicki:2003}.

If the agents in (\ref{eqnMainModel})  maintain a stationary  center of mass (referred to as a {\it decoupled system}), their dynamics reduces to $\ddot r_k = (1-|\dot r_k|^2)\dot r_k-\mathrm A^2 r_k$, a generalized Li\'enard system.  If the motion is further restricted to the $x$-axis (or to any other line through the origin), the equation $\ddot x- \dot x+(\dot x )^3 +\mathrm A^2 x =0$  admits a unique attracting cycle (\cite{Lienard:1928}, \cite[Theorem 4.6]{Verhulst:Book}). However, planar configurations in which agents  pulsate along individual lines through  a fixed center are unstable;
perturbations of velocity off the initial line of motion destabilize the configuration.

Oscillator models are often compared  to the Kuramoto's. Unlike Kuramoto oscillators, the  spatial angles $\tan ^{-1}(y_k/x_k)$ in (\ref{eqnMainModel}) are coupled not only through mutual interaction, but  also via spatial variables $r_k$. This places the model  (\ref{eqnMainModel}) in the class of swarmalators,  \cite{KeefleStrogatzNatureCommunication, Pikovsky:Book}. Studies of swarmalators with uniform phase distribution and specific fast-decaying coupling can be found in \cite{KeefleStrogatzNatureCommunication, PhysRevE.98.022203,  Pikovsky:Book}.

The limit behavior of swarms  whose particles are {\it equally distributed} on a circle (i.e. ring mills) have been numerically analyzed extensively, using both individual-based and continuum models, see, for example, \cite{DorsognaChuangBertozziChayes:2006}, \cite{CHUANG200730}, \cite{AlbiBalagueCarrilloBrecht:2014} for self-propelled particles, and \cite{Kolokolnikov:2011} for the steady state pattern formation. There are very few results that relax the equally-distributed  assumption: \cite{CarrilloSmallKer} addresses the stability of flocks satisfying certain hyperbolicity assumptions,  but for System (\ref{eqnMainModel}) those assumptions are only satisfied for $n=2$ agents: more specifically, if $n\geq 3,$  the translating states and the rotating states of (\ref{eqnMainModel})  have more neutral directions than what Condition H3 of   \cite{CarrilloSmallKer} requires; \cite{HaHaKim, Lear, Popovici} investigate flocking (rather than rotating) states.

\subsection{Configurations with stationary center of mass.}

This section analyzes the motion of agents if $R(t)=R(0)$ for all $t$, for $\mathrm A=1$, as a step toward   understanding swarms whose asymptotic configurations have a stationary center of mass. By translation invariance of (\ref{eqnMainModel}), we set
$R(0)=0$.

For example, $R(t)=0$ for all $t$  can occur when the  $n$ agents can be partitioned into subsets with rotational symmetries. The agents $\{ k_1, k_2, \dots k_p \}$ have rotational symmetry if
$ \displaystyle \mathrm r_{k_m}(t)= e^{(m-1)\frac{2\pi}{p}i} \mathrm r_{k_1}(t)$   for $ m=1,\ldots, p, \; t\geq 0. $

 \begin{figure}[h!]
\centering
\includegraphics[width=.6\textwidth]{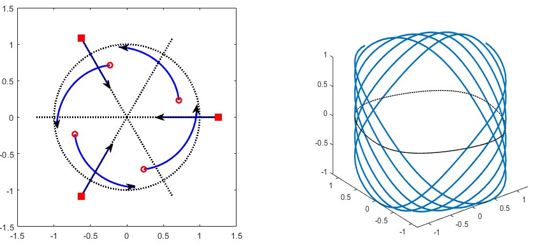}
\caption{\label{FigureSevenParticlesPartition} An illustration of a 7-agent, 28-dimensional system partitioned into a set of three 1-dimensional Van der Pol  oscillators (the square markers) and four rotating agents (the hollow markers), having  $R(t)=(0,0)$ for all $t$.  Left:  Agent trajectories from $t = 0$ to $t = 1.5$ are shown as solid curves; their corresponding limit cycles are indicated by dashed curves.
Right:  projection of the curve $(x_1(t), x_2(t), \dots \dot y_7(t))\subset \mathbb R^{28}$ onto
the 3-dimensional subspace $(x_1, x_4, \dot x_1).$  }
\end{figure}

When $R=0,$ agents' motions decouple, each satisfying:
\begin{equation}\label{EqnDecoupledSystem}\ddot{r}=(1-|\dot{r}|^2)\dot{r}-r\end{equation}

Consider a nonzero solution $r$ of (\ref{EqnDecoupledSystem}). If the vectors $r$ and $\dot{r}$  are parallel at some instant,  then they remain parallel and  can be written as $r(t)=c(t)\hat{r}$  and $\dot r(t)=\dot c(t)\hat{r}$ for some constant unit vector $\hat{r}$ and a scalar function $c$.  Then $c(t)$ satisfies  $\ddot c = \dot c-\dot c ^3-c$,
 a Li\'{e}nard equation (\cite{Lienard:1928}, \cite{Lins:1977}), with  a unique, stable limit cycle.
Next, we examine trajectories of (\ref{EqnDecoupledSystem}) where position and velocity vectors are initially non-parallel.

\begin{theorem}\label{TheoremDecoupledCase}
 Let $r$ be a solution of (\ref{EqnDecoupledSystem}) and let  $\phi$ denote  the angle from $r$ and $\dot r$, measured counterclockwise.  If $0<\phi<\pi$  at some instant, then  $r(t)$ converges to a limit cycle ${\mathrm r}_+(t) = e^{i\theta_*}e^{it}$ for some phase angle $\theta_*$. If  $-\pi<\phi <0$, then $r(t)$ converges to a cycle  $\mathrm{r}_-(t) = e^{i\theta_*}e^{-it}.$  The limit cycles ${\mathrm r}_+$ and ${\mathrm r}_-$ are stable.
\end{theorem}
\begin{proof} Setting $r=(x,y)$, $u = \dot x$, and $v = \dot y$, we can rewrite (\ref{EqnDecoupledSystem}) as
\begin{equation}\label{eq:4dsys}
\begin{array}{ll}
\dot{u}&=(1-u^2-v^2)u-x,\\
\dot{v}&=(1-u^2-v^2)v-y.
\end{array}
\end{equation}
Along trajectories of (\ref{eq:4dsys}), the derivative of $vx-uy$  is equal to $\left(1-u^2-v^2\right) (v x-u y)$. Thus, if $vx=uy$ at any time,
 then the vectors $r$ and $\dot{r}$ remain parallel (or zero),  and the system reduces to the one-dimensional case. We therefore focus on the region 
$\Omega=\{(u,v,x,y)\in\mathbb{R}^4:vx\neq uy\}.$
Note that the subsets $\Omega_+$ and $\Omega_-$ of $\Omega$ satisfying $vx-uy >0$ and $vx-uy <0$, respectively, are invariant.

Define the function $L:\Omega\rightarrow\mathbb{R}$ as follows:
\begin{equation*}L=x^2+y^2+u^2+v^2-\log((vx-uy)^2)-2\end{equation*}
From $  (vx-uy)^2 = (x^2+y^2)(u^2+v^2)-(xu+yv)^2) \leq (x^2+y^2)(u^2+v^2) $ we get
\begin{equation*}
L\geq  (x^2+y^2-\log(x^2+y^2)-1)+(u^2+v^2-\log(u^2+v^2)-1).
\end{equation*}
Since $t-\log(t)\geq 1$ for all $t>0$, the minimum of $L$ is 0,  attained on $\Omega_{\min}=\{(u,v,x,y)\in\Omega:(u^2+v^2=1),(x^2+y^2=1),(xu+yv=0)\}$.
Moreover, if $(x^2+y^2)\rightarrow\infty$ or $(u^2+v^2)\rightarrow\infty$, we have $L\rightarrow\infty$, so $L$ is radially unbounded.

Differentiating $L$ along trajectories of (\ref{eq:4dsys}) yields
$\dot{L} = -2(1-u^2-v^2)^2\leq 0,$
 thus $L$ is decreasing.
By LaSalle's Invariance Principle, every trajectory approaches the largest invariant set inside $\Omega_{\min}$. The trajectories fully contained in $\Omega_{\min}$ satisfy $\ddot x+x=0$ and correspond to solutions $c e^{it}$ and $ce^{-it}$, with $c\in \mathbb C$, $|c|=1$.

Thus, if $0<\phi<\pi$, then the solution $\{r,\dot r\}$ lies in $\Omega_+$ and converges to the limit cycle   $\mathrm{r}_+$. Similarly, solutions for which $-\pi<\phi<0$  approach a limit cycle $\mathrm{r}_-$.
\end{proof}

\subsection{A brief outline of the paper.}

Sections \ref{SectionRingStateChangeOfCoordinates}, \ref{SectionNonDegenerateRingState},  and \ref{SectionDegenerateRingState} are dedicated to the proof of Theorem \ref{TheoremMainResultIntro}.

In Section \ref{SectionRingStateChangeOfCoordinates} we introduce  new coordinate systems based on  rotating frames of reference. We define $\{X_k,Y_k\}_{k=1, \dots n} $ by setting $\mathrm r_k(t)=\frac{1}{\mathrm A}e^{i \mathrm A t}(X_k(\mathrm At)+i Y_k(\mathrm At))$, and establish that the set of equilibrium points in the $\{X_k,Y_k\}$ coordinates corresponds to the set of rotating states {\it centered at the origin} in the original coordinates.

 In Section \ref{SectionNonDegenerateRingState} we present the proof of Theorem \ref{TheoremMainResultIntro} for the case of a non-degenerate rotating state.
We provide a  semi-explicit description of the center manifold as  the  product of a $(n-2)$-dimensional sub-manifold of fixed points (given implicitly) and a plane parametrized by elliptical trajectories.

In Section \ref{SectionNotTaylor} we highlight a key limitation of Taylor approximations of the center manifold in systems with non-isolated fixed points: truncating the vector field can destroy the degeneracies that cause the nullclines to intersect, thereby  eliminating  equilibrium points. We present a new approximation scheme  that  turns having a large set of equilibrium points into a computational advantage: the fixed points are used to anchor the approximation of the center manifold.

 Section \ref{SectionDegenerateRingState} is devoted to the perturbations of degenerate rotating states.
 Using the technique developed in Section \ref{SectionNotTaylor}, we determine the map of the center manifold, now of dimension $n+1$, and prove that the flow on it is stable, with all the limiting dynamics corresponding to rotating states. We show that the convergence to a degenerate configuration has a rate of order $\frac{1}{\sqrt t} $, in contrast to the exponential convergence for non-degenerate limit cycles.

The Appendix includes Mathematica commands for  calculating the determinants, characteristic polynomials, and inverse matrices from Lemma \ref{LemmaNonDEigenspace} and (\ref{NewBlocks}).


\section{Change of Coordinates about a Given Rotating State}
\label{SectionRingStateChangeOfCoordinates}

In this section, we introduce  two-step coordinate changes in which the original rotating states of (\ref{eqnMainModel}) {\it centered at the origin}  become fixed points.

 First we define the {\it rotating frame}  $\{X_k,Y_k\}$ as
 \beq
 \label{defineXY}
 \mathrm r_k(t)=\frac{1}{\mathrm A}e^{i\mathrm A t}(X_k(\mathrm A t)+i Y_k(\mathrm A t)).
\eeq
Next,
we fix a rotating state $\left(  \frac{1}{\mathrm A}e^{i \mathrm A t}e^{i\theta_{0 k}} \right )_{k=1}^n$ and
we apply rotations that {\it depend on the parameters} $\theta_{0 k}$
and translations, so that the state $\left(  \frac{1}{\mathrm A}e^{i \mathrm A t}e^{i\theta_{0 k}} \right )_{k=1}^n$
is mapped  to the origin.

Recall the notation $ \mathrm a= \frac{1}{\mathrm A}$. From (\ref{defineXY}) and (\ref{eqnaMainModel}) we get that $X_k+iY_k$ satisfy
 \footnote{The derivation of (\ref{eqnXYEquilibriumSystem}) follows step by step that in (\ref{eqnChangeCoordinates_ab}) as if all $\theta_{0k}$ were all zero.}
\begin{equation}\label{eqnXYEquilibriumSystem}
\begin{array}{l}
X_k(t)+iY_k(t) = \mathrm A e^{-i t } r_k(\mathrm a t ),
\\
\ddot X_k  = 2 \dot Y_k  + \mathrm a\left (1-(\dot X_k-Y_k)^2-(X_k +\dot Y_k)^2 \right )(\dot X_k - Y_k) +  X_{mean},
\\
\ddot Y_k  =    -2 \dot X_k + \mathrm a \left (1-(\dot X_k-Y_k)^2-(X_k +\dot Y_k)^2 \right )(X_k+ \dot Y_k) + Y_{mean},
\end{array}
\end{equation}
where $ X_{mean} = \frac{1}{n}(X_1+\cdots + X_n)$ and $ Y_{mean} = \frac{1}{n}(Y_1+\cdots + Y_n)$.
Our next result describes the set of equilibrium points for (\ref{eqnXYEquilibriumSystem}).
 \begin{lemma}\label{LemmaEquilibriumSolutions} A point $(X_{0k},Y_{0k},\dot X_{0k}, \dot Y_{0k})$, $k=1,\ldots, n$, is an equilibrium point of (\ref{eqnXYEquilibriumSystem})  if and only if (i) $\dot X_{0k} = 0$, $\dot Y_{0k}=0$ for all $k$,  (ii) $ \frac{1}{n} \sum \limits _{k=1}^n X_{0k} =\frac{1}{n} \sum \limits_{k=1}^n Y_{0k} = 0$ and (iii) for all $k$, either $X_{0k}=Y_{0k} = 0$ or $X_{0k} = \cos(\delta_k)$, $Y_{0k} = \sin(\delta_k)$ for some constants $\delta_k$.
 \end{lemma}
 Note that the fixed points of (\ref{eqnXYEquilibriumSystem}), seen in the original coordinates, correspond to having a juxtaposition between $p$ particles placed at the origin and $N-p$ particles rotating about the origin.

\begin{proof} For convenience, we are dropping the zero subscripts.
In (\ref{eqnXYEquilibriumSystem}), setting $\dot X_k =\ddot X_k= 0$, $\dot Y_k=\ddot Y_k=0$, and solving for $ X_{mean} $ and $ Y_{mean} $, we obtain that for all $k$
\begin{equation}\label{EqnXYAux}
 X_{mean}  = -\mathrm a(1-X_k^2-Y_k^2)Y_k   \mbox{ and }  Y_{mean}  = \mathrm a(1-X_k^2-Y_k^2)X_k .
\end{equation}

It follows that
 \begin{equation}\label{EqnXYCenterMass}(X_{mean})^2 + (Y_{mean})^2 = \mathrm a^2  (1-(X_k^2+Y_k^2))^2(X_k^2+Y_k^2),\quad k=1,\ldots,n.
\end{equation}

Set $s = (X_{mean})^2+(Y_{mean})^2$.  We claim that $s=0$. Indeed, assume towards contradiction that $s>0$. Consider the function $f(t) = \mathrm a ^2 (1-t)^2t$ for $t\geq 0$. Then each value $s_k = X_k^2+Y_k^2$ is a solution of the equation $f(s_k) =s$.  Note that $s_k\neq 1$ for otherwise $s=f(s_k)=0$. Solving (\ref{EqnXYAux}) for $X_k$ and $Y_k$, we get that
\bdima
X_k = \frac{Y_{mean}}{\mathrm a (1-s_k)}\mbox{ and } Y_k = -\frac{X_{mean}}{\mathrm a(1-s_k)}.
\edima
It follows that
\bdima
X_{mean} = \frac{1}{n}\sum_k\frac{Y_{mean}}{\mathrm a (1-s_k)}\mbox{ and } Y_{mean} = -\frac{1}{n}\sum_k\frac{X_{mean}}{\mathrm a(1-s_k)}.
\edima
Multiplying both sides of the first equation by $X_{mean}$ and the second equation by $Y_{mean}$ and then adding them together, we obtain that $s=(X_{mean})^2+(Y_{mean})^2 = 0$, which is a contradiction.

Thus $s=0$. It follows from  Equation (\ref{EqnXYCenterMass}) that either $X_k^2+Y_k^2=0$ or $X_k^2+Y_k^2=1$. Denote by $I$ the set of indices $k$ such that $X_k^2+Y_k^2=1$. For $k\in I$, choose $\delta_k\in [0, 2\pi]$ such that $(X_k,Y_k) =  (\cos(\delta_k),\sin(\delta_k)). $ Note that $\sum_{k\in I} e^{i\delta_k}=n(X_{mean}+ iY_{mean})=0$. Conversely, we notice that any family $(X_k,Y_k)$ of constant functions such that $X_k=Y_k=0$ or $X_k=\cos(\delta_k)$, $Y_k = \sin(\delta_k)$ with $X_{mean} =  Y_{mean} = 0$ is an equilibrium solution for (\ref{eqnXYEquilibriumSystem}).  The result now follows.

\end{proof}

The rotating state $\left ( \frac{1}{\mathrm A}e^{i \mathrm A t}e^{i\theta_{0 k}}\right )_{k=1}^n$ corresponds to the fixed point
\newline $\left ( X_{0k}=\cos\theta_{0 k}, Y_{0k}=\sin \theta_{0 k}\right ) _{k=1}^n$ in the new frame.
Together, the rotating states about the origin of (\ref{eqnMainModel}) correspond to the fixed points of (\ref{eqnXYEquilibriumSystem}) that belong to the  set $ \displaystyle \mathcal F=\bigcap _{k=1}^n \{X_k^2+Y_k^2=1\}$. The $\epsilon$-$\delta$ formulation of Theorem \ref{TheoremMainResultIntro} when $R_0=0$ is equivalent the  stability of the fixed points of (\ref{eqnXYEquilibriumSystem})  from $\mathcal F.$

The second step of the coordinate change is dependent on the parameters of the rotating state.

Fix a  rotating state {\it $\mathrm{r}_k =\frac{1}{A} e^{iA t}e^{i\theta_{0 k}}$}.  Notice that $\sum_{k=1}^ne^{i\theta_{0 k}} = 0$. Shifting the phases by the same angle $\theta^*$ gives a new rotating state $\mathrm r_k = \frac{1}{A}e^{iAt}e^{i(\theta_{0 k}+\theta^*)}$ with the same stability characteristics.
We can always find $\theta^*$ such that $\displaystyle \sum_k \sin(\theta_{0 k}+\theta^*)\cos(\theta_{0 k}+\theta^*)=0$ since $\sum_k \sin(\theta_{0 k}+\theta^*)\cos(\theta_{0 k}+\theta^*)  =c_1\cos(2\theta^*)+ c_2\sin(2\theta^*) $ for
$c_1=\frac{1}{2}\sum_k \sin(2\theta_{0 k})$
 and $c_2 =\frac{1}{2}\sum_k \cos(2\theta_{0 k})$,
and for any real numbers $c_1$ and $c_2$, the equation $c_1\cos(2\theta^*)+c_2\sin(2\theta^*)=0$ always has a solution $\theta^*$.
 Thus,  working with phase-shifted rotating states if necessary,  we can additionally assume that the polar angles $\{\theta_{0 k}\}$ of the rotating state  satisfy:
\begin{equation}\label{EqOrthogonality}
  \sum_{k=1}^n\cos(\theta_{0 k})\sin(\theta_{0 k}) = 0.
\end{equation}
%
%

\begin{definition}
 Consider the following {\it change of coordinates}
\begin{equation}
\label{DefineChangeCoordinates_ab}
\mathrm r_k(t)=\frac{1}{A}e^{iA t}e^{i \theta_{0 k}}\big [(a_k(At)+1) + i b_k(At) \big ],\mbox{ where }a_k(t), b_k(t)\in \mathbb R. \end{equation}
Equivalently, the functions $a_k, b_k $ can be defined as
\beq
\label{affine}
\begin{array}{l}
a_k +i\;  b_k=  e^{-i\theta_{0 k }}( X_k+i Y_k) -(1+0i)\; \mbox{ or as}\\
\mathrm r_k(\mathrm a t)=\mathrm a e^{i t}e^{i \theta_{0 k}}\big [(a_k(t)+1) + i b_k(t) \big ].
\end{array}
\eeq
Let  $u_k = \dot a_k$  and $w_k = \dot b_k.$ Denote by $\bar a$ the column vector $(a_1,\ldots, a_n)^T$. The vectors $\bar b$, $\bar u$, $\bar w$ are defined analogously.
\end{definition}

\begin{proposition}
The conversion of the swarm equations (\ref{eqnMainModel}) to  the $(\bar a, \bar b, \bar u, \bar w)$ coordinate system is:
 \begin{equation}\label{eqnNewSystem_abpq_coord}
 \begin{split}
 \dot a_k &= u_k \\
 \dot b_k &= w_k \\
\dot u_k &= \begin{multlined}[t]
 2w_k
  +\frac{1}{n}\sum_{m=1}^n \cos(\theta_{0 m}-\theta_{0 k})a_m-\frac{1}{n}\sum_{m=1}^n \sin(\theta_{0 m}-\theta_{0 k})b_m\\
  -\mathrm a \big[ (u_k-b_k)^2+2(a_k+w_k)+(a_k+w_k)^2 \big](u_k-b_k)
\end{multlined} \\
\dot w_k &= \begin{multlined}[t]
-2u_k+\frac{1}{n}\sum_{m=1}^n \sin(\theta_{0 m}-\theta_{0 k})a_m+\frac{1}{n}\sum_{m=1}^n \cos(\theta_{0 m}-\theta_{0 k})b_m\\
-2\mathrm  a\;  a_k-2 \mathrm a \; w_k- a \big[(u_k-b_k)^2+(a_k+w_k)^2 \big] \\
 - \mathrm a \big[(u_k-b_k)^2+2(a_k+w_k)+ (a_k+w_k)^2 \big](a_k+w_k).
\end{multlined}
 \end{split}
 \end{equation}
\end{proposition}

\begin{proof}
Substituting time $(\mathrm a t)$ and multiplying by $\mathrm a $ in  (\ref{eqnMainModel}) gives
 \beq
 \label{parameterabuw}
  \mathrm a \; \ddot r_k(\mathrm  a t)+\frac{1}{\mathrm a}r_k(\mathrm  a t)=
\mathrm a (1-|\dot r_k(\mathrm a t)|^2)\; \dot{r}_k(\mathrm a t)+  \frac{1}{n}\sum_{m=1}^n \frac{1}{\mathrm  a}r_m(\mathrm a t ). \eeq
We express the left and right-hand-sides of this equation in terms  of $a_k(t)$ and $ b_k(t)$, starting from
$\displaystyle r_k(\mathrm  a t)=\mathrm  a\;  e^{it} e^{i\theta_{0k}} \left[(a_k+1)+ib_k\right ]$:
\begin{equation}\label{eqnChangeCoordinates_ab}
\begin{array}{l}
\dot { r}_k(\mathrm a t) = e^{i t}e^{i \theta_{0 k}} \big[ i\left ((a_k+1) +i b_k \right )  +(\dot{a}_k+i\dot b_k) \big] \\
\mathrm a \; \ddot { r}_k(\mathrm a t) = e^{i t}e^{i \theta_{0 k}}\big[-\left ((a_k+1)+ib_k \right )+ 2( i \; \dot a_k -\dot b_k)+
(\ddot{a}_k+i\ddot b_k)  \big]\\
\mathrm a \; \ddot r_k(\mathrm  a t)+\frac{1}{\mathrm a}r_k(\mathrm  a t)=e^{i t}e^{i \theta_{0 k}}\big[ 2( i \; \dot a_k -\dot b_k)+(\ddot{a}_k+i\ddot b_k)  \big].
\end{array}
\end{equation}
 Note that $|\dot {r}_k(\mathrm a t)|= |(\dot a_k-b_k)+i((1+a_k+ \dot b_k)|$ and
 $$1-|\dot {r}_k(\mathrm a t)|^2 = 1- (\dot{a}_k - b_k)^2 -(a_k+1+\dot b_k)^2  = -(\dot a_k-b_k)^2-2(a_k+\dot b_k)-(a_k+\dot b_k)^2.$$
 Substituting the latter and equations (\ref{eqnChangeCoordinates_ab}) into (\ref{parameterabuw}), and after dividing both sides by $ \mathrm e^{i t}e^{i \theta_{0 k}}$, we get
  \begin{multline*}
 2( i \; \dot a_k -\dot b_k)+(\ddot{a}_k+i\ddot b_k)= \frac{1}{n}\sum_{m=1}^n e^{i(\theta_{0 m}-\theta_{0 k})} (a_m+1+ib_m)\\
 +\mathrm a \; \left(1-(\dot a_k-b_k)^2-((1+a_k+ \dot b_k)^2\right ) \left ((\dot{a}_k - b_k) +i((a_k+1)+\dot b_k) \right ).
 \end{multline*}
 Note that $\sum_{m=1}^n e^{i(\theta_{0 m}-\theta_{0 k})} = e^{-i\theta_{0 k}}\left ( \sum_{m=1}^n e^{i\theta_{0 m}} \right )=0$. We get
 \begin{multline*}
 \ddot{a}_k+i\ddot b_k=  -2\dot b_k+ 2 i\;  \dot a_k +  \frac{1}{n}\sum_{m=1}^n e^{i(\theta_{0 m}-\theta_{0 k})} (a_m+ib_m)\\
 -\mathrm a \; i \left [ (\dot a_k-b_k)^2+2(a_k+\dot b_k)+(a_k+\dot b_k)^2 \right ]\\
 -\mathrm a \left [ (\dot a_k-b_k)^2+2(a_k+\dot b_k)+(a_k+\dot b_k)^2 \right ]\left [(\dot{a}_k - b_k) +i(a_k+\dot b_k) \right ].
 \end{multline*}
  Separating the real and imaginary parts and using $u_k = \dot a_k, \; w_k = \dot b_k$ completes the proof.
 \end{proof}

The rotating frame coordinates $\{X_k,Y_k\}$ and $\{a_k,b_k\}$ are related  via an affine transformation (\ref{affine}), therefore
there is a one-to-one correspondence between their equilibrium points.
The starting rotating state solution $\mathrm r_k = \frac{1}{A}e^{iAt}e^{i\theta_{0 k}}$, associated with the fixed point $X_k+iY_k=e^{i \theta_0k}$,
corresponds to the fixed point $a_k=0$, $b_k=0$, $u_k = 0$, $w_k =0$, $k=1,\ldots, n$,  in the new coordinates.
Using  Lemma \ref{LemmaEquilibriumSolutions} we get the description of all equilibrium points for the  $(\bar a, \bar b, \bar u, \bar w)$ system (\ref{eqnNewSystem_abpq_coord}):

\begin{lemma}\label{LemmaEquilibriumSolutions_abuw} Let $\Delta = \{(\delta_1,\ldots,\delta_n) \in \mathbb R^n : \sum \limits_{k=1}^n e^{i(\theta_{0k}+\delta_k)}=0\}.$ Let $\mathcal R \subset \mathbb R^{4n}$ be
 $$\mathcal R = \left \{
a_k=cos( \delta_k)-1,b_k = \sin( \delta_k),u_k=0,w_k = 0 :  \{\delta_k\}\in \Delta \right \} .$$
Then $\mathcal R$ consists of  all equilibrium solutions of (\ref{eqnNewSystem_abpq_coord}) near the origin.
\end{lemma}

\vskip2mm

\begin{definition} For a rotating state $\mathrm r_k = \frac{1}{A}e^{iAt}e^{i\theta_{0 k}}$ satisfying (\ref{EqOrthogonality}), denote by $\bar c$ and  $\bar s$  the mutually orthogonal column vectors $$\bar c = (\cos(\theta_{0 1}),\ldots,\cos(\theta_{0 n}))^T\mbox{ and }\bar s = (\sin(\theta_{0 1}),\ldots,\sin(\theta_{0 n}))^T.$$
\end{definition}

Note that equation (\ref{EqOrthogonality})  ensures that  $\bar c \perp \bar s$. Unless one of them is the zero vector, $\bar c$ and $\bar s$ span a two-dimensional linear space.
The equality $\bar c=0$  implies that $\sin \theta_{0k}=\pm 1, $ and given that  $\sum_{k=1}^n e^{i\theta_{0k}}=0,$ that can only be satisfied if $n$ is even, with half of the angles $\theta_{0k}$ equal to $\pi/2$ and the other half equal to $-\pi/2.$ The equation $\bar s=0$  can only be satisfied if $n$ is even, with half of the angles $\theta_{0k}$ equal to $0$ and the other half equal to $\pi.$ Either way, the span of  $\bar c$ and $\bar s$ is two dimensional except  for the degenerate rotating states.

{\bf Jacobian}  We   calculate the Jacobian $J$ of (\ref{eqnNewSystem_abpq_coord}) at the origin, and the nonlinear part of (\ref{eqnNewSystem_abpq_coord}) next. The rows and the columns of the Jacobian are indexed by the variables $\{\{a_k\}, \{b_k\},\{u_k\},\{w_k\}\}$ and their derivatives, respectively.  Then

\renewcommand{\kbldelim}{(}
\renewcommand{\kbrdelim}{)}
\begin{equation}\label{eqJacobianGeneral}
  J = \kbordermatrix{
    & \partial a & \partial b & \partial u & \partial w  \\
    \partial \dot a & \OO_n & \OO_n & \I_n & \OO_n  \\
    \partial \dot b & \OO_n & \OO_n & \OO_n & \I_n  \\
    \partial \dot u & C & -S & \OO_n & 2\I_n  \\
    \partial \dot w & S-2\mathrm a \I & C & -2\I_n & -2\mathrm a\I_n
  }
\end{equation}
where $\OO_n$ is the $n\times n$ zero matrix, $\I_n$ is the $n\times n$ identity matrix,  $$S = \left\{\frac{1}{n}\sin(\theta_{0 m}-\theta_{0 k})\right\}_{k,m=1}^n\mbox{ and }C = \left\{\frac{1}{n}\cos(\theta_{0 m}-\theta_{0 k})\right\}_{k,m=1}^n.$$
The non-linear parts  of $\dot u_k$ and $\dot w_k$ in (\ref{eqnNewSystem_abpq_coord}) can be represented as $(\mathrm a \; \mathcal U_k)$ and $(\mathrm a\; \mathcal W_k)$ where
\begin{equation}
\label{eqn_abuw_sys_nonlinearities0}
\begin{array}{l}
\mathcal U_k =  -((u_k-b_k)^2+2(a_k+w_k)+(a_k+w_k)^2)(u_k-b_k) \\
\mathcal{W}_k=  -(u_k-b_k)^2-(a_k+w_k)^2-((u_k-b_k)^2+2(a_k+w_k)+(a_k+w_k)^2)(a_k+w_k).
\end{array}
\end{equation}
Denote $\mathcal U=(\mathcal U_1, \mathcal U_2, \dots \mathcal U_n)^T$ and $\mathcal W=(\mathcal W_1, \mathcal W_2, \dots \mathcal W_n)^T.$ The system (\ref{eqnNewSystem_abpq_coord}) can be expressed as:

\beq \label{compactabuw}
col (\dot{\bar a}, \dot{\bar b}, \dot{\bar u}, \dot{\bar w})= J col(\bar a, \bar b, \bar u, \bar w) + col(\oo_n, \oo_n, \mathrm a  \; \mathcal U, \mathrm a \; \mathcal W).
\eeq


\section{Stability of Non-Degenerate Rotating States}\label{SectionNonDegenerateRingState}

In this section, we assume that the rotating state $\mathrm r_k = e^{it}e^{i\theta_{0 k}}$ is non-degenerate and satisfies (\ref{EqOrthogonality}).
We block diagonalize the  Jacobian (\ref{eqJacobianGeneral}) according to the stable and  center linear subspaces (there are no unstable directions)
and establish the stability of the dynamics on the center manifold.

 \begin{lemma}
 \label{LemmaSCaction}
 Let $S$ and $C$ be the $n\times n$ blocks used to define the Jacobian  $J$, in (\ref{eqJacobianGeneral}).  For $x \in \R^n$
 \bdima
 \begin{split}
 Sx= \frac{1}{n} ( x^T \bar s )\;  \bar c- \frac{1}{n}( x^T \bar c )\;  \bar s, \\
 Cx= \frac{1}{n}(  x^T \bar c )\;  \bar c+ \frac{1}{n} ( x^T \bar s )\;  \bar s.
 \end{split}
 \edima
Moreover,
 denoting  $f= \frac{1}{n}\sum_{k=1}^n \cos^2(\theta_{0k})$ and $d= \frac{1}{n}\sum_{k=1}^n \sin^2(\theta_{0k})$, we get 
\begin{equation}
\label{EqCSatcs}
C\bar c= f\bar c,\; \;   S\bar c = -f \bar s, \qquad
C\bar s = d\bar s, \; \;  S\bar s=d\bar c,
\end{equation}
and $\mathrm{ker}(S) = \mathrm{ker}(C) = \{\bar s,\bar c\}^\perp$.
 \end{lemma}

\begin{proof}
Let $x\in \R^n. $ The $k$-th entry of the vector $ n( S x) $ is
 \begin{equation*}\begin{split}\sum_{m=1}^n \sin(\theta_{0 m}-\theta_{0 k})x_m & = \cos(\theta_{0 k})\sum_{m=1}^n\sin(\theta_{ 0 m})x_m -\sin(\theta_{0 k})\sum_{m=1}^n\cos(\theta_{ 0 m})x_m \\ & = \cos(\theta_{ 0 k}) (\bar s^T x)-\sin(\theta_{0 k}) (\bar c^T x),
 \end{split}\end{equation*}
therefore
$ Sx= \frac{1}{n}( x^T \bar s )\;  \bar c- \frac{1}{n} ( x^T \bar c )\;  \bar s.$ A similar argument  applies to $Cx.$

Substitute $x=\bar c$ (or $\bar s$) and  $ \bar c ^T\bar s =0$
(per (\ref{EqOrthogonality}))  to obtain
(\ref{EqCSatcs}).

Finally, we use the linear independence of $\bar c$ and $\bar s$ to describe $\mathrm{ker}(S)$  and  $ \mathrm{ker}(C) $.  We get $x\in \mathrm{ker}(S)$  (or $ \mathrm{ker}(C), $ respectively ) if and only if both coefficients
$\frac{1}{n} x^T \bar s $ and $ \frac{1}{n} x^T \bar c $  equal zero. This proves
$\mathrm{ker}(S) = \mathrm{ker}(C) = \{\bar s,\bar c\}^\perp$.

\end{proof}

 \begin{lemma}
 \label{LemmaNonDEigenspace}
 The spectrum of the Jacobian matrix $J$ (introduced in (\ref{eqJacobianGeneral})) consists of $(n-2)$ zero eigenvalues,  $\pm i$, and $3n$ eigenvalues from the half-plane $\mbox{Re} \lambda <0.$
 \end{lemma}

 \begin{proof}
 Let $\V$ be an $n\times (n-2)$ dimensional matrix whose column vectors form an orthonormal basis for $\mathrm{ker}(C) = \mathrm{ker}(S)$.  Consider the linear subspaces $L_1$ and $L_2$ of $\mathbb R^{4n}$ spanned by the columns of the matrices $B_1$ and $B_2$, respectively, where

  \begin{equation*}
  B_1 = \left[
          \begin{array}{cccc}
            \V & \mathbb O & \mathbb O & \mathbb O\\
            \OO_{n,n-2} & \V & \mathbb O & \mathbb O \\
            \OO_{n,n-2} & \mathbb O & \V & \mathbb O \\
            \OO_{n,n-2} & \mathbb O & \mathbb O & \V \\
          \end{array}
        \right],
        \mbox{ and }
         B_2 = \left[
          \begin{array}{cccccccc}
           \bar c & \bar s & 0 & 0 & 0 & 0 & 0 & 0 \\
            0 & 0 & \bar c & \bar s & 0 & 0 & 0 & 0 \\
            0 & 0 & 0 & 0 & \bar c & \bar s & 0 & 0 \\
            0 & 0 & 0 & 0 & 0 & 0 & \bar c & \bar s \\
          \end{array}
        \right],
  \end{equation*}
where $\mathbb O=\OO_{n,n-2}$ is the zero matrix.
Equivalently, $B_1=I_4\otimes\V, B_2= I_4\otimes [\bar c \; \  \bar s]. $
Note that $\mathbb R^{4n} = L_1\oplus L_2$, $L_1\perp L_2$, $\mathrm{dim}(L_1) = 4n-8$, and $\mathrm{dim}(L_2)=8$.
We show next that in the basis given by $[B_1 \; B_2]$ the Jacobian $J$ has a block-diagonal form, with blocks $J_1$ and $J_2$ of dimensions $(4n-8)\times(4n-8)$ and $8\times 8$ respectively, where $J_1$ is given in (\ref{EqJacobianJ1}) and $J_2$ is given in (\ref{EqJacobianRestrictionB2}).
Equivalently, we need to show that $ J [B_1 \;  B_2]= [B_1\;  B_2] \mbox{diag}\{J_1, J_2\},$ i.e. that
$J\; B_1=B_1\; J_1 \; $  and $ \; J\; B_2 = B_2 \; J_2.$

From (\ref{eqJacobianGeneral}), using $C \V= \mathbb O=\OO_{n,n-2}$  we get:
\begin{equation}\label{EqJacobianRestrictionB1}
JB_1 = \left[
          \begin{array}{cccc}
            \OO_{n,n-2} & \mathbb O & \V & \mathbb O \\
            \OO_{n,n-2} & \mathbb O & \mathbb O & \V \\
            \OO_{n,n-2} & \mathbb O & \mathbb O & 2\V \\
            -2a\V & \mathbb O & -2\V & -2 a \V \\
          \end{array}
        \right]=  \left[
\begin{array}{cccc}
 0 & 0 & 1 & 0 \\
 0 & 0 & 0 & 1 \\
 0 & 0 & 0 & 2 \\
 -2 a & 0 & -2 & -2 a \\
\end{array}
\right]\otimes\V.
       \end{equation}
The action of $J$ on the subspace $L_1$ mimics that of a $4\times 4$ matrix.
Denote:
 \begin{equation}\label{EqJacobianJ1}
J_a=\left[
\begin{array}{cccc}
 0 & 0 & 1 & 0 \\ 0 & 0 & 0 & 1 \\ 0 & 0 & 0 & 2 \\ -2 a & 0 & -2 & -2 a \\
\end{array} \right], \
 J_1 = \left[
          \begin{array}{cccc}
            \OO_{n-2} & \OO_{n-2} & \I_{n-2} & \OO_{n-2} \\
            \OO_{n-2} & \OO_{n-2} & \OO_{n-2} & \I_{n-2} \\
            \OO_{n-2} & \OO_{n-2} & \OO_{n-2} & 2\I_{n-2} \\
            -2\mathrm a\I_{n-2} & \OO_{n-2} & -2\I_{n-2} & -2\mathrm a\I_{n-2} \\
          \end{array}
        \right].
     \end{equation}
Note that $J_1=J_a\otimes I_{n-2}$  and   $JB_1=  J_a \otimes\V$, per (\ref{EqJacobianRestrictionB1}).
By direct calculation, or using tensor notation
\footnote{
$ B_1J_1= (I_4\otimes \V)(J_a \otimes I_{n-2})= (I_4J_a)\otimes (\V I_{n-2})= J_a\otimes \V =J B_1$},
we  get that $B_1J_1$ and $JB_1 $ are equal.

The characteristic polynomial of $J_a$ is $\lambda  (\lambda ^3+2 a \lambda ^2+4 \lambda +4 a)$.
We conclude
$$\det(J_1-\lambda \I) = \lambda^{n-2}f_{\mathrm a}(\lambda)^{n-2}, \;
\mbox{ where  }
f_{\mathrm a}(\lambda)=\lambda^3+2a \lambda^2+4\lambda +4a .$$
The cubic polynomial $f_{\mathrm a}$ has roots with negative real parts since its positive coefficients, $1, 2a, 4, 4a$  satisfy Hurwitz' stability condition $(4)(2a)> 1(4a)$.

To understand the action of $J$  on the subspace $L_2$ spanned by $B_2$, calculate $JB_2$ using (\ref{eqJacobianGeneral}) and
$ C\bar c= f\bar c,\;   S\bar c = -f \bar s $,
$ C\bar s = d\bar s, \; \;  S\bar s=d\bar c$ from (\ref{EqCSatcs}). It follows that
\begin{equation}\label{EqJacobianRestrictionB2}
          JB_2 = \left[
          \begin{array}{cccc|cccc}
            \bar 0 & 0 & 0 & 0 & \bar c & \bar s & 0 & 0 \\
            0 & 0 & 0 & 0 & 0 & 0 & \bar c & \bar s \\
            f \bar c & d \bar s & f\bar s & -d \bar c & 0 & 0 & 2\bar c & 2\bar s \\
            -2a\bar c -f \bar s & d\bar c-2a \bar s & f \bar c  & d\bar s & -2\bar c & -2\bar s & -2a\bar c & -2a \bar s \\
          \end{array}
        \right].
\end{equation}
Each of the eight columns of $J B_2$ is a linear combinations of the eight columns of $B_2$; collecting the coefficients yields the  $8\times 8$ matrix $J_2$,
\bdima
 J_2 = \left(
\begin{array}{c|c}
  \OO_4
  & \I_4 \\
\hline
  \begin{array}{cccc}
                  f & 0 & 0 & -d \\
                  0 & d & f & 0 \\
                  -2a & d & f & 0 \\
                  -f & -2a & 0 & d \\
                \end{array}
   &
   \begin{array}{cccc}
   0 & 0 & 2 & 0 \\
   0 & 0 & 0 &2 \\
   -2 & 0 & -2a & 0 \\
   0 & -2 & 0 & -2a
   \end{array}
\end{array}\right).
\edima
A direct calculation of $B_2 J_2$ shows that it equals $JB_2$.

Recall  that $d=\frac{1}{n}\sum_{k=1}^n \sin^2(\theta_{0k})$ and $ f=\frac{1}{n}\sum_{k=1}^n \cos^2(\theta_{0k})$,  thus $d+f=1$. Set $\omega$ such that $f = 1/2 +\omega$ and $d = 1/2 -\omega$. Then $-\frac12<\omega <\frac12 $ and
\bdima 
J_2 =
\left(
\begin{array}{c|c}
  \OO_4
  & \I_4 \\
\hline
  \begin{array}{cccc}
                  \frac12 +\omega & 0 & 0 & -\frac12 +\omega \\
                  0 & \frac12 -\omega & \frac12 +\omega & 0 \\
                  -2a & \frac12 -\omega & \frac12 +\omega & 0 \\
                  -\frac12 -\omega & -2a & 0 & \frac12 -\omega \\
                \end{array}
   &
   \begin{array}{cccc}
   0 & 0 & 2 & 0 \\
   0 & 0 & 0 &2 \\
   -2 & 0 & -2a & 0 \\
   0 & -2 & 0 & -2a
   \end{array}
\end{array}\right).
\edima %
The characteristic polynomial of $J_2$, expanded using Mathematica, is equal to $$\det(J_2-\lambda I)=(1 + \lambda^2)( \lambda^6+ 4 a\lambda^5 + (5+4a^2) \lambda^4 + 14 a \lambda^3 + (4+8a^2) \lambda^2 + 4 a\lambda   + a^2(1 - 4 \omega^2)).$$
Therefore, the spectrum of $J_2$ consists of $\pm i$ and zeros of the polynomial
$$g_{\omega}(\lambda) =   \lambda^6+ 4 a\lambda^5 + (5+4a^2) \lambda^4 + 14 a \lambda^3 + (4+8a^2) \lambda^2 + 4 a\lambda   + a^2(1 - 4 \omega^2) .$$
We employ Hurwitz stability criterium \cite[Ch. XV, Section 6]{GantmacherMatrices} for the polynomial $g_\omega(\lambda)$;
 the Hurwitz matrix $H$ is:
\[H =
\left(
\begin{array}{cccccc}
 4a & 14a & 4a & 0 & 0 & 0 \\
 1 & 5+4a^2 & 4+8a^2 & a^2(1-4 \omega^2) & 0 & 0 \\
 0 & 4a & 14a & 4a & 0 & 0 \\
 0 & 1 & 5+4a^2 & 4+8a^2 & a^2(1-4 \omega^2) & 0 \\
 0 & 0 & 4a & 14a & 4a & 0 \\
 0 & 0 & 1 & 5+4a^2 & 4+8a^2 & a^2(1-4 \omega^2) \\
\end{array}
\right).
\]
The leading principal minors of $H$ are $\Delta_1 = 4a$, $\Delta_2= 6a+16a^3$, and
\begin{equation*}
\begin{array}{l}
\Delta_3 = 96 a^4+36 a^2, \\
\Delta_4 = 72 a^2 + a^6 (576 - 256 \omega^2) + a^4 (408 - 96 \omega^2), \\
\Delta_5 = 8 a^7 \left(144-128 \omega ^4+480 \omega ^2\right)+8 a^5 \left(156 \omega ^2+153\right)+288 a^3,\\
 \Delta_6 = 8(1-4\omega^2)\left[ 16 a^9 \left(9-8 \omega ^4+30 \omega ^2\right)+3 a^7 \left(52 \omega ^2+51\right)+36 a^5\right ],
 \end{array}
 \end{equation*}
all strictly positive for $-1/2<\omega <1/2$. By the Routh-Hurwitz criterion, the roots of the polynomial $g_\omega(\lambda)$ are in the left half-plane.

Overall $\det(J-\lambda I)=\det(J_1-\lambda I)\det(J_2-\lambda I)$, therefore
$$ \det(J-\lambda I)= \lambda ^{n-2} (1 + \lambda^2) f_{\mathrm a}(\lambda)^{n-2}g_\omega(\lambda)$$
where the polynomial $f_{\mathrm a}(\lambda)^{n-2} g_\omega(\lambda)$ has $3n$ roots in the half-plane  $\mbox{Re}\lambda <0.$
\end{proof}

\subsection*{Notation}  
To keep notations compact, we will adopt the following convention: for a vector $\bar \delta = (\delta_1,\ldots,\delta_n)^T$ and a real-variable function $f$, we write $f(\bar \delta)$ to denote the vector $(f(\delta_1),\ldots, f(\delta_n))^T$. Let $\mathbb 1_n=(1, 1, \dots , 1)^T\in \R^n .$

\begin{remark} Let $F:\R^n \to \R^2$ be the function
$\displaystyle F(\bar \delta)= \left [ \begin{array}{l}\bar c ^{\scriptscriptstyle T}  \cos (\bar \delta) -\bar s ^{\scriptscriptstyle T}  \sin (\bar \delta) \\
\bar c ^{\scriptscriptstyle T}    \sin (\bar \delta)  + \bar s ^{\scriptscriptstyle T} \cos (\bar \delta)
\end{array} \right ]$. Then $\Delta=\{ \bar \delta \in \R^n |\; F(\bar \delta )= (0, 0)^{\scriptscriptstyle T}\} $
is a smooth $n-2$ dimensional surface  near the origin in $\R^n.$  Similarly,  the set $\mathcal R$ of fixed points for (\ref{eqnNewSystem_abpq_coord}) is an $(n-2)$-dimensional manifold near  the origin in $\mathbb R^{4n}$.
\end{remark}

\begin{proof}
To study $\Delta$, note that $\displaystyle F(0)=\left [ \begin{array}{l} \sum _{k=1}^n\cos \theta _{0k}\\
 \sum _{k=1}^n\sin \theta _{0k} \end{array} \right ]  =\left [ \begin{array}{l} 0\\0
\end{array} \right ]. $
The linear approximation of $F$ centered at $\bar \delta _0 =0$ is a map of rank 2, since
$$ F(\bar \delta) \simeq \left [ \begin{array}{l} -\bar s ^{\scriptscriptstyle T} \bar \delta \\
\bar c ^{\scriptscriptstyle T} \bar \delta \end{array} \right ], \;  \nabla F= \left [ \begin{array}{l} -\bar s  \\
\bar c  \end{array} \right ],
$$
with linearly independent vectors $\bar s$ and $\bar c$.
We conclude that  $\Delta =F^{-1}((0, 0)^{\scriptscriptstyle T})$
 is a smooth $(n-2)$-dimensional surface for  $\bar \delta$ near the origin in $\R^n.$

Consider the map $G:\R^n\to \R^{4n}$ defined by $G(\bar \delta)=[\cos (\bar \delta ) - \mathbb 1_n, \sin (\bar \delta ), 0_n, 0_n].$ The linearization of $G$  has full rank near the origin due to the block $\sin (\bar \delta ) \simeq \bar \delta$.  Locally, $\mathcal R$ is the image of the smooth $(n-2)$-dimensional surface $\Delta $  under $G$, and thus  $\mathcal R$ is also a smooth  $(n-2)$-dimensional surface.
\end{proof}

  \begin{lemma}
 \label{LemmaNonDManifold}
 The center manifold  of (\ref{eqnNewSystem_abpq_coord}) coincides  near the origin with the surface
 \begin{equation}
 \label{eqManifoldM}\mathcal M = \mathcal R+\mathrm{Span} \left\{ \left(
 \begin{array}{c}
  \bar c \\  -\bar s \\ -\bar s \\ -\bar c \\
  \end{array} \right),
   \left(
    \begin{array}{c}
    \bar s \\ \bar c \\ \bar c \\ -\bar s \\
 \end{array} \right) \right \},
\end{equation}
where the set $\mathcal R $, defined in Lemma \ref{LemmaEquilibriumSolutions_abuw},
consists of the equilibrium points of (\ref{eqnNewSystem_abpq_coord}).
On $\mathcal M$, the dynamics of (\ref{eqnNewSystem_abpq_coord})  is stable near the origin.
\end{lemma}

\begin{proof}

Lemma \ref{LemmaNonDEigenspace} shows that the center manifold $M$ of $J$ at the origin is $n$-dimensional.
Let $M_\epsilon$ denote the part of $M \cap B_\epsilon,$ where $B_\epsilon$ is a small ball centered at the origin in $\R^{4n}$.

The  $\pm i $ eigenvalues of $J$ suggest that the center manifold may contain nested periodic orbits. We will identify such periodic solutions in the  $\{\bar a, \bar b, \bar u, \bar w\}$-system, by matching them to rotating states with a nonzero center of mass, $R_0=|R_0|e^{i\rho}$, from the  coordinates  $(x, y, \dot x, \dot y)$. The parameter $|R_0|$ governs the nesting.

Consider a nearby rotating state
   $\mathrm r_k(t) =\frac{1}{A} e^{i(\theta_{0k}+\delta_k)}e^{iAt}+R_0$, $k=1,\ldots,n$,
    viewed as a perturbation of the  rotating state $\{\frac{1}{A}e^{i\theta_{0k}}e^{iAt}\}$.
   Necessarily, $\bar \delta \in \Delta.$
   Using  $A=\frac{1}{a}$, we get $\mathrm r_k(at) = \frac{1}{A} e^{i(\theta_{0k}+\delta_k)}e^{it}+R_0$. Substituting into
   $(a_k+1) + i b_k  = A\mathrm r_k(at) \; e^{-i t}e^{-i \theta_{0 k}}$,
   per Equation (\ref{DefineChangeCoordinates_ab}),  gives $a_k + i b_k  =e^{i\delta_k}-1+ A R_0 e^{-it}e^{-i\theta_{0k}}$.
     We get $\dot a_k+i \dot b_k=-iA|R_0|e^{i\rho}e^{-it}e^{-i\theta_{0k}}$, or $\bar u +i \bar w= A|R_0|e^{(\rho-t)i}(-\bar s-i\bar c).$

   In $\{\bar a, \bar b, \bar u, \bar w\}$-coordinates, the nearby rotating state centered at   $R_0=|R_0|e^{i\rho} $ is:
\begin{equation}\label{eqnRingStateInABPQ_Coord} \left[
                   \begin{array}{c}
                     \bar a \\
                     \bar b \\
                     \bar u \\
                     \bar w \\
                   \end{array}
                 \right] =
                 \left[
                   \begin{array}{c}
                     \cos(\bar \delta)-\mathbb 1_n \\
                     \sin(\bar \delta) \\
                     0 \\
                     0 \\
                   \end{array}
                 \right]
                 + A|R_0|\cos(\rho-t)
                 \left[
                   \begin{array}{c}
                     \bar c \\
                     -\bar s \\
                     -\bar s \\
                     -\bar c \\
                   \end{array}
                 \right]
                  + A|R_0|\sin(\rho-t)
                 \left[
                   \begin{array}{c}
                     \bar s \\
                     \bar c \\
                     \bar c \\
                     -\bar s \\
                   \end{array}
                 \right]
\end{equation}
The first term of (\ref{eqnRingStateInABPQ_Coord}) is a point in $\mathcal R$, thus a fixed point of (\ref{eqnNewSystem_abpq_coord}).
 Allowing $\bar \delta$ to be arbitrary in $\Delta$,
$|R_0|$ to vary in $[0, \infty)$ and $\rho $ in $[0, 2\pi]$ leads to the first term attaining range $\mathcal R$ and the last two terms spanning a plane. Thus the union of these periodic orbits is the set $\mathcal M $ given in (\ref{eqManifoldM}). Near the origin, $\mathcal M $ is an $n$-dimensional, smooth manifold. Figure \ref{FigureCenterManifold} illustrates this construction.

 \begin{figure}[h!]
\centering
\includegraphics[width=.9\textwidth]{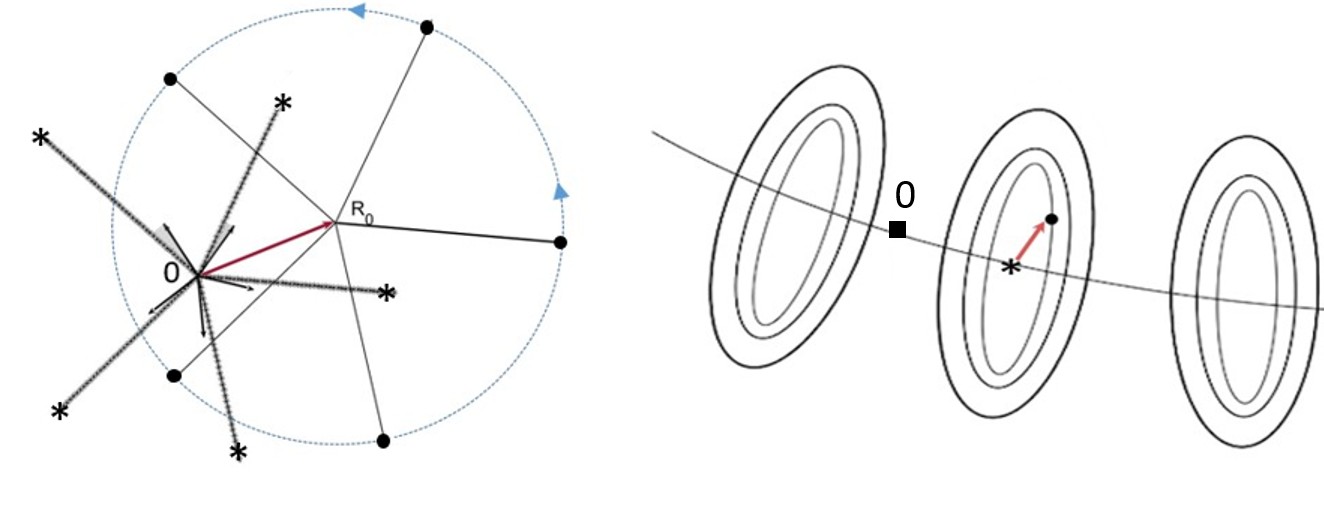}
\caption{\label{FigureCenterManifold}
Depiction of  the manifold $\mathcal M$ for $ K=1, n=5.$
Left: Initial locations for two perturbations of a rotating state, shown in physical coordinates. The short vectors from the origin indicate the angles $\theta_{k,0}$.
Star markers represent $e^{i( \delta_k+\theta_{k,0})}$, and round markers indicate $R_0+e^{i( \delta_k+\theta_{k,0})}$. Shaded angles illustrate phase perturbations $\delta_1$ and $\delta_2.$
Right: The manifold  $\mathcal M$ in $(\bar a, \bar b, \bar u, \bar w)$-coordinates.
The unperturbed rotating state corresponds to the origin (square marker).
 The submanifold $\mathcal R$ (shown as a line)  consists of fixed points  of (\ref{eqnNewSystem_abpq_coord}) representing rotating states about the origin
(star marker configurations).
 Through each point in $\mathcal R$  passes a plane  composed of closed orbits (concentric ellipses) corresponding to rotating states with nonzero center of mass $R_0$ (round markers). In physical coordinates $R_0$ is stationary; in the counterclockwise co-rotating frame, it moves clockwise along an elliptical path.}
\end{figure}

If $|\bar \delta |$ and $|R_0|$ are small,
 the  periodic orbits in (\ref{eqnRingStateInABPQ_Coord})  remain near the origin in $\R^{4n}$ for all $t\in \R,$ and therefore they must belong to the center manifold $M_\epsilon$.

  We get that the manifold $\mathcal M \cap B_\epsilon $ is a subset of $ M_\epsilon $ and is therefore tangent to the center eigenspace of $J$;  all these manifolds are $n$-dimensional. Moreover $\mathcal M \cap B_\epsilon $ consists of trajectories of (\ref{eqnNewSystem_abpq_coord}), and as a consequence, it is invariant under its flow. This implies that $\mathcal M \cap B_\epsilon $ is a center manifold.

The dynamics on $\mathcal M \cap B_\epsilon $ is given explicitly in (\ref{eqnRingStateInABPQ_Coord}), demonstrating stability.
\end{proof}

In view of center manifold theory \cite[Theorem 2 in Section 2.4]{Carr:CentreManifold},  Lemma \ref{LemmaNonDManifold} establishes the stability of the zero solution (\ref{eqnNewSystem_abpq_coord}).
Furthermore,  any solution of (\ref{eqnNewSystem_abpq_coord}) that starts in a small neighborhood of the origin in $\mathbb R^{4n}$  approaches the center manifold exponentially fast. In other words, in the {\it original} coordinate system,  all solutions with initial conditions near those of a non-degenerate rotating state will converge to a rotating state configuration (possibly {\it another} nearby rotating state).

\vskip2mm
Note that for $n$ even, as the phases $\theta_{0k}$ approach a degenerate configuration, i.e. $\omega^2 $ approaches $\frac14,$ one of the roots of $g_\omega$ approaches zero. This increases the multiplicity of the zero eigenvalue for the characteristic polynomial of $J$. For degenerate configurations, the center manifold has dimension $n+1$, with the set of fixed points $\mathcal R$ having codimension three. The manifold of periodic orbits is a codimension-one submanifold of the central manifold, and no longer accounts for the whole dynamics on the central manifold. The latter renders the geometric argument developed in this section inapplicable, and motivates Section \ref{SectionNotTaylor}.

For fixed $a$, if $n$ is odd, considering all possible rotating states generates values of $\omega$ that form a compact subset of $(-1/2, 1/2), $ therefore the roots of the polynomials $g_\omega$ lie in a compact subset of the left-half plane.  Thus there exits $\eta >0$ such for {\em all} the rotating states, the roots of $f_{\mathrm a}g_\omega$ satisfy $\mbox{Re} \lambda \leq -2 \eta, $ implying
trajectories approach their limit cycle at a rate faster than  $e^{-\eta t }$.

%

\section{A Pseudo-Orbit Based Approximation of Center Manifolds }
\label{SectionNotTaylor}

Non-isolated fixed points are common in dynamical systems with symmetries or identical agents, yet rigorous analytical tools for such settings are lacking. Traditional Taylor expansions of the center manifold may not accurately capture the system’s dynamics, as truncated flows can differ qualitatively from the true flow, regardless of polynomial order. We propose a framework for approximating the center manifold that captures the true dynamics near non-isolated fixed points, outperforming Taylor-based methods for stable dynamics (if orbits escape away from the non-isolated fixed
points, the scale of the flow get bigger, diminishing the impact of our method). The method is applied in Section \ref{SectionDegenerateRingState}.

Our construction uses  contraction operators acting on function spaces and their fixed points. To avoid confusion, we use {\em equilibrium} to refer to points that are left unchanged by the flow of a dynamical system, and {\em fixed} to describe functions that are left unchanged by these operators, with some flexibility in how the terms are applied provided the context is clear.

\subsection*{Notation}: We use the standard big-O notation  $\mathcal O(f(x))$ to convey a vector that has magnitude at most a positive multiple of $|f(x)| $ for $x$ near zero. For vectors $x\in \R^{d_1}, y\in \R^{d_2}$ we denote $col(x, y)$ the column vector $(x_1, \dots x_{d_1}, y_1, \dots, y_{d_2})^{\scriptscriptstyle T}.$


\subsection*{Example}
We illustrate the breakdown of Taylor approximations for the center manifold near non-isolated equilibria and present our construction, which accurately captures the dynamics, using the  three-dimensional system below.
\beq
\begin{array}{ll}
\dot{x}=& x(y-x\sin x) \\
\dot{y}=& -y^2(y-z)\\
\dot{z}=&-z+ x(y-x\sin x)(\sin x + x\cos x) + x \sin x.
\end{array}
\label{Taylorexample}
\eeq
The set of equilibrium points  is $\{ (x, x\sin x, x \sin x ), \; x \in \mathbb R\} \ni (0,0,0)$. The origin has the  Jacobian matrix $\textrm{diag} \{0,0,-1\}.$

The line $x=0, y=0$ is invariant, and has the direction of the $(-1)$ eigenvector, thus the stable manifold is the $z$-axis.
The surface parametrized  by $(x, y, x \sin x)$ is invariant under the flow
of  $(\ref{Taylorexample} )$ since
$\displaystyle d(z-x\sin x)/dt= (-1) (z-x\sin x).$ Moreover, the surface is tangent to the $(x, y)$-plane at the origin, therefore it is the center manifold.
Using the center manifold function $z=h(x,y)=x \sin x$,  the study of stability of $(\ref{Taylorexample} )$ can be reduced to the study of
\beq
\begin{array}{ll}
\dot{x}=& x(y-x\sin x) \\
\dot{y}=& -y^2(y-x\sin x)\\
\end{array}.
\label{TaylorexampleTrue}
\eeq
The points with $y= x \sin x$, $x \in \mathbb R$, are equilibrium points for the System $(\ref{TaylorexampleTrue})$. The flow on the center manifold is illustrated in Figure  \ref{SineTaylor}, Left.
 One can show that the origin is  stable (but not asymptotically stable) and that the $\omega$-limit points for (\ref{TaylorexampleTrue}) are precisely the equilibrium points $\{(x, x \sin x)\}.$

If Taylor polynomials are used to approximate $h$, we get a truncated dynamical system
\beq
\begin{array}{ll}
\dot{x}=& x(y-x\sin x) \\
\dot{y}=& -y^2(y-xT_{2k-1}(x)),\\
\end{array}
\label{TaylorexampleTruncated}
\eeq
where  $T_{2k-1}(x)$ denotes the Taylor polynomial of degree $2k-1$ for  $\sin x .$
Depending on whether $k$ itself is even or odd, $xT_{2k-1}(x)$ is an under- or over-estimate of $h$.

We only focus on the case when $k$ is even. 
Then $ x T(x) \leq x \sin x $ (we dropped the subscript $2k-1$ from the Taylor polynomial).

The truncated system (\ref{TaylorexampleTruncated}) splits the curve of fixed points $\{(x, x\sin x) \}$ into two separate nullclines. 
In (\ref{TaylorexampleTruncated}) the  flow breaches the nullcline $y=x \sin x$ creating transport across what is supposed to be a set of fixed points for the original system. Once a trajectory of (\ref{TaylorexampleTruncated}) enters the region between the nullclines  $y=x T(x)$ and $y=x \sin x$ it is trapped there, and it approaches the origin  \footnote{One can verify that the region $ xT(x)\leq y\leq x\sin x$ below $y=1/6$ is forward-invariant and the function $L=x^2+y^2$ is decreasing in this region. Applying LaSalle's invariance principle and using the fact that the only points where $L$ has zero derivative satisfy
two of the conditions $x=0, y=0, y=x \sin x, y=x T(x),$ we obtain that the origin is the only $\omega$-limit point for (\ref{TaylorexampleTruncated}).}. Thus, unlike the original system,  (\ref{TaylorexampleTruncated})  has the origin as an {\it asymptotically stable} fixed point. The truncated flow is illustrated in Figure  \ref{SineTaylor}, Right.

 \begin{figure}[h!]
 \centering
\includegraphics[scale=0.7]{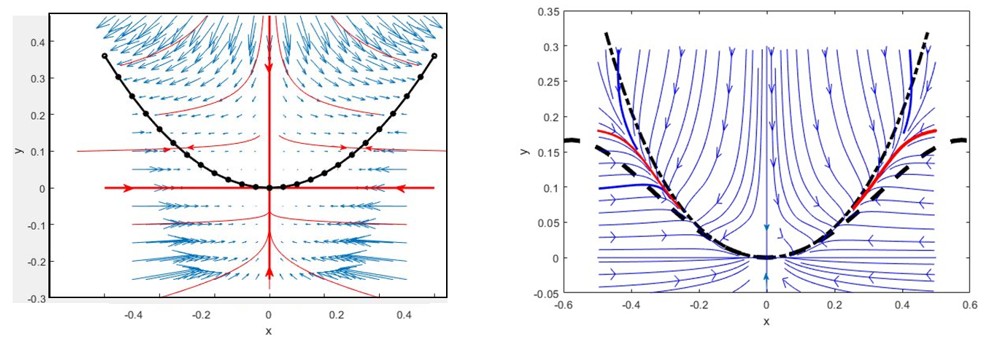}
\caption{\label{SineTaylor}
Flow near the origin for system $(\ref{TaylorexampleTrue})$  and its truncated system (\ref{TaylorexampleTruncated}). Left: The accurate phase portrait.  Fixed points along $y = x\sin x$ are shown as black dots. Blue arrows indicate the vector field; red curves show trajectories converging to fixed points. The origin is stable but not asymptotically stable; in the upper half-plane, only points on the $y$-axis converge to it.
   Right: Phase portrait of the truncated system (\ref{TaylorexampleTruncated}) with nullclines $y = x\sin x$ and $y = xT(x)$ (dashed). Trajectories in the first quadrant are eventually trapped between the two nullclines and converge to the origin, which is asymptotically stable, unlike the original  system (\ref{Taylorexample})  from which the approximation is derived.}
\end{figure}

\subsection*{A New Approximation of the Central Manifold}
\label{SectionANewApproach}

We present our center manifold technique in the general setting, i.e for a system having the origin as an  equilibrium point with stable linearization, expressed in coordinates $(x_s, x_c)$ that  block-diagonalize the system:

\beq
\begin{array}{ll}
\dot x_s= & A_s x_s + f(x_s, x_c) \\
\dot x_c= & B x_c + g(x, x_c)\\
\end{array}
\label{standardcentersetup}
\eeq
Here $A_s$ is an $s\times s$ matrix whose eigenvalues have negative real part, $B$ is a $c\times c$  with zero eigenvalues. The functions $f$ and $g$ are  $C^r$ functions ($r\geq2 $)  that equal zero and have zero gradients at the origin. Denote by $E$ the set of equilibrium points for (\ref{standardcentersetup}), and by $E_c$ its projection onto the zero eigenspace. We  assume the origin to be non-isolated in $E$.

 The center manifold theorem guarantees the existence of a map $h=h(x_c) $ defined
  in a neighborhood $\mathcal N _c$ of the origin in $\R^c$ with values in $\R^s$ with $h(0)=0, \nabla h(0)=0$ and such that its graph $\{(h(x_c),x_c), \; x_c\in \mathcal N _c \} $ is
 locally invariant under $(\ref{standardcentersetup})$. Moreover, the local stability of $(\ref{standardcentersetup})$   is equivalent to the local stability of the $c$-dimensional system
 \beq
 \dot x_c=   B x_c + g(h(x_c), x_c).
 \label{standardcentersetupreduced}
 \eeq

 Our method of approximating the function $h$  references its construction via fixed points for operators defined on the space of slow-growing functions,  \cite{Bressan:2007}.
Let  $\eta>0 $ be less than the spectral gap of $A_s;$  define the linear space
$$X_\eta=\left\{ w:(-\infty, \infty ) \to \R ^{s+c}, w \mbox{ continuous },\sup _{t\in (-\infty, \infty )} |w(t)e^{-\eta |t|}|
 < \infty \right\}.$$%
 $X_\eta$, with the norm $||w||_\eta=\sup _{t\in (-\infty, \infty )} |w(t)e^{-\eta |t|}|$, is a Banach space.
 For example, the constant function $w(t)= w_0$ is in $X_{\eta}$ , where
 its norm coincides with the norm $|w_0|$ of $w_0$ in $\mathbb R^{s+c} $; the linear function $t w_0$ is in $ X_\eta$ and has norm  $\displaystyle \frac{1}{e \eta} |w_0|$ .

Define $\Gamma : \mathcal N_c \times X_\eta \to X_\eta $, where $\mathcal N_c $ is a neighborhood of $0$ in $ \R^c$, as follows \footnote{Technically the functions $f, g$  are modified outside of a neighborhood of the origin to have compact support, with their support contained within a region where the nonlinearities are small. Different modifications lead to different operators, and consequently, to different possible maps $h$.}:
\beq
\label{standardGamma}
\Gamma_{x_c}(w)(t)=\left [
\begin{array}{c}
\bigintss _{-\infty }^t\; e^{A_{s}(t-\tau)} f(w(\tau)) d\tau
\\
\\
e^{Bt}\; x_c +\bigintss _0^t\; e^{B (t-\tau)} g(w(\tau)) d \tau
\end{array}
\right ], \; \mbox{ for } x_c \in \mathcal N _c, w\in X_\eta.
\eeq
Each operator $\Gamma _{x_c}$ is a contraction.
  Let $F_{x_c}$ denote the function that is the fixed point for the operator $\Gamma _{x_c}$. Then the set $\mathcal M = \{F_{x_c}(0) : x_c\in \mathcal N_c\}$ is the center manifold of the system (\ref{standardcentersetup}).
The projection of $F_{x_c}(0)$ onto the subspace $\mathbb R^s $ defines the  center manifold map, meaning that
 $h: \mathcal N_c\rightarrow \mathbb R^s $ satisfies $(h(x_c),x_c)=F_{x_c}(0)$.

 Moreover, for any function $w$ in $ X_\eta$ the sequence of iterates $\Gamma _{x_c} ^n(w) $ converges to the fixed point function $F_{x_c}(t)$ in $ X_\eta$ and $\Gamma _{x_c} ^n(w)(0) $ converges to $(h(x_c),x_c)$ in $\mathbb R^{s+c}.$
 If $\mathcal N_c$ is sufficiently small, the contraction factors of $\Gamma _{x_c}$ are  below $\frac12, $ thus for any $w\in X_\eta$ we have $ \|F_{x_c}-w\|_\eta \leq 2 \| \Gamma _{x_c}(w)-w\|_\eta .$

Although $h$ is not unique, its value at equilibrium points is uniquely determined:  if $(x_{s,0}, x_{c,0})$ is an equilibrium point for (\ref{standardcentersetup}) ,
then $h(x_{c,0})=x_{s,0}.$

Our technique is as follows: to approximate $h(x_c),$ we shadow the orbit $\{\Gamma _ {x_c} ^n(w)\} _n ,$ and track the ensuing errors, with the goal of having errors below the magnitude of the vector field $g(h(x_c),x_c).$ We use the discrepancy $\| \Gamma _{x_c}(w)-w\|_\eta $ as a proxy for the error $(h(x_c), x_c)-w$.
To initiate the iterates of $\Gamma _{x_c}(w)$ we use a function $w(t)$ that is a constant $w_0$ (to be specified). We then shadow the orbit with functions that are constant.
For the swarm systems studied in the present paper, zero eigenvalues have no associated Jordan blocks, so we focus the shadowing and  error-tracking in the case where $B$ is the zero matrix.

Let $x_c\in \mathcal N_c$. We evaluate  $\Gamma _{x_c}$ when applied to a constant function
 $w(t)= col (w_{0s}, w_ {0c}) \in X_\eta $.
 We get that
  \beq
  \label{constantGammaG1}
\Gamma_{x_c}(w)(t)=\left [
\begin{array}{c}
\left ( \bigintss _{-\infty }^t\; e^{A_{s}(t-\tau)} d\tau \right ) f(w_{0s}, w_ {0c})
\\
\\
 x_c +  \left ( \bigintss _0^t\; e^{B (t-\tau)} d \tau \right ) g(w_{0s}, w_ {0c})
\end{array}
\right ] =
\left [
\begin{array}{c}
-A_{s}^{-1}f(w_{0s}, w_ {0c})
\\
\\
 x_c+ t g(w_{0s}, w_ {0c})
\end{array}
\right ].
  \eeq
  Once the term linear in $t$ is removed, the last $c$ components of $\Gamma$ are always $x_c$, implying we should work in the family of constant functions $\displaystyle w(t)= col( w_{0s}, x_c ). $
We get a discrepancy
\beq
\Gamma_{x_c}\left ( \begin{array}{c} w_{0s}\\ x_c \end{array} \right )(t) - \left [
\begin{array}{c}w_{0s}\\
x_c \end{array}
\right ]
 =\left [
\begin{array}{c}
-A_{s}^{-1}(Aw_{0s}+f(w_{0s}, x_c)
\\
 t g(w_{0s},x_c)
\end{array}
\right ].
\label{constantGammaGeneral}
  \eeq

We conclude that the error made by using one iteration of $col (w_{0s}, x_c)$ in approximating $h(x_c)$ is comparable (or smaller) to the larger of $|g(w_{0s}, x_c)|$ and $| Aw_{0s}+f(w_{0s}, x_c)|$, which  reflect of how close the current iterate is from the nullclines of the system: there exists a constant $C$ such that
\beq
\label{eqHerrorGeneral}
|h(x_c)-w_{0s}|\leq C \max \{|g(w_{0s}, x_c)|, | A w_{0s}+f(w_{0s}, x_c)|\}.
\eeq

If the current error is not yet satisfactory, on the basis of (\ref{constantGammaG1}), the next approximation is produced by evaluating the operator at
$\displaystyle \left [
\begin{array}{c}
-A_{s}^{-1}f(w_{0s}, x_c)
\\
 x_c
\end{array}
\right ].$

Given $x_c$ in $\mathcal N_c,$ we aim to initiate or shadow the approximation with a point $(w_{0s}, x_c)$  close to the nullclines of the system. (For $x_c$ in $\mathcal N_c$, the nullcline equation $ A_s x+f(x, x_c) =0$ has a unique solution $x$, since  $A_s$ is non-singular. However, in high dimensional systems it may be impossible to solve explicitly.)

 \vskip5mm We apply the technique to the three-dimensional example (\ref{Taylorexample}),  using the notations of  (\ref{standardcentersetup}).
The $z$-axis is the stable subspace, thus $x_s=z$, with  matrix $A_s=-1$.
The zero eigenspace is the $(x,y)$-plane, thus $x_c=(x,y)$, with matrix $B=0.$

Use (\ref{standardGamma}):  evaluated at a {\it constant}  function $[ w_{0s}, (x,y)]^T$, $\Gamma_{(x,y)}$ equals
$$ \Gamma_{(x,y)}\left [ \begin{array}{l} w_{0s} \\x\\ y \end{array} \right ](t) =
 \left [ \begin{array}{l}
 x(y-x\sin x)(\sin x + x\cos x) + x \sin x\\ x+ t\;  x( y-x \sin x) \\ y-t\; y^2(y-w_{0s})
 \end{array} \right ].$$
 We use as $w_{0s}$ the value obtained from the equation $\dot z= 0$, namely
 $w_{0s}=\alpha (x,y)= x(y-x\sin x)(\sin x + x\cos x) + x \sin x. $ This simplifies the top component of $\Gamma_{(x,y)}$ to just $\alpha .$
 Note that the factor $(y-w_{0s})$ appearing in the 3rd component becomes $(y-x \sin x) (1-x(\sin x + x\cos x)).$
 \beq
 \label{etanormexample}
  \Gamma_{(x,y)}\left [ \begin{array}{c}
\alpha \\x\\ y \end{array} \right ] -
\left [ \begin{array}{l}
 \alpha\\ x\\ y \end{array} \right ] =t (y-x\sin x) \left [ \begin{array}{l}0 \\ x \\ -y^2(1-x\sin x-x\cos x) \end{array} \right ].
 \eeq
As a point in  $X_\eta$, the linear function of $t$ appearing in the right side of (\ref{etanormexample}) has norm  below $|y-x\sin x | \mathcal O (|x,y|), $ implying
 that the fixed point $F_{x,y}$  of $\Gamma_{(x,y)}$ is within $(y-x\sin x) \mathcal O (|x,y|)$ from $\alpha(x,y).$
 Since $|\alpha(x,y)-x\sin x|= (y-x\sin x) \mathcal O (|x,y|)$,
 we get that the center manifold map is
$h(x,y)= x\sin x + (y-x\sin x) \mathcal O (|x,y|), $ and the dynamics on  center manifold is described by:

\beq
\begin{array}{l}
\dot{x}= x(y-x\sin x) \\
\dot{y}= -y^2(y-x \sin x )\left ( 1+ \mathcal O (|(x,y)|) \right) .
\end{array}
\label{TaylorexampleReduced}
\eeq
We note that System $(\ref{TaylorexampleReduced})$ faithfully captures the locations of equilibrium  points,  $(x, x\sin x),$ and the behavior of the trajectories as converging to the origin or to an off-origin equilibrium point, depending on whether the initial condition is from $\{ (x_0, y_0)\; |\;  x_0=0 \; \mbox{ or } \; y_0< x_0 \sin x_0\}$ or not.
\footnote{
The phase plane of $(\ref{TaylorexampleReduced})$ splits into three invariant regions: the lower half plane,  the region with $0<y<x \sin x , $ and the region above the fixed point curve (i.e $y> x \sin x$).
One can 
verify that within the region $y<0$ the (Lyapunov) function $L_1=x^2+|y|=x^2-y$ decreases
forcing orbits to accumulate to the origin. Within the region $0<y<x \sin x $, the function $L_2=x^2$ decreases (which bounds $y(t)$ by $x_0\sin x_0$) and all the $\omega$-limit points satisfy $y= x \sin x$. In the region
$y> x \sin x$ the function $L_3=y$ decreases (this also bounds the $x$ variable), with the $\omega$-limit points satisfying $y= x \sin x.$}

\vskip5mm

 We conclude this section  by bounding the error introduced by applying our technique to system (\ref{standardcentersetup}) in regions close to the set of equilibrium points, its intended scope. In a practical setting, for a small $x_c$,  we begin the iteration with an initial guess (or a shadow from a prior iteration step) $x_s=\alpha (x_c)$ where $|\alpha (x_c)|\leq |x_c|$, to be consistent with $\alpha (x_c)$ approximating the zero-gradient function $h$ (a priori $h(x_c)=\mathcal O(|x_c|^2)).$ For (\ref{Taylorexample}) we used the nullcline $\dot z=0$  to construct the guess $\alpha $;  in Section \ref{SectionDegenerateRingState}, we will use the explicit formula of the map $h$ for equilibrium points and extend its domain from the set $E_c$ to the neighborhood $\mathcal N_c$. Both techniques lead to $\alpha(x_c)$ satisfying
  \beq
  \label{eqnearE}
   A_s\alpha(x_c)+f(\alpha(x_c), x_c) =\mbox{dist}(x_c,E_c) \mco(|x_c|),
  \eeq
where $\mbox{dist}(x_c, E_c)$ denotes the distance from a point $x_c\in \mathcal N _ c$ to the set of projected equilibriums $E_c.$ Note that  $\mbox{dist}(x_c, E_c)\leq |x_c|,$ since $0\in E_c$, but for points very close to $E $ we could have $\mbox{dist}(x_c, E_c)\ll |x_c|.$
Necessarily, guesses $\alpha $ satisfying (\ref{eqnearE}) are exact at equilibrium points, meaning $\alpha (x_c)=h(x_c)$ if $x_c\in E_c$, otherwise both
 satisfy the equation $A_sx_s+f(x_s, x_c)=0$, which has a unique solution $x_s$.

\begin{remark}
\label{RemarkApproxNullcline}
 If $\alpha :\mathcal N_c \to \R^s $ satisfies (\ref{eqnearE}) and $\alpha (x_c)\leq |x_c|$ for all $x_c\ \in \mathcal N_c$,
then
\beq
\label{equationherror}
 |h(x_c)-\alpha (x_c)|=\mbox{dist}(x_c,E) \mco(|x_c|)
  \eeq
  and the error in the approximation of the vector field $g$ of (\ref{standardcentersetupreduced}) is
$$ | g(h(x_c), x_c)-g(\alpha(x_c),x_c)|= \mbox{dist}(x_c,E) \mco(|x_c|^2). $$
\end{remark}

\begin{proof}

Let $x_c\in \mathcal N_c$ be fixed. Recall that $h(x_c)$ has the $s$-components of the fixed point $F_{x_c}$ of at $\Gamma_{x_c}$. We use $  \|F_{x_c}- (\alpha (x_c), x_c)\|_{\eta} \leq 2  \|\Gamma_{x_c}(\alpha (x_c), x_c)- (\alpha (x_c), x_c)\|_{\eta}.$

Use $(\ref{constantGammaGeneral})$  to simplify the latter.  We obtain that
\beq\label{eqnGammaAppendixAux1}
\Gamma_{x_c}(\alpha (x_c), x_c)(t) -(\alpha (x_c), x_c)=\left [
\begin{array}{c}
-A_{s}^{-1}(A_s \alpha (x_c)+f(\alpha (x_c), x_c)\\
t g(\alpha (x_c), x_c)
\end{array} \right ],
  \eeq
 having an $X_\eta$-norm that
 is comparable (or smaller) to the larger of $|g(\alpha (x_c), x_c)|$ and $| A_s \alpha (x_c)+f(\alpha (x_c), x_c)|$.
To complete the proof of (\ref{equationherror}), we need to show that  $|g(\alpha (x_c), x_c)|$ and $| A_s \alpha (x_c)+f(\alpha (x_c), x_c)|$ are $\mbox{dist}(x_c,E) \mco(|x_c|);$ the latter is, by assumption (\ref{eqnearE}).

For $x_c$, denote by $x_E$ (one of) the point(s) in $E_c$ that is nearest to $x_c $. Note that $x_E $ lies  in the closed ball  $ B(x_c ,|x_c|)$   and satisfies  $g(\alpha (x_E), x_E), x_E)=0.$
Since $|\alpha (x_E)|\leq | x_E|\leq 2|x_c|$ we conclude that the segment $S$
from $(\alpha (x_c), x_c)$ to $(\alpha (x_E), x_E)$ is contained in a ball centered at the origin of $R^{s+c}$ of radius $\mco (|x_c|)$, where the gradient of $g$ (which is linear or smaller near the origin) is also $\mco(|x_c|).$

For points along the segment $S$, for some constant $C$, we have  $|\nabla g|\leq C|x_c|$  and
$$ |g(\alpha (x_c),x_c)|= |g(\alpha (x_c), x_c)-g(\alpha (x_E), x_E)|\leq C|x_c| |x_c-x_E|= C|x_c| \mbox{dist}(x_c, E_c). $$
This proves the inequality (\ref{equationherror}), which we use next to  refine the error estimate for the vector field $g$. As above, for points on the segment from $(\alpha (x_c),x_c)$ to $(h(x_c),x_c)$ the gradient of $g$ is bounded by $C' |x_c|$ and
$$ | g(h(x_c),x_c)-g(\alpha (x_c),x_c)|\leq C'|x_c| |h(x_c)-\alpha(x_c)|= |x_c| \mbox{dist}(x_c, E_c)\mco (|x_c|).  $$
\end{proof}

%

\section{Stability of Degenerate Rotating States}\label{SectionDegenerateRingState}

In this section, we prove the Stability Theorem \ref{TheoremMainResultIntro} for the case of degenerate rotating states,  $\mathrm r_k(t) = \frac{1}{A}e^{iAt} e^{i  \theta_{0 k}}$, where $\theta_k \in \{ 0, \pi\}$. Necessarily, the number of particles $n$ must be even.   Set $$n =2N.$$

Assume that $\theta_{01}=\cdots = \theta_{0N} = 0$ and $\theta_{0 N+1}=\cdots = \theta_{0n} = \pi$. We begin in the $(\bar a, \bar b, \bar u, \bar w)$-coordinates around $\{\mathrm r_k \}$ as outlined in Section \ref{SectionRingStateChangeOfCoordinates}. The proof of stability consists of the following main steps, numbered to  match the corresponding subsections.

\begin{enumerate}
\item We change the $(\bar a, \bar b, \bar u,\bar w)$-coordinates to explicitly separate neutral and stable directions. We denote the neutral variables associated with eigenvalue $\lambda =0$ by $Z\in \mathbb R^{n-1}$ and with $\lambda = \pm i$ by  $\delta_1,\delta_2\in \mathbb R$. The stable variables will be denoted by $\be,\bd,\ba\in \mathbb R^{n-1}$ and $\gamma_1,\gamma_2\in \mathbb R$. We determine the differential equations governing their dynamics, which decouple the first $4n-2$ components from  the variables $\delta_1$ and $\delta_2$.
\item We  approximate the center manifold for the system in $(Z,\be,\bd,\ba,\gamma_1,\gamma_2)$, a.k.a the reduced  system, and find an approximation for the flow on the center manifold. 
\item We establish stability for the reduced system about the origin.
\item We reintroduce $\delta_1$ and $\delta_2$ and establish stability for the full system. 
\end{enumerate}

%
%

\subsection{The change of basis from $(\bar a,\bar b, \bar u, \bar w)$ to $(Z,\be,\bd,\ba,\gamma_1,\gamma_2, \delta_1, \delta_2)$}

\vskip1mm

\subsection*{Notation}
If $p$ is a positive integer, let $\mathbb 1_p \in \R^p  $ and $\mathbb 1_{p,p}\in \R^{p\times p}$ denote the vector and square matrix of dimension $p$ with all entries equal to 1,  that is
$$ \mathbb 1_p=(1, 1, \dots , 1)^{\scriptscriptstyle T}, \; \;
\mathbb 1_{p, p}=\mathbb 1_p \mathbb 1_p^{\scriptscriptstyle T}.$$
Let  $X=(1, \dots 1, -1, \dots -1)^{\scriptscriptstyle T}=col(\mathbb 1_N, -\mathbb 1_N). $  Note that $X$ has the same role as that of the vector $\bar c$ with components $\cos \theta _{0,k}$ in the non-degenerate setting.

 Given vectors $\bar x = (x_1, x_2, \ldots,x_p)^{\scriptscriptstyle T}$ and $\bar y = (y_1, y_2, \ldots, y_p)^{\scriptscriptstyle T}$, we use ``$\odot$''  to denote the element-wise product:  $\bar x \odot \bar y $ is the vector in $\R^p$ of components $x_ky_k.$
 Given and
 a real-variable function $f$,  denote by $f(\bar x)$  the vector $(f(x_1),\ldots, f(x_n))^{\scriptscriptstyle T}$ .

\vskip2mm

We proceed from the definition of $(\bar a, \bar b, \bar u, \bar w)$ in (\ref{DefineChangeCoordinates_ab}) applied to
$\theta_{01}=\cdots = \theta_{0N} = 0$ and $\theta_{0 N+1}=\cdots = \theta_{0n} = \pi$.
For a more compact expression of the nonlinear functions $\mathcal U, \mathcal W$
define $\mathbb s(t)\in \R^n$ and $\mathbb c(t)\in \R^n$ as \footnote{Note that $\displaystyle \left ((1+\mathbb c_k) +i \; \mathbb s_k \right) i e^{it}$ represents the velocity of particle $k$ in the original coordinates at time $(\mathrm a t)$ for $k$ in $1 \dots N$ and it is the opposite of the velocity for $k$  in $N+1 \dots 2N.$}
$$\mathbb s_k = u_k-b_k\mbox{ and }\mathbb c_k=a_k+w_k.$$
We get
\bdima
\begin{array}{l}
\mathcal U_k =  -(\sbb_k^2+2\cbb_k+\cbb_k^2)\sbb_k = -\mathbb{s}_k \left ((\mathbb{c}_k+1)^2 +\mathbb{s}_k^2 -1)\right ) \\
\mathcal{W}_k =  -(\sbb_k^2+\cbb_k^2)-(\sbb_k^2+2\cbb_k+\cbb_k^2)\cbb_k =
(-1)(\mathbb{c}_k +1)\left ( (\mathbb{c}_k+1)^2 +\mathbb{s}_k^2 -1\right )+2\mathbb{c}_k.
\end{array}
\edima
The nonlinear part of the vector field (\ref{eqnNewSystem_abpq_coord}) has block components $0, 0, \mathrm a\; \mathcal U, \mathrm a \; \mathcal W$ with
\begin{equation}
\label{eqn_abuw_sys_nonlinearities}
\begin{array}{l}
\mathcal U =  -\mathbb{s}\odot \left ((\mathbb{c}+\mathbb 1_n)^2 +\mathbb{s}^2 - \mathbb 1_n \right ) \\
\mathcal{W} =
(-1)(\mathbb{c} +\mathbb 1_n ) \odot \left ( (\mathbb{c}+ \mathbb 1_n)^2 +\mathbb{s}^2 -\mathbb 1_n \right )+2\mathbb{c}
\end{array}
\end{equation}

This subsection is dedicated to identifying a  basis of $\R^{4n}$ that separates the stable and center directions of (\ref{eqnNewSystem_abpq_coord}), and to establishing the equation of motion in the new coordinates, (\ref{eqnNewCoordsDegFull}) and (\ref{eqnNewCoordsJustDelta}).

The Jacobian for (\ref{eqnNewSystem_abpq_coord}) is given in   (\ref{eqJacobianGeneral}).  Using  $\sin(\theta_{0m}-\theta_{0k})=0$, we get:
\begin{equation}\label{eqJacobianDegenerate}
  J = \kbordermatrix{
    & \partial a & \partial b & \partial u & \partial w  \\
    \partial \dot a & \OO_n & \OO_n & \I_n & \OO_n  \\
    \partial \dot b & \OO_n & \OO_n & \OO_n & \I_n  \\
    \partial \dot u & C & \OO_n & \OO_n & 2\I_n  \\
    \partial \dot w & -2a\I_n & C & -2\I_n & -2a\I_n
  }
\end{equation}
where $C$ is the $n\times n$ matrix
  $$C =\frac{1}{n}\left[\begin{array}{cc} \mathbb 1_{\scriptscriptstyle N,N}& -\mathbb 1_{\scriptscriptstyle N,N}\\-\mathbb 1_{\scriptscriptstyle N,N}& \mathbb 1_{ \scriptscriptstyle N,N}   \end{array}\right ] .$$ 
To block-diagonalize the Jacobian matrix $J$, we mirror the construction from Section \ref{SectionNonDegenerateRingState}, relying on vectors in the kernel of the matrix $C.$ In the degenerate case,
we construct a basis for the kernel of $C$ in $\R^{2N}$ that reflects the symmetry of $C$, associated with
  the splitting of agents into polar opposite groups.

We begin with a  basis for  $ \ii_{\scriptscriptstyle N} ^ {\scriptscriptstyle \bot}= \{v\in \mathbb R^N: v^{\scriptscriptstyle T} \ii_{\scriptscriptstyle N} = 0\}$.
Arrange these basis vectors as columns of a matrix; denote the resulting $N\times (N-1)$-matrix by $V$.
Set
\begin{equation}\label{eqnMatrixV}
\mathbb V = \left [
\begin{array}{ccc}
 V &   \OO_{\scriptscriptstyle N,N-1} & \ii_{\scriptscriptstyle N} \\
    \OO_{\scriptscriptstyle N,N-1} & V& \ii_{\scriptscriptstyle N}
\end{array}
\right ]
\end{equation}

Notice that $\mathbb V$ is an $n\times (n-1)$-matrix satisfying
\begin{equation}\label{eq_CVequalsZero}
X^{\scriptscriptstyle T} \mathbb V=0 \; \mbox{ and } \;  C  \mathbb V = \OO_{n,n-1}, \; \mbox{ since } \ii_{\scriptscriptstyle N} ^{\scriptscriptstyle T} V= 0_{\scriptscriptstyle N-1}.
\end{equation}

\begin{remark}\label{RemarkOrthogonalBasis} The columns of the matrix $\mathbb V$ are orthogonal to the vector $X$ and together with $X$ they form a basis of $\mathbb R^{2N}$. Similarly, the matrix $[V \; \; \ii_N]$ is invertible; necessarily the last row of its inverse is $\frac{1}{N} \ii_{\scriptscriptstyle N}^{\scriptscriptstyle T}$. Denote by $T$ the $N-1$ by $N$ matrix whose rows are rows $1$ through $N-1$ of $[V \; \; \ii_{\scriptscriptstyle N}]^{-1}.$ Then
\beq
\label{eqnMatrixT}
[\mathbb V \;\;  X]^{-1}= \left [
\begin{array}{c} \mathbb T \\ \frac{1}{n}X^{\scriptscriptstyle T} \end{array} \right ], \; \mbox{ where } \;
\mathbb{T}= \left [
\begin{array}{cc} T & \OO_{\scriptscriptstyle N-1,N}\\ \OO_{\scriptscriptstyle N-1,N} & T \\ & \\\frac{1}{n}\ii_{\scriptscriptstyle N}^{\scriptscriptstyle T} & \frac{1}{n}   \ii_{\scriptscriptstyle N} ^{\scriptscriptstyle T}
\end{array} \right ].
\eeq
\end{remark}

The first step towards the diagonalization of  $J$ repeats the construction from Lemma \ref{LemmaNonDEigenspace}.
Let $ B_1=I_4\otimes\V, B_2= I_4\otimes X.$ Then, as in calculations (\ref{EqJacobianRestrictionB1}) and (\ref{EqJacobianRestrictionB2}), using $CX=X$, we get
\bdima
J B_1=\left[
\begin{array}{cccc}
 0 & 0 & 1 & 0 \\
 0 & 0 & 0 & 1 \\
 0 & 0 & 0 & 2 \\
 -2 a & 0 & -2 & -2 a \\
\end{array}
\right ] \otimes \V \; \mbox{ and }
J B_2= \left[ \begin{array}{cccc}
 0 & 0 & 1 & 0 \\
 0 & 0 & 0 & 1 \\
 1 & 0 & 0 & 2 \\
 -2 a & 1 & -2 & -2 a \\
\end{array}
\right]\otimes X
\edima
The  $4\times 4$ matrices, denoted $J_1$ (left) and $J_2$ (right)  can be further block diagonalized to separate the eigenspaces $\lambda =0$ and $\lambda =\pm i$ from those with $\mbox{Re} \lambda<0$:
\begin{flushleft}
\begin{equation}\label{NewBlocks}
\mbox{If } \; P_1=\left[
\begin{array}{cccc}
 0 & 1 & 0 & 0 \\
 1 & 0 & \frac{1}{2} & 0 \\
 0 & 0 & 1 & 0 \\
 0 & 0 & 0 & 1 \\
\end{array}
\right], \mbox{ then } P_1^{-1}\; J_1 \; P_1= \left[
\begin{array}{c|ccc}
 0 & 0 & 0 & 0 \\
 \hline
 0 & 0 & 1 & 0 \\
 0 & 0 & 0 & 2 \\
 0 & -2 a & -2 & -2 a \\
\end{array}
\right],
\end{equation}
\end{flushleft}
\begin{flushleft}
\begin{displaymath}
\mbox{If }  P_2=\left [
\begin{array}{cccc}
 -a & -1 & 0 & -1 \\
 1-a^2 & -a & 1 & 0 \\
 1 & 0 & 1 & 0 \\
 0 & 1 & 0 & 1 \\
\end{array}
\right], \mbox{ then } P_2^{-1}\; J_2 \; P_2= \left[
\begin{array}{cc|cc}
 -a & 1 & 0 & 0 \\
 a^2-1 & -a & 0 & 0 \\
 \hline
 0 & 0 & 0 & 1 \\
 0 & 0 & -1 & 0 \\
\end{array}
\right].
\end{displaymath}
\end{flushleft}
The $2\times 2$ upper block in the diagonalization  of $P_2$ will appear later in the equation for $\dot \gamma$, so we  denote it by $J_\gamma$. Denote by $B_3$ the $3\times 3$ block in (\ref{NewBlocks}). We get:
\beq
B_3=\left[
\begin{array}{ccc}
 0 & 1 & 0 \\
 0 & 0 & 2 \\
 -2 a & -2 & -2 a \\
\end{array}
\right], \; \; B_3^{-1}=\left[
\begin{array}{ccc}
 -\frac{1}{a} & -\frac{1}{2} & -\frac{1}{2 a} \\
 1 & 0 & 0 \\
 0 & \frac{1}{2} & 0 \\
\end{array}
\right], \mbox{ and }
\eeq
\beq \label{eq_Jginv}
 J_\gamma =\left[
\begin{array}{cc}
 -a & 1 \\
 a^2-1 & -a \\
\end{array}
\right], \; \; J_\gamma ^{-1}=\left[
\begin{array}{cc}
 -a & -1 \\
 1-a^2 & -a \\
\end{array}
\right].
\eeq

We introduce the change of coordinates  $Z,\be,\bd,\ba \in \mathbb R^{n-1}$, and $\gamma_1,\gamma_2,\delta_1,\delta_2\in \mathbb R$ to match the $1+3$ and $2+2$ sub-block structures of  (\ref{NewBlocks}):
  $$ col(\bar a,\bar b, \bar u,\bar w) = P \; col \left( Z,(\be,\bd,\ba),(\gamma_1,\gamma_2),(\delta_1,\delta_2) \right ),$$
where the  transformation matrix $P$ is $[ P_1 \otimes \V \; \; \; P_2 \otimes X].$ Equivalently:
\begin{equation}\label{eqTransformationMatrix}
  P = \kbordermatrix{
    & Z & \be & \bd & \ba & \gamma_1 & \gamma_2 & \delta_1 & \delta_2  \\
     \bar a & \OO_{n,n-1} &  \mathbb V & \OO & \OO  & -aX & -X & \oo_n & - X   \\
    \bar b & \mathbb V  & \OO & \frac{1}{2}\mathbb V &  \OO & (1-a^2)X & -a X & X & \oo_n  \\
     \bar u & \OO & \OO & \mathbb V & \OO  & X & \oo_n & X & \oo_n\\
    \bar w & \OO & \OO & \OO & \mathbb V & \oo_n & X & \oo_n & X
  },
\end{equation}
where the zero blocks of $P$ with unstated dimensions are $n \times (n-1),$ just like $  \mathbb V.$

\begin{proposition}\label{PropZcoord}
 In the coordinates introduced by the matrix  $P$, the  system of differential equations (\ref{eqnNewSystem_abpq_coord}), also given in (\ref{compactabuw}), becomes:
\begin{equation}\label{eqnNewCoordsDegFull}
 \begin{array}{cl}
\dot Z & = \;  \OO_{n-1, n-1} \;  Z\; + \;  (-\mathrm a) \; \frac{1}{2}\; \mathbb  T\; \mathcal{U} \\[.1cm]
\left ( \begin{array}{l}
 \dot \be \\ \dot \bd \\  \dot\ba \end{array} \right) &=
\left [ \begin{array}{ccc} 0 & I_{n-1} & 0\\ 0&0& 2I_{n-1} \\-2 \mathrm  a I_{n-1} & -2I_{n-1} &-2\mathrm  a I_{n-1} \end{array} \right ]\left ( \begin{array}{l} \be \\  \bd \\ \ba \end{array} \right) +
\left ( \begin{array} {l}0\\ \mathrm a \; \mathbb T\mathcal{U}\\  \mathrm a \; \mathbb T\mathcal{W}\end{array} \right ) \\[.1cm]
\left ( \begin{array}{l}\dot \gamma_1\\ \dot\gamma_2\\ \end{array} \right ) & =
\left [ \begin{array}{cr}-a& 1\\a^2-1& -a \end{array} \right ]
\left ( \begin{array}{l} \gamma_1\\ \gamma_2\\ \end{array} \right )
+
\left ( \begin{array}{l}-\frac{1}{n}X^{\scriptscriptstyle T}  \mathcal W\\[.1cm]
\frac{1}{n}X^{\scriptscriptstyle T} (\mathcal U + \mathrm a \; \mathcal W) \\
\end{array} \right )
\end{array}
\end{equation}
%
%
\begin{equation}\label{eqnNewCoordsJustDelta}
 \left ( \begin{array}{l}
  \dot\delta_1 \\  \dot\delta_2 \end{array} \right ) =
\left[\begin{array}{cc}0 & 1 \\ -1 & 0\end{array} \right]
\left ( \begin{array}{l}
\delta_1 \\  \delta_2 \end{array} \right )%
+
\left ( \begin{array}{c}
\frac{1}{n}X^{\scriptscriptstyle T} (\mathrm a \; \mathcal U + \mathcal W)  \\
- \frac{1}{n}X^{\scriptscriptstyle T} \mathcal U
\end{array} \right ), \; \mbox{ where }
\end{equation}
the nonlinear functions $\mathcal U =\mathcal U(\mathbb c, \mathbb s)$  and $\mathcal W =\mathcal W (\mathbb c,\mathbb s)$  were defined in (\ref{eqn_abuw_sys_nonlinearities}) as
$\displaystyle
\mathcal U =  -\mathbb{s}\odot \left ((\mathbb{c}+\mathbb 1_n)^2 +\mathbb{s}^2 - \mathbb 1_n \right ) $ and
$\mathcal{W} =
(-1)(\mathbb{c} +\mathbb 1_n ) \odot \left ( (\mathbb{c}+ \mathbb 1_n)^2 +\mathbb{s}^2 -\mathbb 1_n \right )+2\mathbb{c}, $ with
\begin{equation}\label{eqn_cs_in_newcoordinates}
\begin{array}{l}
\mathbb c = \mathbb V \be+\mathbb V \ba -\mathrm a \; \gamma_1 X, \\
\mathbb s  = -\mathbb V Z+\frac{1}{2}\mathbb V \bd + (\mathrm a^2 \gamma_1+\mathrm a\; \gamma_2 )X.
\end{array}
\end{equation}
%
\end{proposition}

\begin{proof}
We use the connection between $\mathbb V $ and $\mathbb T$ from (\ref{eqnMatrixT}) to find the inverse of the transformation $P=[ P_1 \otimes \V \; \; \; P_2 \otimes X]$. The matrix $P^{-1}$ consists of the blocks $P_1^{-1} \otimes \mathbb T$ and $ P_2^{-1}\otimes \frac{1}{n} X^{\scriptscriptstyle T}$ :

\begin{equation}\label{eqnPInverse}
\left[ \begin{array}{l}Z\\ \be\\ \bd \\ \ba \\ \gamma_1 \\ \gamma_2 \\ \delta_1 \\ \delta 2 \end{array} \right] =
P^{-1} \left [\begin{array}{l} \bar a \\ \bar b \\ \bar u  \\ \bar w \end{array} \right],  \;
 P^{-1} =
 \left[ \begin{array}{l}
 \kbordermatrix{
 & \bar a & & \bar b & & \bar u & \bar w \\
  Z & 0 & & 1 & & -\frac{1}{2} & 0 \\
\be & 1 && 0 & & 0 & 0 \\
 \bd & 0 && 0 & & 1 & 0  \\
 \ba& 0 &&  0 & & 0 & 1} \otimes \mathbb T
 \\
  \kbordermatrix{
&  &  &  &  \\
\gamma_1& -\frac{1}{a} & 0 & 0  & -\frac{1}{a} \\
 \gamma_2 & 1 & -\frac{1}{a} & \frac{1}{a} & 1 \\
\delta_1 &  \frac{1}{a} &0 & 1 & \frac{1}{a} \\
\delta_2 &  - 1 & \frac{1}{a} & -\frac{1}{a} & 0
} \otimes \frac{1}{n}X^{\scriptscriptstyle T} \end{array}
 \right ]
  \end{equation}
 The linear part of the system (\ref{compactabuw}) has matrix $P^{-1}JP$ , with blocks $\OO, B_3\otimes I_{n-1}, J_\gamma$ associated with
 the coordinates $Z, col(\be, \bd, \ba), (\gamma_1, \gamma_2)$, consistent with (\ref{eqnNewCoordsDegFull}), and the block $\displaystyle \left [ \begin{array}{ll} 0 &1\\-1&0 \end{array}\right ]$ associated with $\delta _1, \delta_2,$ consistent with (\ref{eqnNewCoordsJustDelta}).

 The nonlinearities of the system, obtained by applying $P^{-1}$ to the $col(0, 0, \mathrm a\; \mathcal U,  \mathrm a \; \mathcal W)$ system in (\ref{compactabuw}), are listed in the table below: the top row is new  coordinate, the bottom row is the corresponding nonlinearity:
 \bdima
 \begin{array}{c|ccc|cc | cc}
 \dot Z& \dot \be & \dot \bd & \dot \ba & \dot \gamma_1& \dot \gamma_2& \dot  \delta_1 & \dot \delta_2 \\
&&&& &&& \\
\frac{- \mathrm a}{2}\mathbb  T\mathcal{U} & 0 & \mathrm a   \mathbb T\mathcal{U} &  \mathrm a\mathbb T\mathcal{W}&
-\frac{1}{n}X^{\scriptscriptstyle T} \mathcal W&
\frac{1}{n}X^{\scriptscriptstyle T}(\mathcal U + \mathrm a  \mathcal W) &
\frac{1}{n}X^{\scriptscriptstyle T} (\mathrm a  \mathcal U +\mathcal W) &
-\frac{1}{n}X^{\scriptscriptstyle T} \mathcal U.
 \end{array}
\edima
This proves the validity of (\ref{eqnNewCoordsDegFull}) and (\ref{eqnNewCoordsJustDelta}).

Finally, we calculate $\mathbb c=\bar a+\bar w $ and $\mathbb s=\bar u-\bar b$. Based on (\ref{eqTransformationMatrix})
the  change of coordinates defines %
\bdima 
\begin{array}{l}
\bar a  =\mathbb V \be - \mathrm a \; \gamma_1 X-\gamma_2 X-\delta_2 X  \\
\bar b  = \mathbb V Z+\frac{1}{2} \mathbb V \bd +(1-\mathrm a^2)\gamma_1 X - \mathrm a\; \gamma_2 X + \delta_1 X\\
\bar u  = \mathbb V \bd + \gamma_1 X +\delta_1 X \\
\bar w  = \mathbb V \ba +\gamma_2 X + \delta_2 X, \; \mbox{ thus }
\end{array}
\edima
$ \mathbb c =\bar a+\bar w = \mathbb V \be+\mathbb V \ba -\mathrm a\gamma_1 X $ and
$\mathbb s =\bar u-\bar b = -\mathbb V Z+\frac{1}{2}\mathbb V \bd +(\mathrm a ^2 \gamma_1+ \mathrm a \gamma_2 )X .$%

\end{proof}
The characteristic polynomials of the blocks $ B_3\otimes I_{n-1}$ and $J_\gamma$, given by
 $(\lambda^3+2a \lambda^2+4\lambda +4a)^{n-1}$, and   $\lambda^2 +2a \lambda +1$, respectively,
 have all roots in the negative half plane.
  Thus, the variables $\be$, $\bd$, $\ba$, $\gamma_1$, and $\gamma_2$ capture the $3n-1$ stable directions of the system consisting of  (\ref{eqnNewCoordsDegFull}) and (\ref{eqnNewCoordsJustDelta}). The variables $Z$, $\delta_1$, $\delta_2$ represent the neutral directions of the system.

The nonlinear functions of  (\ref{eqnNewCoordsDegFull}) depend only on $\mathbb c$ and $\mathbb s$, which are independent of $\delta_1, \delta_2, $ thus:

\begin{lemma} The vector field governing the variables $\{Z,\be,\bd,\ba,\gamma_1,\gamma_2\}$ in (\ref{eqnNewCoordsDegFull}) is independent of $\delta_1$ and $\delta_2$; that is, it decouples from $\delta_1$ and $\delta_2$.
\end{lemma}

For particles in a  rotating state $R_x+iR_y+a e^{i( \theta_{0k}+At)}$ near the degenerate state with angles $0$ and $\pi$, the components $(\delta _1,\delta _2)$ are equal to the center of mass $(R_x, R_y)$. Thus, decoupling $\delta _1$ and $\delta _2$ suppresses the location of the center of rotation of the limit cycle (once we prove that said limit cycle exists).

We study the stability of the reduced $\{Z,\be,\bd,\ba,\gamma_1,\gamma_2\}$ system first, and then show that  incorporating the variables $\delta_1$ and $\delta_2$  maintains the stability of the origin.
We think of the $(\delta_1, \delta_2)$-components as being solutions to a linear system  driven by the output of the decoupled $(Z,\be,\bd,\ba,\gamma_1,\gamma_2)$-system.

\subsection{Approximating the center manifold for the reduced system.}\label{SubsectionApproxCenterManifold}

In this  section we approximate the center manifold map $h=h(Z)$ of the system (\ref{eqnNewCoordsDegFull}) near the origin using the technique from Section \ref{SectionNotTaylor}.
For brevity, we denote by $Y$ the $(3n-1)$-dimensional vector whose components are those of $\be, \bd,\ba$, and $\gamma $,  of the reduced $(Z, Y)$ system.
All our arguments apply while trajectories remain in a small neighborhood of the origin in $\mathbb R^{4n-2}$, therefore assuming that
 $Z$ itself is in a small enough neighborhood of the origin in $\mathbb R^{n-1}$.

 The change of coordinates (\ref{eqnPInverse}) defines $Z=\mathbb T (\bar b -\frac12 \bar u)$. In keeping with the three-block vertical structure of the matrix $\mathbb T$ (see (\ref{eqnMatrixT})),
 and the two-block structure of $\mathbb V$, we distinguish between upper and lower components of our variables:

\begin{definition}\label{DefEnergyFunction}
Let $Z=(z_1,\ldots,z_{n-1})^T$ be in a neighborhood of $0\in  \mathbb R^{n-1}$.
\begin{enumerate}

\item  Denote by $Z_U$ (upper) and $Z_L$  (lower) the $(N-1)$-dimensional vectors $Z_U=(z_1, \dots z_{N-1})^{\scriptscriptstyle T}$,
$Z_L=(z_N, \dots z_{2N-2})^{\scriptscriptstyle T}$ to  get
$$ Z = \left ( \begin{array}{c} Z_U \\ Z_L \\  z_{n-1} \end{array}\right).$$

\item Denote by $ |Z| = (\sum_{i=1}^{n-1} |z_i|^2) ^{1/2}$, the
       Euclidean norm of $Z$. Denote by $|Z|_r$ the {\it reduced norm}, actually a seminorm, $|Z|_r^2 = |Z_U|^2+|Z_L|^2=|Z|^2-z_{n-1}^2.$

\item
 Define the variables $\Theta= \Theta (Z) = \{\theta_k\}_{k=1}^n$ implicitly as
$\sin(\theta_i) = (\mathbb V Z)_i$. Denote by $\Theta_U=(\theta_1, \dots \theta_N)^{\scriptscriptstyle T}$ and $\Theta_L=(\theta_{N+1}, \dots \theta_{2N})^{\scriptscriptstyle T}$.

 Using the earlier function shorthand, $\displaystyle
\sin(\Theta) = (\sin(\theta_1),\ldots, \sin(\theta_n))^{\scriptscriptstyle T}, $ and $\displaystyle  \cos(\Theta) = (\cos(\theta_1),\ldots, \cos(\theta_n))^{\scriptscriptstyle T}, $ we obtain
 \beq \label{VandTheta}
  \mathbb V Z= \sin(\Theta) = \left(\begin{array}{c} \sin(\Theta_U)  \\ \sin(\Theta_L) \end{array}\right)
 =\left(\begin{array}{c} V\;  Z_U+ z_{n-1}\mathbb 1_{\scriptscriptstyle N} \\ V\; Z_L+ z_{n-1}\mathbb 1_{\scriptscriptstyle N}\end{array}\right).
 \eeq

\item
 In view of Remark \ref{RemarkOrthogonalBasis}, the vector $X$  and  the columns of the matrix $\mathbb V$ form a basis  in $\mathbb R^{n}$.  Define the scalar $\mathcal E = \mathcal E(Z)$ and vector $\alpha=(\alpha_i(Z))_{i=1,\ldots,n-1}$ as the coefficients of $\cos(\Theta)-\ii_n$ in this basis.
  Analogous to the splitting of $Z$  we denote by $\alpha_{\scriptscriptstyle U}, \alpha _{\scriptscriptstyle L}$ the components of $\alpha $ from $1$ to $N-1$, and from $N$ to $2N-2$, respectively, to represent
\begin{equation} \label{eqnEnergyFunctionDef}
 \cos (\Theta) - \ii_{n}= \mathbb V  \alpha + \mathcal E X =
  \left(\begin{array}{c} V\;  \alpha_{\scriptscriptstyle U} \\ V\;  \alpha_{\scriptscriptstyle L} \end{array}\right) + \alpha _{n-1} \mathbb 1_n + \mathcal E X.
\end{equation}
\end{enumerate}
\end{definition}


\begin{proposition} \label{remarkPropertiesOfEnergy} 
The scalars $ z_{n-1}, \alpha _{n-1}$, and $ \mathcal E(Z)$
satisfy:

\begin{equation}\label{eq_CosMean}
\begin{array}{ll}
z_{n-1} =& \frac{1}{N}\sum \limits_{k=1}^N \sin \theta _k = \frac{1}{N} \sum \limits_{l=1+N}^{2N} \sin \theta_l , \\
\alpha _{n-1}= &\frac{1}{n} \sum \limits_{k=1}^n \cos \theta _k \; -1, \\
\mathcal E(Z)=& \frac{1}{n} X^T (\cos ( \Theta) - \ii_{n})
= \frac{1}{n}(\sum \limits_{k=1}^N \cos \theta _k -  \sum \limits_{l=1+N}^{2N} \cos \theta_l ).
\end{array}
\end{equation}
\end{proposition}
\begin{proof}
We use orthogonality of $X$ to the columns of $\mathbb V$ (which include  $\mathbb 1_n$), and the orthogonality of $\mathbb 1_N$  to the  columns of $V$.
The $z_{n-1}$  equation  follows from applying dot product with $\mathbb 1_N$  to   $\displaystyle \sin(\Theta_U) = VZ_U+z_{n-1}\ii_N$ and  $\displaystyle \sin(\Theta_L) = VZ_L+z_{n-1}\ii_N.$
For $\alpha _{n-1}$, use  $\displaystyle \cos ( \Theta) - \ii_{n}= \left[\begin{array}{c} V\;  \alpha_U \\ V\;  \alpha_L \end{array}\right] + \alpha _{n-1} \mathbb 1_n + \mathcal E X$, and the dot product with $\mathbb 1_n$. To find $\mathcal E(Z), $ take dot product with $X$ in (\ref{eqnEnergyFunctionDef}).
\end{proof}


\begin{remark}\label{RemarkGeometricMeaning}
To illustrate the geometrical meaning of $  \Theta$ and $z_{n-1}$,   we consider a perturbation given by $( e^{i\omega _k}\frac{1}{A} e^{itA})_{k=1}^N$ and $  (e^{i\omega _k}\frac{1}{A} e^{i\pi}e^{iAt})_{k=1+N}^n$ of the state
$(\frac{1}{A} e^{iAt})_{k=1}^N$, $(\frac{1}{A} e^{i\pi}e^{iAt})_{k=N+1}^n$, that is itself a rotating state. Let $\bar \omega =(\omega_1, \dots \omega _n)^{\scriptscriptstyle T}.$
We will show in the calculation below that $\Theta =\bar \omega$, meaning that the variables $\{\theta_i\}$ capture the perturbation angles relative to those of the degenerate rotating state, as illustrated in Figure (\ref{figureGeometryTheta}).

From (\ref{DefineChangeCoordinates_ab}), in the  $(\bar a, \bar b, \bar u, \bar w)$-coordinate system  $\bar a = \cos(\bar \omega)-\mathbb 1$, $\bar b=\sin(\bar \omega)$, and $\bar u=\bar w=0$. Using the transformation matrix $P^{-1}$ given in (\ref{eqnPInverse}), we obtain $Z = \mathbb T \sin(\bar \omega)$. Since the last row of the matrix $\mathbb T$ is $\frac{1}{n} \ii_n^{\scriptscriptstyle T}$, we have that $z_{n-1}= \frac{1}{n}\ii_n ^{\scriptscriptstyle T} \sin(\bar \omega).$

The perturbed configuration under consideration is a rotating state, thus the polar angles $\{\omega_k\}$ satisfy
$\sum_{k=1}^N\left ( e^{i\omega_k} + e^{i(\pi+\omega_{k+N})} \right )= 0$. We conclude that
$$\frac{1}{N}\sum_{k=1}^N \sin \omega_k = \frac{1}{N}\sum_{k=1}^N  \sin \omega_{k+N}= z_{n-1}, \; \mbox{ and } X^{\scriptscriptstyle T} \sin(\bar \omega)=0.$$
Using block notation,
$ \displaystyle \left [\begin{array}{l} Z\\0\end{array}\right ]=
\left [\begin{array}{l} \mathbb T \sin(\bar \omega ) \\ \frac{1}{n}X^{\scriptscriptstyle T} \sin(\bar \omega) \end{array}\right ]=
\left [\begin{array}{l} \mathbb T  \\ \frac{1}{n}X^{\scriptscriptstyle T}  \end{array}\right ]\sin(\bar \omega ).$
Multiplying by the inverse of $\displaystyle \left [\begin{array}{l} \mathbb T  \\ \frac{1}{n}X^{\scriptscriptstyle T}  \end{array}\right ]$, which is $[\mathbb V, X]$ , per (\ref{eqnMatrixT}), we get
$ \sin(\bar \omega)= \mathbb V Z+0.$

From part (3) of Definition \ref{DefEnergyFunction},  $\sin(\Theta) = \mathbb V Z =\sin(\bar \omega)$, implying $\Theta =\bar  \omega$.

 Moreover, if the perturbed state is degenerate, with agents on a slanted diameter of polar angle $\omega_1$, then $Z_U= Z_L=0, z_{n-1}=\sin\omega_1$ and $\alpha_{n-1}=\cos(\omega_1)-1.$
\end{remark}

\begin{figure}[h!]
\centering
\includegraphics[width=.9\textwidth]{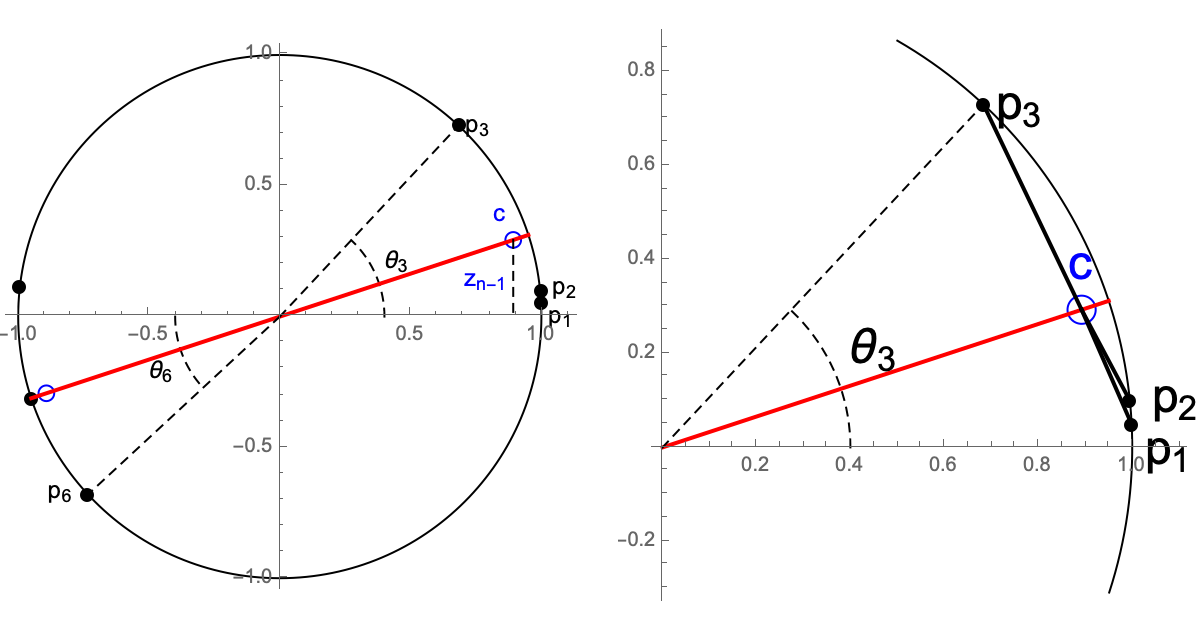}%
\caption{\label{Figure6Particles} A perturbation of the degenerate rotating state of six particles into a non-degenerate rotating state. The particles split into two groups: one appears only in the left image and the other  appears in both. The center of mass of each group is depicted as a hollow circle on the slanted diameter (red solid) line. These centers are opposites of each other, and have $y$-coordinates equal to $\pm z_{n-1}$.
The variables $\theta$'s represent the perturbation of the polar angles relative to the $x$-axis. The quantity $\mcd_U = |p_3-c|^2+|p_2-c|^2+|p_1-c|^2$ captures the dispersion of the right group and will be used as a Lyapunov function in the proof of stability.}
\label{figureGeometryTheta}
\end{figure}

 %

\begin{lemma}
\label{LemmaBigOEstimates} For  $Z$ in a small neighborhood of $ \mathbb R^{n-1}$, the vector-valued   $\Theta(Z)$, $\alpha (Z) $,  and scalar  function $\mce(Z)$ are well defined and depend smoothly on $Z.$ Moreover
\begin{enumerate}
\item $\sin(\Theta) = \mathcal O(|Z|)$,
\item $\sin(\Theta)-z_{n-1}\ii_n= \mathcal O(|Z|_r), \; \mbox{ and } \; \sin^2(\Theta)-z^2_{n-1}\ii_n = \mathcal O(|Z|_r|Z|)$,
\item $\mce(Z) = \mathcal O(|Z|_r |Z|)$,
\item  $\cos(\Theta) - (1+\alpha_{n-1})\ii_n=\mathcal O(|Z|_r |Z|)$.
\end{enumerate}
\end{lemma}
\begin{proof} From $\sin(\Theta) = \mathbb V Z$, we get  $|\sin(\Theta)| = \mathcal O(|Z|)$. Additionally, $\sin(\Theta_U) = VZ_U+z_{n-1}\ii_N$  and $\sin(\Theta_L) = VZ_L+z_{n-1}\ii_N$. Hence, $\sin(\Theta)-z_{n-1}\ii_n= \mathcal O(|Z|_r)$, which, in particular, implies that  $\sin^2(\Theta)-z^2_{n-1}\ii_n = \mathcal O(|Z|_r|Z|)$.

Using the triangle inequality, we get that  $$|\sin(\theta_i)-\sin(\theta_j)| \leq |\sin(\theta_i)-z_{n-1}|+|\sin(\theta_j)-z_{n-1}|\leq |VZ_U| + |VZ_L| = \mathcal O(|Z|_r).$$
Now we need the following fact: $\frac{d}{dt}\sqrt{1-t^2} = \mathcal O(t)$. Applying the mean value theorem to the function $\sqrt{1-t^2}$, we obtain that
\begin{multline*}|\cos(\theta_i)-\cos(\theta_j)| = |\sqrt{1-\sin^2(\theta_i)} - \sqrt{1-\sin^2(\theta_j)}| \\ =  \mathcal O(|\sin(\Theta)|) |\sin(\theta_i)-\sin(\theta_j)|  = \mathcal O(|Z|_r |Z|).
\end{multline*}
It follows that $$\cos(\theta_k) - \alpha_{n-1} = \frac{1}{n}\sum_{m=1}^n(\cos(\theta_k)-\cos(\theta_m)) = \mathcal O(|Z|_r |Z|).$$
Similarly,
$\mathcal E(Z) = \frac{1}{n}\sum_{m=1}^{N}(\cos(\theta_m)-\cos(\theta_{N+m})) = \mathcal O(|Z|_r |Z|). $
\end{proof}

%
%

Our subsequent calculations frequently rely on the relationship between $\mathbb V$ and $\mathbb T$, introduced in (\ref{eqnMatrixT}), particularly
\beq \label{EqSimplifyTV}
\mathbb T \;  \mathbb V= I_{n-1} \; \; \mbox{ and }  \mathbb T\; X =0.
\eeq

Next, we  describe the equilibrium points of (\ref{eqnNewCoordsDegFull}) near the origin, previously shown to correspond to rotations about the origin in the swarm's physical coordinates.

\begin{lemma}\label{lemmaEquilibriumPointsReducedSystem} Suppose $Z_0 \in \mathbb R ^{n-1}$ is small enough with $ \mathcal E (Z_0)=0$. Then the point $col( Z_0,\alpha (Z_0), \oo_{n-1}, \oo_{n-1}, 0,0)$ is an equilibrium point for the reduced  flow (\ref{eqnNewCoordsDegFull}).\end{lemma}
\begin{proof}

We use the notations of Proposition \ref{PropZcoord}, where (\ref{eqnNewCoordsDegFull}) was introduced.

 Let $P_0=col( Z_0,\alpha (Z_0), \oo_{n-1}, \oo_{n-1}, 0,0)$.
 We  calculate the nonlinear components of the vector field (\ref{eqnNewCoordsDegFull}) at $P_0$  first. Evaluate $\mathbb c$ and $\mathbb s$ using
 $\be = \alpha(Z_0)$, $\bd =\ba = \oo_{n-1}$, $\gamma_1=\gamma_2= 0$ in (\ref{eqn_cs_in_newcoordinates}), then convert to the $\Theta$-notation from (\ref{VandTheta}) and (\ref{eqnEnergyFunctionDef}):
$$\mathbb c = \mathbb V \alpha(Z_0) = \mathbb V \alpha(Z_0) + \mathcal E(Z_0) X= \cos(\Theta) - \ii_n\mbox{ and }\mathbb s = -\mathbb V Z_0 = -\sin(\Theta).$$

Hence, $(\mathbb{c}+\mathbb 1_n)^2 +\mathbb{s}^2 -\mathbb 1_n= (\cos \Theta )^2+(-\sin \Theta )^2-\mathbb 1_n=0.$

It follows  that  $\mathcal U = 0$ and $\mathcal W = 2\mathbb c$. 
Simplify $\mathbb T\mathcal{W}$ and $X^T\mathcal{W}$  using (\ref{EqSimplifyTV}):
\bdima \begin{array} {l}
 \mathbb T\mathcal{W}= \mathbb T(2\mathbb{c})= 2\mathbb T\;  \mathbb V  \alpha (Z_0)=2\alpha (Z_0)\; \mbox{ and} \\
  X^T\mathcal{W}=X^T(2\mathbb{c})= 2X^T \mathbb V  \alpha (Z_0)=0. \end{array}
 \edima
At $P_0$, the linear part of the vector field (\ref{eqnNewCoordsDegFull}) is zero for the $\dot Z$ and $\dot \gamma $ blocks; for the $\dot \be, \dot \bd, \dot \ba$ block  it equals
\bdima
  \left[
\begin{array}{ccc}
 0 & I_{n-1} & 0 \\ 0 & 0 & 2I_{n-1}  \\ -2 \mathrm a I_{n-1}  & -2I_{n-1}  & -2 \mathrm a I_{n-1}  \\\end{array}
 \right]
 \left[ \begin{array}{l}\alpha (Z_0)\\ \oo_{n-1}\\ \oo_{n-1} \end{array} \right]=
 \left[ \begin{array}{l}  \oo_{n-1}\\ \oo_{n-1} \\-2\mathrm a \; \alpha (Z_0) \end{array} \right].
\edima

Both  the linear and nonlinear  parts of $\dot Z, \dot \be, \dot \bd $ and $  \dot \gamma$ are zero, and
$$ \dot \ba = -2\mathrm a \be + \mathrm a\; \mathbb T \mathcal W= -2\mathrm a \;  \alpha + \mathrm a\; ( 2\alpha) =0, $$
 establishing that $P$ is an equilibrium point for the reduced system  (\ref{eqnNewCoordsDegFull}).
\end{proof}

\begin{remark}\label{RemarkZeroEnergyOnCenterManifold} (1) Every center manifold $
\mathcal M$  contains all nearby equilibrium points of the flow. In view of Lemma \ref{lemmaEquilibriumPointsReducedSystem}, the  points  $(Z,\alpha (Z), \oo_{n-1},\oo_{n-1},  0,0)$ with $\mathcal E(Z) = 0$ belong to the center manifold.

(2) Additionally, since $\mathcal M$ can be locally represented by the graph of a function $h: \mathbb R^{n-1}\rightarrow \mathbb R^{3n-1}$, if  $Y = h(Z)$ and $\mathcal E(Z) = 0$, then $Y = (\alpha (Z), \oo_{n-1},\oo_{n-1},0,0)$.

(3)  The set $\{(Z,Y)\in \mathcal M : \mathcal E(Z) = 0\}$ is flow-invariant (it consists of fixed points).
Given a point on the center manifold with initial value  $ \mathcal E (Z)>0,$  we have $ \mathcal E (Z(t)) >0$  for all $t $ for which $ \mathcal E (Z(t)) $ is defined.  A similar result holds for  $ \mathcal E<0$.
\end{remark}

\subsection*{Shadowing the orbit of the contraction operators defining the center manifold}

To approximate the center manifold $\mathcal M$ of the reduced system (\ref{eqnNewCoordsDegFull}), we will use the
technique introduced  in Section \ref{SectionANewApproach} to build up to the result in Theorem \ref{ThmApproximationTheorem}.
Using the notations from (\ref{standardcentersetup}), the stable and center directions for (\ref{eqnNewCoordsDegFull}) are
 $x_s=col (\be, \bd,  \ba, \gamma_1, \gamma_2)$, and   $x_c=Z$, with $B=0$ and $\displaystyle g= (-\mathrm a ) \frac12  \mathbb T \; \mathcal U. $

 We seek an approximation of the  map $h$ whose graph $x_s=h(Z)$ is the center manifold.
 Technically $(Z, h(Z))$ represents the limit of the iterates of the operator $\Gamma_Z$, but we shadow the orbit with constant functions $(Z, w_{0s})$, as in (\ref{constantGammaG1}),  and evaluate the accuracy using (\ref{eqHerrorGeneral}).

Motivated by the location of the equilibrium points, where $x_s=(\alpha, \oo, \oo, 0,0), $
we proceed to shadow the orbit of $\Gamma _Z$ using points satisfying  $\be= \alpha (Z)$, $\bd=\oo_{n-1}$ and  $ \ba=\oo_{n-1},$ i.e. from the shadow set $\mathcal S$:
\bdima
\mathcal S=\{ col (Z, \be, \bd,  \ba, \gamma_1, \gamma_2)\; |\; \be=\alpha (Z), \bd=\ba=\oo_{n-1} \}.
\edima
Evaluating $\dot \be, \dot \bd, \dot \ba $  at points in the shadow set, gives $\dot \be=\oo_{n-1}$,
$\dot \bd= \mathrm a  \mathbb T\mathcal U, $ and $\dot \ba= -2  \mathrm a \alpha + \mathrm a \mathbb T\mathcal{W}.$
This shows the shadow set $\mathcal S$ is {\em not} invariant  under the flow of (\ref{eqnNewCoordsDegFull}). Applying the operators $\Gamma _Z$ to constant functions from $\mathcal S$  does {\em not } yield functions
with range in $\mathcal S$. Approximation schemes of the center manifold that use points $(Z, w_s(Z))$ from the shadow set lead to errors satisfying (per (\ref{eqHerrorGeneral})):
\beq
\label{shadow_error}
|w_s(Z)-h(Z)|\leq C
\max\{ |\mathbb T\mathcal U|, |-2 \alpha(Z)+\mathbb T\mathcal W|, \left |\left.\dot{\gamma} \right |_{\scriptscriptstyle (Z, w_s(Z))}\right |\}.
\eeq

\begin{remark}\label{ShadowVersions}
 For points $(Z, \alpha,\oo_{n-1}, \oo_{n-1}, \gamma_1, \gamma_2) $ in the shadow set $\mathcal S$ the following simplifications apply:
\beq \label{eq_cs_shadow}
\begin{array}{l}
\mathbb c= \cos \Theta -\ii_n -\mce X +\mathrm a \gamma_1   X\\
\mathbb s= -\sin \Theta + (\mathrm a^2 \gamma_1+\mathrm a\; \gamma_2 )X, \\
\mathbb T \mathbb c= \alpha, \;  \mbox{ and }  \; \frac{1}{n} X^{\scriptscriptstyle T} \mathbb c= -\mathrm a \gamma_1.
\end{array}
\eeq
They follow from first using $\be=\alpha (Z), \bd=\ba=\oo_{n-1} $ in the definition (\ref{eqn_cs_in_newcoordinates}) of $\mathbb c, \mathbb s$,  and then using the  definitions of $\Theta $,  $\alpha$, and  $ \mce$ from (\ref{VandTheta}) and (\ref{eqnEnergyFunctionDef}):
$$ \mathbb c = \mathbb V \alpha -\mathrm a \; \gamma_1 X= (\cos (\Theta) - \ii_{n}-\mathcal E X) -\mathrm a \; \gamma_1 X.$$
\end{remark}

We emphasize that the scope of the iterate-and-shadow is limited to  the $\gamma$
block, since
we keep $\be, \ba, \bd$ set as $\alpha (Z), \oo_{n-1}, \oo_{n-1}$. 
Recall that an iterative step started with $w_{s}^{(0)}$ produces $w_s=-A_{s}^{-1}f(w_{s}^{(0)}, x_c), $ with good initialization  values $w_{s}^{(0)}$ obtained from  nullclines, or based on equilibrium points.

\subsection*{First iteration} Let   $Z$ be a non-equilibrium point, thus  $\mce(Z) \neq 0.$
As a first guess, use $w_s^{(0)}=(\alpha (Z), \oo_{n-1},\oo_{n-1}, \gamma^{(0)} )$ with
   $\gamma_1^{(0)}=0, \gamma_2^{(0)}=0$, consistent with the $w_s$-components of the equilibrium points. The iteration produces $w_s$ whose $\gamma $ components are
 $\displaystyle \gamma = (-1) J_{\scriptscriptstyle \gamma}^{\scriptscriptstyle -1}
 \left [ \begin{array}{l}-\frac{1}{n}X^{\scriptscriptstyle T}  \mathcal W\\[.1cm]
\frac{1}{n}X^{\scriptscriptstyle T} (\mathcal U + \mathrm a  \mathcal W) \\
\end{array} \right ].$
According to (\ref{eq_Jginv}), $\displaystyle \gamma=
\left[
\begin{array}{cc}
 a & 1 \\
 a^2-1 & a \\
\end{array}
\right]\left [\begin{array}{l}-\frac{1}{n}X^{\scriptscriptstyle T}  \mathcal W\\[.1cm]
\frac{1}{n}X^{\scriptscriptstyle T} (\mathcal U + \mathrm a \mathcal W) \\
\end{array} \right ].$
Our next $\gamma^{(1)}$ is obtained by  removing the terms of size $\mce |Z|$  or smaller from $\gamma.$ This requires that we calculate $\mathcal U $ and $\mathcal W$ first. Using (\ref{eq_cs_shadow}), we get
\bdima
\begin{array}{l}
\mathbb c^{(0)} +\ii_n= \cos \Theta  -(\mce +\mathrm a \gamma_1 )  X = \cos \Theta  -\mce   X, \\
\mathbb s^{(0)}= -\sin \Theta + (\mathrm a^2 \gamma_1+\mathrm a\; \gamma_2 )X = -\sin \Theta , \\
(\mathbb{c}+\mathbb 1_n)^2 +\mathbb{s}^2 - \mathbb 1_n  =  -2 \mce \cos \Theta \odot  X + \mce ^2\ii_n\simeq  -2 \mce   \ii_n\odot X+0= -2 \mce   X\\
\mathcal U =  -\mathbb{s}\odot \left ((\mathbb{c}+\mathbb 1_n)^2 +\mathbb{s}^2 - \mathbb 1_n \right )\simeq 0\\
\mathcal{W} =(-1)(\mathbb{c} +\mathbb 1_n ) \odot \left ( (\mathbb{c}+ \mathbb 1_n)^2 +\mathbb{s}^2 -\mathbb 1_n \right )+2\mathbb{c}\simeq (-1)(0+\mathbb 1_n)\odot(- 2 \mce   X) + 2\mathbb c.
\end{array}
\edima
We get  $\displaystyle \frac{1}{n} X^{\scriptscriptstyle T} \mathcal{W}\simeq\frac{1}{n} X^{\scriptscriptstyle T}(  2 \mce   X) + \frac{1}{n} X^{\scriptscriptstyle T} (2 \mathbb c ) =2 \mce +0, $ per (\ref{eq_cs_shadow}) and
\bdima
\gamma \simeq \left[
\begin{array}{cc}
 a & 1 \\
 a^2-1 & a \\
\end{array}
\right] \left [ \begin{array}{l} -2 \mce \\ 2 \mathrm a \; \mce
\end{array} \right] =
\left [ \begin{array}{l} 0 \\ 2 \mce
\end{array} \right] \stackrel{\text{def}}{=} \gamma^{(1)} .
\edima
The first iterate-and-shadow step yields
$\displaystyle w_s^{(1)}=col (\alpha (Z), \oo_{n-1},\oo_{n-1}, \left [  \begin{array}{l} 0\\2\mce  \end{array} \right ]  )$.

\subsection*{Second iteration}

Based of the first iteration, define a new, more restricted shadowing set, $\mathcal S^{(2)}$. (The first set, $\mathcal S$, consists of points with $\be=\alpha(Z)$, $\ba=\oo_{n-1}$ , $ \bd=\oo_{n-1}$. ) We constrain  $\gamma $ to only differ from $\gamma ^{(1)}$ by small multiples of $\mce.$ %
\bdima
\mathcal S^{(2)} =\mathcal S\cap \{(Z, w_s) \; | \;  \gamma = \left [ \begin{array}{c} 0 \\ 2 \mce \end{array}\right]+
\mce  \left [ \begin{array}{cc} -\mathrm a &0\\ \mathrm a^2 & \mathrm a \end{array}\right] ^{-1}
 \left [ \begin{array}{l}  x \\  y \end{array}\right],\mbox{ for small } x, y \}.
\edima
Since the zeros of $\mce $ correspond to the fixed points of system, which form an $n-2$ dimensional manifold, the added constraint on $\gamma $ pinches the neighborhood of the origin at the fixed points.

The next approximation $(Z, w_s^{(2)})$ is based on points from $\mathcal S^{(2)}$ that lie on the nullclines $\dot \gamma_1=0$ and $\dot \gamma_2=0$. We perform  preliminary calculations of $\mathcal U  , \mathcal W $ and $\dot \gamma $ from (\ref{eqnNewCoordsDegFull}) to describe this intersection.

Let $(Z, w_s)\in \mathcal S^{(2)}$ have the $\gamma $ block $\displaystyle \gamma^{(2)} = \left [ \begin{array}{c} 0 \\ 2 \mce \end{array}\right]+
\mce  \left [ \begin{array}{cc} -\mathrm a &0\\ \mathrm a^2 & \mathrm a \end{array}\right] ^{-1}
 \left [ \begin{array}{l}  x \\  y \end{array}\right] . $

At $\gamma^{(2)}$,  the linear part of the vector field $\dot \gamma $ from (\ref{eqnNewCoordsDegFull}) is $J_{\scriptscriptstyle \gamma} \gamma ^{(2)}=$
\bdima 
\left [ \begin{array}{cr}-a& 1\\a^2-1& -a \end{array} \right ]\mce
\left ( \left [ \begin{array}{cc} -\mathrm a &0\\ \mathrm a^2 & \mathrm a \end{array}\right] ^{-1}
\left [ \begin{array}{l} x \\ y \end{array}\right] +\left [ \begin{array}{c} 0 \\ 2 \end{array}\right] \right )= \mce \left ( L  \left [ \begin{array}{l} x \\ y \end{array}\right] +\left [ \begin{array}{c} 2 \\ -2 \mathrm a  \end{array}\right] \right ),
\edima
where  $L$ is the nonsingular matrix $\displaystyle \left [ \begin{array}{cr}-a& 1\\a^2-1& -a \end{array} \right ] \left [ \begin{array}{cc} -\mathrm a &0\\ \mathrm a^2 & \mathrm a \end{array}\right] ^{-1}$.

Let $\mathcal U^{\scriptscriptstyle (2)} $ and $\mathcal W^{\scriptscriptstyle (2)}$ be the nonlinear functions
$\mathcal U $ and $\mathcal W$ evaluated at $(Z, w_s^{(2)})$.
 We get that:
\bdima
\frac{1}{\mce}\left [ \begin{array}{l}\dot \gamma_1\\ \dot\gamma_2\\ \end{array} \right ] =
 L  \left [ \begin{array}{l} x \\ y \end{array}\right] +\left [ \begin{array}{c} 2 \\ -2 \mathrm a  \end{array}\right]
+ \frac{1}{\mce}
\left [ \begin{array}{l}-\frac{1}{n}X^{\scriptscriptstyle T}  \mathcal W^{\scriptscriptstyle (2)} \\[.1cm]\frac{1}{n}X^{\scriptscriptstyle T} (\mathcal U^{\scriptscriptstyle (2)} + \mathrm a \; \mathcal W^{\scriptscriptstyle (2)})
\end{array} \right ]\stackrel{\text{def}}{=} F
\edima
 \subsection*{Claim}  For $(x, y, Z)$ near zero, $F=F(x, y, Z)$  is a smooth function
  satisfying $ \displaystyle F(x, y, 0)=L  \left [ \begin{array}{l} x \\ y \end{array}\right].$
 \begin{proof} To prove the claim, we first evaluate $\mathbb c$ and $ \mathbb s$ at $(Z,w_s^{(2)})$ to get  $\mathbb c^{(2)}, \mathbb s^{(2)},$  then calculate $\mathcal U^{\scriptscriptstyle (2)} $ and $\mathcal W^{\scriptscriptstyle (2)}$.
  For the latter, only the presence of the factor $\mce $ is relevant.

From (\ref{eq_cs_shadow}) and
$\displaystyle \left [ \begin{array}{cc} -\mathrm a &0\\ \mathrm a^2 & \mathrm a \end{array}\right]  \left [ \begin{array}{c} \gamma^{(2)}_1 \\ \gamma^{(2)}_2 \end{array}\right]
=  \left [ \begin{array}{c} 0 \\ 2 \mathrm a \; \mce \end{array}\right]+
\mce
 \left [ \begin{array}{l}  x \\  y \end{array}\right] , $
we obtain
\beq \label{two_step_cs}
\begin{array}{l}
\mathbb c^{(2)}+ \ii_n = \cos \Theta  -\mce X -\mathrm a \gamma_1^{(2)}  X   =\cos \Theta  -\mce X +x \mce X \\
\mathbb s^{(2)}=-\sin \Theta + (\mathrm a^2 \gamma_1^{(2)}+\mathrm a\; \gamma_2^{(2)} )X = -\sin \Theta+ 2 \mathrm a \mce  X + y  \mce  X.
\end{array}
\eeq

From (\ref{two_step_cs}), dropping the superscripts ${}^{(2)}$, we get that in  $\displaystyle (\mathbb c+ \ii_n)^2+ \mathbb s^2-\ii_n$ the terms $\cos ^2 \Theta +\sin^2 \Theta -\mathbb 1_n^2$ cancel out, and the remaining terms have factors of $\mce$:
\beq \label{eq_speed}
\begin{array}{ll}
(\mathbb c+ \ii_n)^2+ \mathbb s^2-\ii_n &=
 2(-1+x)\mce \cos \Theta \odot X -2 (2 \mathrm a+y)\mce \sin \Theta \odot X + \\
 & \left ( (x-1)^2 +(2 \mathrm a+y)^2\right )  \mce^2 \ii_n.
 \end{array}
 \eeq
We get $ (\mathbb c+ \ii_n)^2+ \mathbb s^2-\ii_n = \mce G_1$, where $G_1$ is defined as
\newline
 $G_1 (x, y, Z)=  2(-1+x) \cos \Theta \odot X -2 (2 \mathrm a+y) \sin \Theta \odot X + ((x-1)^2+(2 \mathrm a+y)^2 )\mce \ii_n. $
The vector-valued function $G_1$ depends smoothly of $(x, y, Z)$ if $Z$ is small, and satisfies
$G_1(x, y, 0)=2(-1+x) X$, since $\mce (0)=0, \Theta (0)=0.$

This and (\ref{eqn_abuw_sys_nonlinearities}) imply  that $\mathcal U $ and $\mathcal W -2\mathbb c$ have factors of $\mce$, denoted $G_2$ and $G_3$:
\bdima
\begin{array}{l}
G_2(x, y, Z)=\frac{1}{\mce} \mathcal U =(-1) \left (\sin \Theta- (2 \mathrm a+y)\mce  X \right ) \odot G_1, \\
G_3(x, y, Z)=\frac{1}{\mce}( \mathcal W -2\mathbb c) = (-1)(\mathbb{c} +\mathbb 1_n ) \odot  G_1 =
(-1)(\cos \Theta +(-1+x)\mce X) \odot  G_1.
\end{array}
\edima
$G_2$ and $G_3$ are smooth, with  $G_2(x, y, 0)=0$  and $G_3(x, y, 0)= (-1)\mathbb 1_n \odot  G(x, y, 0)=2(1-x)X.$
Denote the scalar functions $\displaystyle \frac{1}{n}X^{\scriptscriptstyle T} G_2$ by $g_4$, and
$\displaystyle \frac{1}{n}X^{\scriptscriptstyle T} G_3$ by $g_5.$ We get $g_4(x, y, 0)=0$ and  $g_5(x, y, 0)=2-2x.$

Note that $\mce $ factors into  $\displaystyle \frac{1}{n} X^{\scriptscriptstyle T} \mathcal W$ as well, since
$\displaystyle \frac{1}{n}X^{\scriptscriptstyle T} \mathcal{W} = \frac{1}{n}X^{\scriptscriptstyle T} \mce G_3+
\frac{1}{n}X^{\scriptscriptstyle T} (2\mathbb c) = \mce g_5 -2a \gamma_1 =\mce g_5+2x \mce . $ The latter used (\ref{eq_cs_shadow}) and $-a\gamma_1^{(2)}=x$.
This proves that the function $F$  is a combination of smooth functions:
\bdima
F(x, y, Z)=
 L  \left [ \begin{array}{l} x \\ y \end{array}\right] +\left [ \begin{array}{c} 2 \\ -2 \mathrm a  \end{array}\right]
+
\left [ \begin{array}{l}-(g_5+2x) \\[.1cm]g_4+a(g_5+2x)
\end{array} \right ], 
\edima
with
$ F(x, y, 0)=
 L  \left [ \begin{array}{l} x \\ y \end{array}\right] +\left [ \begin{array}{c} 2 \\ -2 \mathrm a  \end{array}\right]
+
\left [ \begin{array}{l}-(2-2x+2x) \\[.1cm]a(2-2x+2x).
\end{array} \right ]=  L  \left [ \begin{array}{l} x \\ y \end{array}\right].$
\end{proof}

\begin{proposition}\label{Prop_Definegamma} Let $Z$ be in a small neighborhood of the origin in $\R^{n-1}.$ There exist smooth scalar functions $x=x(Z)$ and $y=y(Z)$ with $x(\mathbb 0_{n-1})= y(\mathbb 0_{n-1})=0$ such that the point $(Z, Y_{\scriptscriptstyle Z} )\in \R^{4n-2}$, with $ Y_{\scriptscriptstyle Z}=col( \alpha(Z), \mathbb 0_{n-1}, \mathbb 0_{n-1}, \gamma_{\scriptscriptstyle Z})$
having $\gamma_{\scriptscriptstyle Z}= \mce \left( \left [ \begin{array}{c} 0 \\ 2 \end{array}\right]+
\left [ \begin{array}{cc} -\mathrm a &0\\ \mathrm a^2 & \mathrm a \end{array}\right] ^{-1}
 \left [ \begin{array}{l}  x(Z) \\  y(Z) \end{array}\right]\right ),$
is on the nullclines $\dot \gamma_1=\dot \gamma_2=0.$

Moreover, at $(Z, Y_{\scriptscriptstyle Z})$ the nonlinear functions $\mathcal U$ and $\mathcal W$ satisfy:
\beq \label{UWnullcline}
\mathbb T  \mathcal{U} =-2 \mce \left [ \begin{array}{c} Z_U\\-Z_L\\0 \end{array} \right ]+\mce \mco(|Z|_r |Z|) \; \mbox{ and } \;
\mathbb T  \mathcal{W}  = 2 \alpha (Z)+\mce \mco(|Z|_r |Z|).
\eeq
\end{proposition}

\begin{proof}

Use the just-established fact that for small $x, y, Z$ if $\gamma^{(2)} $ denotes
\linebreak $\gamma^{(2)}= \mce \left( \left [ \begin{array}{c}  0 \\ 2 \end{array}\right]+
\left [ \begin{array}{cc} -\mathrm a &0\\ \mathrm a^2 & \mathrm a \end{array}\right] ^{-1}
 \left [ \begin{array}{l}  x \\  y \end{array}\right]\right )$ then
 at the point 
 of coordinates
\linebreak $col( Z, \alpha(Z), \mathbb 0_{n-1}, \mathbb 0_{n-1}, \gamma ^{(2)})$, we have
$\displaystyle \left [ \begin{array}{l}\dot \gamma_1 \\\dot \gamma_2  \end{array} \right ] = \mce F(x, y,  Z)$ where $F$ is  smooth,  $F(x, y, \mathbb 0_{n-1})=0$, and at $Z=\mathbb 0_{n-1}, $ the partial derivatives  $\frac{\partial F}{\partial x}, \frac{\partial F}{\partial x}$ form the nonsingular matrix $L$.
By the Implicit Function Theorem, there exists a neighborhood of $\mathbb 0_{n-1}$ where the equation $F(x, y, Z) =0$ has unique solutions $x=x(Z), y=y(Z)$, which depend smoothly on $Z.$
Therefore using $\gamma_{\scriptscriptstyle Z}= \left . \gamma^{(2)} \right| _{ x=x(Z), y=y(Z)} $  yields a point $(Z, Y_{\scriptscriptstyle Z})$ on the stated nullclines.

In order to establish the approximations of $\mathcal U$ and $\mathcal W,$  we use $ o$ for  $\mce \mco(|Z|_r |Z|)$ errors.  We use the approximations from  Lemma \ref{LemmaBigOEstimates} to eventually replace most sines and cosines with their respective averages, $z_{n-1}$ and $(1+\alpha_{n-1})$. In (\ref{eq_speed}), replacing $\mce \cos \Theta $ by $\mce (\alpha_{n-1}+1)\ii_n$ and discarding $\mce ^2 $ introduce $o$ errors; replacing
$x, y, \alpha $ and $\Theta $ by zero introduces $\mco (|Z|)$ errors,
thus
\bdima
\begin{split}
(\mathbb c+ \ii_n)^2+ \mathbb s^2-\ii_n =
  2\mce (-1+x) (\alpha_{n-1}+1) X -2 \mce (2 \mathrm a+y) \sin \Theta \odot X  + o=\\
  2\mce \left (-1+ \mco(|Z|)\right )X -2\mce (2 \mathrm a+y) \left[\begin{array}{c} \sin \Theta_U  \\ - \sin \Theta_L \end{array}\right]+o.
  \end{split}
\edima
Substitute the latter  and (\ref{two_step_cs}) in the definition of $\mathcal U$, (\ref{eqn_abuw_sys_nonlinearities}):
\bdima
\begin{split}
\mathcal U = \left (\sin \Theta -(2a+y)\mce X \right ) \odot \left ( 2\mce \left (-1+ \mco(|Z|)\right )X -2\mce (2 \mathrm a+y) \left[\begin{array}{c} \sin \Theta_U  \\ - \sin \Theta_L \end{array}\right] \right )+o \\
 =\sin \Theta  \odot \left (2\mce \left (-1+ \mco(|Z|)\right )X -2\mce (2 \mathrm a+y) \left[\begin{array}{c} \sin \Theta_U  \\ - \sin \Theta_L \end{array} \right ] \right )  +o
 \end{split}
\edima
Use $\sin ^2\theta_k -z_{n-1}^2= \mco(|Z|_r|Z|)$, we get $\displaystyle \mce \sin \Theta  \odot  \left[\begin{array}{c} \sin \Theta_U  \\ - \sin \Theta_L \end{array} \right ] = \mce z_{n-1}^2 X+ o$ and
\bdima
\begin{split}
\mathcal U= 2\mce \left (-1+ \mco(|Z|)\; \right )\left[\begin{array}{c} \sin \Theta_U  \\ - \sin \Theta_L \end{array} \right ]
-2\mce (2 \mathrm a+y) z_{n-1}^2 X+ o.
\end{split}
\edima
From the block structure of $\mathbb T $, $TV = \ii_{\scriptscriptstyle N-1}$,   $T\ii_{\scriptscriptstyle N}  = \oo_{\scriptscriptstyle N-1}$, and $\ii_{\scriptscriptstyle N}^{\scriptscriptstyle T} V=0$  we get:
\bdima
\mathbb T \left[\begin{array}{c} \sin \Theta_U  \\ - \sin \Theta_L \end{array} \right ]  = \left [\begin{array}{cc}
T & \OO_{\scriptscriptstyle N-1,N} \\
\OO_{\scriptscriptstyle N-1,N} & T \\
\frac{1}{n}\ii_{\scriptscriptstyle N}^{\scriptscriptstyle T} & \frac{1}{n}\ii_{\scriptscriptstyle N}^{\scriptscriptstyle T}
\end{array} \right ]
\left [ \begin{array}{c}
 VZ_U+z_{n-1}\ii_{N}\\
-VZ_L - z_{n-1}\ii_{N}
\end{array} \right ]=
\left [\begin{array}{c} Z_U \\ -Z_L \\ 0 \end{array}\right],
\edima
a vector of magnitude $|Z|_r, $ which multiplied by $\mce \mco (|Z|) $ produces $o$. This and $\mathbb T X =0$ imply
$\mathbb T  \mathcal{U} =(-2 \mce + \mce \mco( |Z|))\left [ \begin{array}{c} Z_U\\-Z_L\\0 \end{array} \right ]+o=
-2 \mce\left [ \begin{array}{c} Z_U\\-Z_L\\0 \end{array} \right ]+o $.

To estimate  $\mathbb T  \mathcal W,$ we  apply  a similar technique to
$\displaystyle  (\mathbb c+\ii_n) \odot ( (\mathbb c^2+ \ii_n)^2+ \mathbb s^2-\ii_n),$ which we denote by  $\mathcal V.$  Note that  $\mathcal W= -\mathcal V+2\mathbb c$, and $\mathbb T \mathbb c= \alpha$, from (\ref{eq_cs_shadow}).
Proving
 $\mathbb T \mathcal W= 2\alpha +o $ becomes equivalent to proving $\mathbb T \mathcal V=o.$
 Using (\ref{two_step_cs}), we have
 \bdima
 \mathcal V= (\cos \Theta  +(-1+x) \mce X)\odot
 \left (   2\mce \left (-1+ \mco(|Z|)\right )X -2\mce (2 \mathrm a+y) \left[\begin{array}{c} \sin \Theta_U  \\ - \sin \Theta_L \end{array}\right]+o   \right ).
 \edima
Replacing $ \cos \Theta  $ by $ (1+\alpha_{n-1}) \ii_n $ introduces an error $\mco(|Z|_r |Z|)$, which after multiplication by $\mce $ becomes $o.$ Thus
 \bdima
 \begin{split}
\mathcal V=  (1+\alpha_{n-1}) \ii_n \odot  \left (   2\mce \left (-1+ \mco(|Z|)\right )X -2\mce (2 \mathrm a+y) \left[\begin{array}{c} \sin \Theta_U  \\ - \sin \Theta_L \end{array}\right]  \right )+ o=\\
   2  (1+\alpha_{n-1}) \mce \left (-1+ \mco(|Z|)\right )X- 2 \alpha_{n-1} \mce (2 \mathrm a+y) \left[\begin{array}{c} \sin \Theta_U  \\ - \sin \Theta_L \end{array}\right]+o.
 \end{split}
 \edima
Use
$\mathbb T X =0$ and $\mathbb T \left[\begin{array}{c} \sin \Theta_U  \\ - \sin \Theta_L \end{array} \right ]=\mco(|Z|_r)$ to get $\mathbb T \mathcal V= \alpha_{n-1} \mce \mco(|Z|_r)=o.$
\end{proof}

\begin{theorem}[The Central Manifold Approximation Theorem] \label{ThmApproximationTheorem}
 Let $\mathcal M$ be the local center manifold for the reduced system (\ref{eqnNewCoordsDegFull}), and $h$ its map.

Let $Y_{\scriptscriptstyle Z}\in\R^{3n-1}$  be as in Proposition \ref{Prop_Definegamma}, so that $(Z, Y_{\scriptscriptstyle Z}) $ is on the nullclines $\dot \gamma_1=\dot \gamma_2=0$ for $Z$ near $\mathbb 0_{n-1}.$
Then
\bdima
|h(Z)-Y_{\scriptscriptstyle Z}|= \mce(Z) \mathcal O(|Z|_r).
\edima
Moreover, the differential equations governing the flow on the center manifold are:
\begin{equation}\label{FlowEq}
 \frac{dZ}{dt}=\mathrm a \; \mce(Z)  \left [ \begin{array}{c} Z_U \\-Z_L\\ 0\end{array} \right ] + \mce(Z) \mathcal O(|Z|_r |Z|).
\end{equation}
On $\mathcal M$,  $\left. \dot{\gamma }\right |_{\scriptscriptstyle (Z, h(Z)) }= \mce(Z) \mathcal O(|Z|_r)$.
\end{theorem}

\begin{proof}
The estimate (\ref{shadow_error}) can be applied to the error $|h(Z)-Y_{\scriptscriptstyle Z}|,$ since the point $(Z, Y_{\scriptscriptstyle Z}) $ constructed in Proposition \ref{Prop_Definegamma} is in the shadow set $\mathcal S$.
Use $ \left. \dot{\gamma }\right |_{(Z, Y_{\scriptscriptstyle Z})}=0$ to conclude
$|Y_{\scriptscriptstyle Z}-h(Z)|\leq C
  \max \left \{ \left |\left. \mathbb T\mathcal U\right |_{ (Z, Y_{\scriptscriptstyle Z})} \right |, \;
   \left |-2 \alpha(Z)+ \left. \mathbb T\mathcal W  \right |_{ (Z, Y_{\scriptscriptstyle Z})} \right |, 0 \right \}$.

   The estimates (\ref{UWnullcline}) of Proposition \ref{Prop_Definegamma} show that the aforementioned maximum comes from
 $\mathbb T\mathcal U$, which has magnitude    $\mce |Z|_r, $ therefore $|Y_{\scriptscriptstyle Z}-h(Z)|= \mce \mco( |Z|_r).$

To prove (\ref{FlowEq}), we compare $\dot Z= -\frac{\mathrm a}{2} \left. \mathbb T\mathcal U\right |_{ (Z, h(Z))} $ and
$-\frac{\mathrm a}{2} \left. \mathbb T\mathcal U\right |_{ (Z, Y_Z)}$,  which was computed in (\ref{UWnullcline}). We get
\bdima
\dot Z= -\frac{\mathrm a}{2} \left ( \mathbb T \left( \left.  \mathcal U\right |_{ (Z, h(Z))}- \left. \mathcal U\right |_{ (Z, Y_Z)} \right ) -2\mce \left [ \begin{array}{c} Z_U \\-Z_L\\ 0\end{array} \right ] + \mce(Z) \mathcal O(|Z|_r |Z|)  \right ),
\edima
reducing the proof of (\ref{FlowEq}) to that of $  \left.  \mathcal{ U}\right |_{ (Z, h(Z))}-
 \left. \mathcal{ U}\right |_{ (Z, Y_Z)} =\mce \mathcal O(|Z|_r |Z|) .$

Note that the points $(Z, Y_{\scriptscriptstyle Z})$ and $(Z, h(Z))$ are at a distance $\mco(|Z|)$ from the origin, therefore the segment that connects them is within a ball $B$ centered at the origin, of radius $C|Z|$. The function $\mathcal U$ has gradient $\nabla \mathcal U$ that vanishes at the origin, therefore the magnitude of $\nabla \mathcal U$ at points from the ball $B$ is $\mco(|Z|).$ Using the mean value theorem we obtain that
\bdima
 \left |\mathcal U (Z, h(Z)) -\mathcal U(Z, Y_{\scriptscriptstyle Z})\right  | \leq \left( \mbox{max} | \nabla \mathcal U | \right )  |h(Z)- Y_{\scriptscriptstyle Z}| = \mathcal O(|Z|)\; \mce \mathcal O(|Z|_r),
\edima
completing the proof of (\ref{FlowEq}). Use a similar argument to estimate $\dot \gamma.$ The vector field defining $\dot \gamma $ is Lipschitz continuous, thus for some constant $C$ we have
\bdima
| \left. \dot{\gamma }\right |_{\scriptscriptstyle (Z, h(Z)) } -  \left. \dot{\gamma }\right |_{\scriptscriptstyle (Z, Y_Z) }|\leq C|h(Z)-Y_{ \scriptscriptstyle Z}| = \mce \mco (|Z|_r).
\edima
Since at $(Z, Y_{ \scriptscriptstyle Z})$ the vector field $\dot \gamma$ is zero, we conclude $\left. \dot{\gamma }\right |_{(Z, h(Z)) }= \mce(Z) \mco(|Z|_r)$.
\end{proof}

%
%

\subsection{Stability of the reduced flow.}

In view of center manifold theory, to establish the stability of the reduced ($Z,\be,\bd,\ba,\gamma_1,\gamma_2)$-system, it suffices to establish the stability of the flow on the center manifold. In fact, we will show that every solution that starts near the origin approaches an equilibrium point. Recall that in view of Lemma \ref{LemmaEquilibriumSolutions} and Remark \ref{RemarkZeroEnergyOnCenterManifold}, the set $\mathcal M\cap \{\mce = 0\}$ consists of equilibrium points of the flow, which in turn coincides with the set of ring state solutions centered at the origin.

We will use the following functions
\begin{equation}
\label{calD}
\mcd_U(Z) =\frac{1}{N}(VZ_U)^{\scriptscriptstyle  T} \cdot VZ_U \; \; \mbox{ and } \; \; \mcd_L(Z) = \frac{1}{N}(VZ_L)^{\scriptscriptstyle T}\cdot VZ_L.
\end{equation}
Using the definition of $\sin(\Theta)$, we notice that $\mcd_U(Z) = \frac{1}{N} \sum \limits_{k=1}^N(\sin (\theta_k) -z_{n-1})^2$. Furthermore, in view of Remark \ref{remarkPropertiesOfEnergy}, $z_{n-1} = \frac{1}{N} \sum \limits_{k=1}^N\sin (\theta_k) $. Thus,  $\mcd_U(Z)$ can be interpreted as the dispersion/scattering of the right group of particles.
Similarly, $\mcd_L(Z)$ can be viewed as the dispersion of the left group of particles.

$\mcd_U(Z) $ and $\mcd_L(Z)$ are comparable with $|Z_U|^2$ and $|Z_L|^2$, respectively, since the $N \times(N-1)$ matrix $V$ is full rank.
In the following lemma we show that  $\mcd_L(Z)$ is comparable to $|Z|_r$ in  $\{\mce>0\}$ and that  $\mcd_U(Z)$ is comparable to $|Z|_r$ in $\{\mce <0\}$.

  \begin{lemma}\label{lemmaDLBounds}
  For $Z$ small enough,
  \beq \label{mcemcd}
  \mce= \frac{1}{4 \sqrt{1 - z_{n-1}^2}^3} \left(\mcd_L(Z) - \mcd_U(Z) \right) + \mco(|Z|_r^3 |Z|).
  \eeq
 There exists a constant $C>0$ such that if $Z$ satisfies $\mce (Z)>0, $ then
  \bdima
  \frac{1}{C}|Z|_r \leq  \sqrt{\mcd_L(Z)} \leq C |Z|_r \;  \mbox{ and } \; \mcd_L(Z)\geq \frac{1}{2}\mcd_U(Z).
  \edima
 If $\mce (Z)<0$, then   $\displaystyle \frac{1}{C}|Z|_r \leq \sqrt{\mcd_U(Z)} \leq C |Z|_r $ and $\mcd_U(Z)\geq \frac{1}{2}\mcd_L(Z).$
  \end{lemma}

  \begin{proof} By definition of $\mce$, we have that
   \begin{equation*}
   \begin{split}\mce(Z) & = \frac{1}{n}\sum_{k=1}^N\big(\cos(\theta_k) - \cos(\theta_{N+k}) \big) = \frac{2}{n}\sum_{k=1}^N \left [ \sin^2\left(\frac{\theta_{k+N}}{2}\right) - \sin^2\left(\frac{\theta_k}{2}\right ) \right] \\
   & = \frac{1}{N}\sum_{k=1}^N \left [ \sin^2\left(\frac{1}{2}\sin^{-1}(t_{N+k}+z_{n-1})\right) - \sin^2\left(\frac{1}{2}\sin^{-1}(t_{k}+z_{n-1})\right ) \right],
   \end{split}
  \end{equation*}
  where $t_i = \sin(\theta_i)-z_{n-1}$ or, equivalently, the $i$-th component of the vector $VZ_U$ for $1\leq i\leq N$ and of the vector $VZ_L$ if $N+1\leq i\leq n$.

  Given a small $z\in \mathbb R$, consider the function $t\mapsto \sin^2(\frac{1}{2}\sin^{-1}(t+z))$ and its Taylor  expansion of order 2 at $t=0.$ Then,
    $$\sin ^2 \left(\frac{1}{2} \sin^{-1}( t+z)\right )= \sin ^2 \left(\frac{1}{2}\sin z\right ) + \frac{z}{2\sqrt{1 - z^2}} t
  + \frac{1}{4 \sqrt{1 - z^2}^3}t^2+ R_2(t).$$
Expressing the remainder $R_2(t)$ in the Lagrange form, we can directly verify that $R_2(t) = t^3g(t,z)$, where $|g(t,z)|\leq K (|t|+|z|)$ for some constant $K>0$. Note that the constant $K$ depends only on the radius of the neighborhood of the origin.

Substituting the Taylor series expansion  with $t= t_i$ and $z=z_{n-1}$ in the equation for $\mce$ above, the terms $\pm \sin ^2 \left(\frac{1}{2} \sin^{-1}( z_{n-1})\right )$ cancel out and we get:
\begin{multline}\label{eqnELemmaAux}
\mce(Z) = \frac{1}{N} \frac{z_{n-1}}{2\sqrt{1 - z_{n-1}^2}} \left( \sum _{k=1}^N t_{N+k} -\sum _{k=1}^N t_{k}\right)
 \\ +  \frac{1}{N}\frac{1}{4 \sqrt{1 - z_{n-1}^2}^3} \left( \sum _{k=1}^N t_{N+k}^2 -\sum_{k=1}^N t_{k}^2\right )
+ \frac{1}{N}\sum _{k=1}^N (R_2(t_{N+k})-R_2(t_k)).
\end{multline}
Note that $\sum_{k=1}^N t_k = \ii_N^T VZ_U = 0$, as the columns of the matrix $V$ are orthogonal to the vector $\ii_N$. Similarly,
$\sum_{k=1}^N t_{N+k}=0$. Note also that $\mcd_U = \frac{1}{N}\sum_{k=1}^N t_k^2 $ and $\mcd_L = \frac{1}{N}\sum_{k=1}^N t_{N+k}^2$. Finally, observe that $g(t_k,z_{n-1}) = \mathcal O(|Z|)$ and $ t_k = \mathcal O(|Z|_r) = \mathcal O(\sqrt{\mcd_U+\mcd_L})$. Thus, continuing with Equation (\ref{eqnELemmaAux}), we obtain that
\begin{equation}
\label{edudl}
\begin{split}
\mce(Z) &= \frac{1}{4 \sqrt{1 - z_{n-1}^2}^3} \left(\mcd_L(Z) - \mcd_U(Z) \right) \\
& + \frac{1}{N}\sum _{k=1}^N \left(t_{N+k}^2 g(t_{N+k},z_{n-1}) - t_k^2 g(t_k,z_{n-1})\right)  \\
& =  \frac{1}{4 \sqrt{1 - z_{n-1}^2}^3} \left(\mcd_L(Z) - \mcd_U(Z) \right) + (\mcd_L(Z) + \mcd_U(Z))^{3/2} H(Z),
\end{split}
\end{equation}
where $H(Z) = \mathcal O(|Z|)$. This proves (\ref{mcemcd}). Set $f(Z) = 4 \sqrt{1 - z_{n-1}^2}^3$. Then $$\mcd_U(Z) -\mcd_L(Z) = -f(Z)\mce(Z) + (\mcd_L(Z) + \mcd_U(Z))^{3/2} H(Z)f(Z).$$

Consider the case $\mce(Z)>0$.  Choose a small neighborhood of the origin in $\mathbb R^{n-1}$ in which $(3/2)^{3/2}\mcd_U(Z)^{1/2} |H(Z)|f(Z) < \frac{1}{2} $. We claim that in this neighborhood $\mcd_L(Z)\geq \frac{1}{2}\mcd_U(Z)$. Indeed, assume towards contradiction that $\mcd_L(Z)< \frac{1}{2}\mcd_U(Z)$. It follows that $\mcd_L(Z) -\mcd_U(Z)< -\frac{1}{2}\mcd_U(Z)$. Hence,
\begin{equation*}
\begin{split}
\frac{1}{2}\mcd_U(Z) & < \mcd_U(Z) -\mcd_L(Z) = -f(Z)\mce(Z) + (\mcd_L(Z) + \mcd_U(Z))^{3/2} H(Z)f(Z)
\\
& \leq  (\mcd_L(Z) + \mcd_U(Z))^{3/2} H(Z)f(Z) .
\end{split}
\end{equation*}
   The last inequality implies that $H(Z)> 0$. It follows that
    $$\frac{\mcd_U(Z)}{2} < (\mcd_L(Z) + \mcd_U(Z))^{3/2} H(Z)f(Z) < (\frac{3}{2} \mcd_U(Z) )^{3/2} H(Z)f(Z) \leq \frac{\mcd_U(Z)}{2},$$
     which is a contradiction.  Therefore, $|Z|_r = \mathcal O(\sqrt{\mcd_U + \mcd_L}) = \mathcal O(\sqrt{\mcd_L})$ and the result follows.
     The  proof in the case when $\mce(Z)<0$ is similar.
   \end{proof}

 %

\begin{lemma}\label{lemmaDifEqDL}
  On the center manifold $\mathcal M$, if $Z$ is near the origin, then
  \begin{equation} \label{EqnDifEqnDispersion}
  \begin{split}
\frac{d\mcd _L}{dt} &= -2 \mathrm a \; \mce(Z)    \mcd_L(Z) ( 1 + \mathcal O(|Z|) )\mbox{ whenever } \mce (Z)>0 \\
\frac{d\mcd _U}{dt} &= 2 \mathrm a \; \mce(Z)  \mcd_U(Z) \left( 1+  \mathcal O(|Z|)  \right ) \mbox{ whenever } \mce(Z)<0.
\end{split}
\end{equation}
\end{lemma}
 \begin{proof} Consider the case $\mce(Z)>0$. In view of Lemma \ref{lemmaDLBounds}, $|Z|_r = \mathcal O(\sqrt{D_L})$.
 By definition, $\mcd_L = \frac{1}{N} (VZ_L)^{\scriptscriptstyle T} VZ_L$. Differentiating both sides of the equation along trajectories of the flow on the center manifold and applying Theorem \ref{ThmApproximationTheorem}, we obtain that
 \begin{equation*}
 \begin{split}\frac{d\mcd_L}{dt} & = \frac{2}{N}(VZ_L)^{\scriptscriptstyle T}\cdot V\dot Z_L  = -\frac{2}{N}\mce(Z)(VZ_L)^{\scriptscriptstyle T} \cdot \left [\mathrm a VZ_L+\mathcal O(|Z|_r |Z|)  \right]
 \\
 & = -2 \mathrm a \; \mce(Z) ( \mcd_L(Z) + \mathcal O(|Z|_r^2|Z|)) = -2\mathrm a \; \mce(Z)  \mcd_L(Z)(1 + \mathcal O(|Z|)).
 \end{split}
 \end{equation*}
The proof for the case $\mce (Z)<0$ is analogous and will be omitted.
\end{proof}

\begin{theorem}\label{TheoremStabilityReducedSystemCenterManifold} The flow on the center manifold $\mathcal M$ of the reduced system (\ref{eqnNewCoordsDegFull})
is stable at the origin. Moreover, any trajectory that starts near the origin approaches an equilibrium point of the form $(Z_\omega,\alpha(Z_\omega),\oo_{2n})$, with $\mce(Z_\omega)=0$.
 \end{theorem}

 \begin{proof}

From Remark \ref{RemarkZeroEnergyOnCenterManifold}, the sets $\mathcal M_0=\mathcal M \cap \{\mce = 0\}$, $\mathcal M_+ = \mathcal M \cap \{\mce > 0\}$, and $\mathcal M_- = \mathcal M \cap \{\mce < 0\}$ are flow-invariant and the set $\mathcal M_0$ consists of equilibrium points. Thus, to establish stability of the flow on $\mathcal M$, we will separately show that the flow on $\mathcal M_+$ and on $\mathcal M_-$ is stable.
We will prove that $\mcd_L$ and $\mcd_U$ are Lyapunov functions for the flow on $\mathcal M_+$ and $\mathcal M_-$, respectively, which will ensure the stability of the flow near the origin. Finally, applying  LaSalle's Invariance Principle, we will establish the convergence of solutions to the set $\mathcal M_0$.

We proceed to show that the flow on $\mathcal M_+$ is stable. (The proof for the stability on $\mathcal M_-$ uses similar arguments, with swapped $U$ and $L$ subscripts, and $-\mce $ in place of $\mce $.)  Recall that the subset $\mathcal M_+$ is flow-invariant, consisting of points with $\mce (Z)>0.$

From Equation (\ref{EqnDifEqnDispersion}), the function $t\mapsto \mcd_L(Z(t))$ is decreasing along trajectories inside $\mathcal M_+$. Applying Lemma \ref{lemmaDLBounds}, we get:
\bdima
\frac{1}{C} |Z(t)|_r\leq \sqrt{\mcd_L(Z(t))} \leq \sqrt{\mcd_L(Z(0))} \leq C|Z(0)|
\edima
thus there exists $K>0$  such that $|Z(t)|_r \leq K |Z(0)|$.

We estimate $z_{n-1}(t)$ using the functions $\sqrt \mcd_L \pm z_{n-1}$. Their derivatives along trajectories, computed using
 Theorem \ref{ThmApproximationTheorem} and the previous lemma are:

\begin{equation*}
\begin{split}
\frac{d(\sqrt \mcd_L \pm z_{n-1})}{dt} & = \frac{-2 \mathrm a \; \mce(Z)   \mcd_L(Z) ( 1 + \mathcal O(|Z|) ) }{2\sqrt{ \mcd_L(Z)}}\pm  \mce(Z)\mathcal O(|Z|_r|Z|) \\
& =  -2 \mathrm a \; \mce(Z) \sqrt{ \mcd_L(Z)} ( 1 + \mathcal O(|Z|) ) \pm  \mce(Z)\sqrt{\mcd_L(Z)}\mathcal O(|Z|) \\
& = -2 \mathrm a \; \mce(Z)\sqrt{\mcd_L(Z)} \left[1+ \mathcal O(|Z|)  \right],
\end{split}
\end{equation*}
which shows that for small $Z$, the functions $\sqrt \mcd_L \pm z_{n-1}$ are decreasing. Thus for $t>0$:

\bdima
\begin{split}
z_{n-1}(t) \leq -\sqrt{ \mcd _L (Z(t))}+ z_{n-1} (0)+  \sqrt{ \mcd_L (Z(0))} \leq\\
  z_{n-1} (0)+  \sqrt{ \mcd_L (Z(0))}  \leq (C+1) |Z(0)| \; \mbox{ and}
  \end{split}
  \edima
\bdima
\begin{split}
- z_{n-1}(t)  \leq -\sqrt{ \mcd _L (Z(t)) }+  \sqrt{ \mcd _L (Z(0))} -z_{n-1} (0) \leq\\
 \sqrt{ \mcd _L (Z(0))} -z_{n-1} (0) \leq (C+1)|Z(0)|.
 \end{split}
 \edima
 Thus $
 |z_{n-1}(t)| \leq  (C+1) |Z(0)|$. This and the prior estimate of $|Z(t)|_r \leq K |Z(0)|$ imply
 $|Z(t)| = |Z(t)|_r + |z_{n-1}(t)|  \leq (K+C+1)|Z(0)|,$ establishing the stability of the system near the origin in $\mathcal M_+$ and, using similar arguments, in $\mathcal M_-$. Thus the flow on the reduced center manifold is stable near the origin.

Finally, for the asymptotic limit:  consider an initial condition $Z_0$ small.  If $\mce (Z_0)=0,$ then $Z_0$ is a fixed point.

 If $\mce (Z_0)> 0, $ then any limit point $Z_\omega$  of $Z(t)$ satisfies $\mce (Z_\omega) \geq 0.$
Applying LaSalle's Invariance Principle to the negative semi-definite ($\dot \mcd _L\leq 0$) function $\mcd _L$, we obtain that $Z_\omega$ satisfies  $\dot \mcd_L(Z_\omega) = 0$, which according to Lemma \ref{lemmaDifEqDL} is equivalent to  $\mce(Z_\omega) = 0$ or $\mcd_L(Z_\omega)=0$.  If $\mce(Z_\omega) = 0$, then in view of Remark \ref{RemarkZeroEnergyOnCenterManifold} the limit point $(Z_\omega,\alpha(Z_\omega),\oo_{2n})$ is  an equilibrium point of the flow and the result follows. If $\mce(Z_\omega)> 0$, then $\mcd_L(Z_\omega)$ must be equal to zero. Thus, Lemma \ref{lemmaDLBounds} implies that $|Z_\omega|_r = 0$ and $\sin(\Theta_\omega ) = z_{\omega, n-1}\ii_{n-1}$. It follows that for all $k$, $\cos \theta _{\omega, k}= \sqrt{1-z_{\omega, n-1}^2}$ and  $\mce(Z_\omega) = 0$, which is a contradiction.

  The case $\mce (Z_0)<0$, follows similarly, using $\mcd _U$.
  \end{proof}

Theorem  \ref{TheoremStabilityReducedSystemCenterManifold} and center manifold theory  imply the stability of the reduced system (\ref{eqnNewCoordsDegFull}) near the origin. 
Every trajectory that starts near the origin approaches a trajectory in the center manifold $\mathcal M$. Since trajectories in $\mathcal M$ approach fixed points  we conclude:

    \begin{corollary}\label{CorollaryStabilityReducedSystem} The reduced $(Z,\be,\bd,\ba,\gamma_1,\gamma_2)$-system is stable near the origin. Each trajectory that starts near the origin approaches a fixed point of the system.
    \end{corollary}

\begin{remark}
  Consider a trajectory in $\mathcal M$ that converges to a degenerate configuration.
  Then the rate of convergence is slower than $1/\sqrt t.$
 \end{remark}

\begin{proof}
Consider a point $(Z, h(Z)$ in $\mathcal M$, not a fixed point (thus $\mce (Z) \neq 0$),  whose limit is a degenerate state. Assume $\mce (Z)>0.$ From Lemma \ref{lemmaDLBounds}, $\mcd _L$ is comparable to $|Z|_r^2$ and we get:
\bdima
\begin{split}
\mce(Z) = \frac{1}{4 \sqrt{1 - z_{n-1}^2}^3} \left(\mcd_L(Z) - \mcd_U(Z) \right) + \mathcal O(|Z|_r^3)\leq \\
\frac{1}{4 \sqrt{1 - z_{n-1}^2}^3} (\mcd_L(Z)-0) + \mathcal O(|Z|_r^3)\leq K \mcd_L(Z) \; \mbox{ for some } K.
\end{split}
\edima
Substituting into $\dot \mcd_L$  given in  Lemma \ref{lemmaDifEqDL}, we get
\bdima
\frac{d\mcd _L}{dt} = -2 \mathrm a \mce   \mcd_L ( 1 + \mathcal O(|Z|) )\geq -2 K \mathrm a   \mcd_L^2 ( 1 + \mathcal O(|Z|) )\geq -K' \mcd_L^2,
\edima
implying that $\mcd_L(t)\geq \frac{\mcd_L(0)}{1+t (K'\mcd_L(0))}\geq \frac{K''}{t}$ for large $t$.

Let $Z_\star= \lim_{t\to \infty} Z(t)$. Since the limit is a degenerate state, from Remark \ref{RemarkGeometricMeaning}, the agents are placed on the slanted diameter of angle $\arcsin z_{\star, n-1}$  and $Z_{\scriptscriptstyle \star, L}=Z_{\scriptscriptstyle \star, U}=0$. The inequality $\mcd_L(t) \geq K''/t$
becomes $|V\left (Z_{\scriptscriptstyle L}(t)-Z_{\scriptscriptstyle \star, L}\right )|^2\geq K''/t.$ The full rank of $V$ implies the rate of convergence of $Z_{\scriptscriptstyle L}$ to $Z_{\scriptscriptstyle \star, L}$ is above $1/\sqrt t.$

\end{proof}

\begin{remark}  Consider a swarm with initial conditions near those of a ring state, whose right group particles $r_1,\ldots,r_N$ start out with common initial conditions and initial velocities. 
Assume the left particles start dispersed, $\mcd _L(0)>0.$ Then the mutual distances within the left group of particles approach zero at a rate comparable to $\frac{1}{\sqrt t}$, and $\mce (t)$ approaches zero at the rate of $1/t$.
\end{remark}

\begin{proof}
 The right particles follow the same trajectory at all times, implying $\mcd_U =0$ and $Z_U=0$ at all times.
Thus   $|Z|_r^2$ equals $|Z|_L^2$ and  is comparable to $\mcd _L$. Using  (\ref{mcemcd}), we get
$$\mce(Z) = \frac{1}{4 \sqrt{1 - z_{n-1}^2}^3} \mcd_L(Z) + \mathcal O(|Z|_r^3)= \frac{1}{4 \sqrt{1 - z_{n-1}^2}^3} \mcd_L(Z)(1 + \mathcal O(|Z|_r)) .$$
Substituting this into $\dot \mcd_L$ as presented in  Lemma \ref{lemmaDifEqDL}, we obtain that
 $$\frac{d\mcd_L}{dt} = -\frac{\mathrm a}{2 \sqrt{1 - z_{n-1}^2}^3} \mcd_L^2(Z)\left(1+\mathcal O(|Z|)\right).$$
 Thus, for small enough $Z$, the quantity $-\frac{\dot \mcd_L}{\mcd_L^2} = -\frac{d}{dt}\left( \frac{1}{\mcd_L}\right)$ remains within a neighborhood of $\frac{\mathrm a}{2}$, say,
 within $[\frac{\mathrm a}{3},\mathrm a]$. It follows that

 $$\frac{\mcd_L(0)}{1+\mathrm a\mcd_L(0)t}\leq \mcd_L(t)\leq \frac{\mcd_L(0)}{1+\frac{\mathrm a}{3}\mcd_L(0)t},$$
 which shows that mutual distances within the right group of particles approach zero at a rate comparable to $\frac{1}{\sqrt t}$.
 Thus $|Z|_r\simeq \frac{1}{\sqrt t}.$

 Use $\mce(Z) = \frac{1}{4 \sqrt{1 - z_{n-1}^2}^3} \mcd_L(Z) + \mathcal O(|Z|_r^3)$ to  get $\mce \simeq \frac{1}{t}$ for large $t.$
\end{proof}

%

\subsection{Stability of the full system.}

 In this section, we establish stability of the origin for the $4n$-dimensional system in
 $(Z,\be,\bd,\ba,\gamma_1 ,\gamma_2,\delta_1,\delta_2)$ coordinates, governed by
  (\ref{eqnNewCoordsDegFull}) and (\ref{eqnNewCoordsJustDelta}). For notational brevity,
  we identify $\delta=(\delta_1, \delta _2)^{\scriptscriptstyle T}$ with the complex number $\delta_1+i \delta _2.$ Without mentioning it explicitly, we will  work within  a ball centered at the origin, with a radius small enough that its first $4n-2$ components lie in the region of stability for  (\ref{eqnNewCoordsDegFull}).

%

\begin{theorem}\label{TheoremStabilityFullSystem} The $(Z,\be,\bd,\ba,\gamma_1 ,\gamma_2,\delta_1,\delta_2)$-system governed by (\ref{eqnNewCoordsDegFull}) and (\ref{eqnNewCoordsJustDelta}) is stable at the origin.

 Moreover, there exists a neighborhood $\mathcal N$ of the origin in $\mathbb R^{4n}$ such that the flow of (\ref{eqnNewCoordsDegFull}) and (\ref{eqnNewCoordsJustDelta}) that starts in $\mathcal N$ asymptotically approaches the periodic ``steady-state'' oscillatory solution of the form 
 $(Q_*, e^{-it}(\delta_{*,1}+i\delta_{*,2}))$,
 where  $Q_*$ is an equilibrium point for the reduced system (\ref{eqnNewCoordsDegFull}) and $\delta_{*,1},\delta_{*,2}\in \mathbb R$.
\end{theorem}

\begin{proof}
Denote by $\varphi_t$ the flow of the $4n$-dimensional  system  (\ref{eqnNewCoordsDegFull}) and (\ref{eqnNewCoordsJustDelta}).
The system  has $(Z,\delta_1,\delta_2)$ as neutral coordinates, thus has an $(n+1)$-dimensional center manifold;  the coordinates $(\be,\bd,\ba,\gamma_1,\gamma_2)$ are stable.

Denote by $\mathcal M_{n-1}$ the center manifold of the reduced system constructed in Theorem \ref{ThmApproximationTheorem}, and let $h$ be its map. Note that the manifold $\mathcal M_{n+1} = \mathcal M_{n-1}\times \mathbb R^{2}$  is (locally) forward invariant under (\ref{eqnNewCoordsDegFull}) and (\ref{eqnNewCoordsJustDelta}), and tangent to the linear center subspace. Thus $\mathcal M_{n+1}$ is the center manifold for (\ref{eqnNewCoordsDegFull})  and (\ref{eqnNewCoordsJustDelta}). Note that $\mathcal M_{n+1}\cap \{\mce = 0\}$, $\mathcal M_{n+1}\cap \{\mce > 0\}$, and $\mathcal M_{n+1}\cap \{\mce < 0\}$
are invariant under $\varphi_t$.  Their dynamics will be studied separately, with the last case omitted, due to its similarity to the former.

Theorem  \ref{TheoremStabilityReducedSystemCenterManifold} established that for small initial conditions,  the projection of the flow $(\mathcal M_{n+1},\phi_t)$  onto $\mathbb R^{4n-2}$ is independent of $\delta $ and remains close to the origin for $t>0$. To establish stability in $\R^{4n}$ is reduced to proving that the $(\delta_1, \delta_2)$ coordinates remain close to $(0,0)$ for small initial conditions from $\mathcal M_{n+1}.$

 If $Z\in \mathbb R^{n-1}$ is such that $\mce(Z) = 0$, then in view of Lemma \ref{lemmaEquilibriumPointsReducedSystem} and its proof,
  $(Z, h(Z))$  is an equilibrium point  on the center manifold $\mathcal M_{n-1}$, with $X^{\scriptscriptstyle T}\mathcal U = 0$ and
 $X^{\scriptscriptstyle T}\mathcal W= 0$,  implying that $(\delta_1,\delta_2)$ equations become:
 \bdima
 \left[ \begin{array}{c}\dot \delta_1 \\ \dot \delta_2 \end{array}\right] = \left[\begin{array}{cc}0 & 1 \\ -1 & 0 \end{array} \right] \left[ \begin{array}{c} \delta_1 \\  \delta_2 \end{array}\right]+
 \left[ \begin{array}{c}0 \\ 0 \end{array}\right].
 \edima
 The solutions are $e^{-it }\delta(0)$; that they remain near the origin for all $t>0$.

 Next, we establish stability in the flow-invariant region $\mathcal M_{n+1}\cap \{\mce >0\}$, using the
 underlying argument  that $(\delta _1, \delta _2)$ are the solutions of a linear harmonic oscillator driven by functions that are non-oscillatory in nature. Computationally, it is easier to follow the changes in $\gamma=(\gamma_1, \gamma_2)^{\scriptscriptstyle T}$ than the functions $ \frac{1}{n}X^T \mathcal{U}$ and  $ \frac{1}{n}X^T \mathcal{W}$ forcing the system. We use matrix notation for  (\ref{eqnNewCoordsDegFull}) and (\ref{eqnNewCoordsJustDelta}):
 \bdima
 \begin{array}{ll}
 \dot \gamma= J_{\scriptscriptstyle \gamma} \gamma + M_1 \left [ \begin{array}{c}
\frac{1}{n}X^{\scriptscriptstyle T}  \mathcal U   \\
\frac{1}{n}X^{\scriptscriptstyle T} \mathcal W
\end{array} \right ]& \mbox{with } J_{\scriptscriptstyle \gamma}= \left [ \begin{array}{cr}-\mathrm a& 1\\\mathrm a^2-1& -\mathrm a \end{array} \right ], \; \; M_1=\left [ \begin{array}{cr}0 & -1\\1& \mathrm a \end{array} \right ], \\
\\
 \dot \delta= J_{\scriptscriptstyle \delta} \delta + M_2 \left [ \begin{array}{c}
\frac{1}{n}X^{\scriptscriptstyle T}  \mathcal U   \\
\frac{1}{n}X^{\scriptscriptstyle T} \mathcal W
\end{array} \right ]& \mbox{with } J_{\scriptscriptstyle \delta}= \left [ \begin{array}{cr}0& 1\\-1& 0 \end{array} \right ], \; \; M_2=\left [ \begin{array}{cr} \mathrm a  & 1\\- 1& 0 \end{array} \right ].
\end{array}
 \edima
 Solve for $\displaystyle  \left [ \begin{array}{c} \frac{1}{n}  X^{\scriptscriptstyle T}  \mathcal U   \\
\frac{1}{n} X^{\scriptscriptstyle T} \mathcal W \end{array} \right ]$in  the $\dot \gamma$ block;
substitute it in $\dot \delta$:
$\displaystyle
\; \dot \delta= J_{\scriptscriptstyle \delta} \delta + M_2M_1^{-1}(\dot \gamma - J_{\scriptscriptstyle \gamma} \gamma),$ thus
$(\delta-M_2M_1^{-1} \gamma)^{\cdot}=J_{\scriptscriptstyle \delta} \delta - M_2M_1^{-1}J_{\scriptscriptstyle \gamma} \gamma
= J_{\scriptscriptstyle \delta} (\delta-M_2M_1^{-1} \gamma)+
J_{\scriptscriptstyle \delta}M_2M_1^{-1} \gamma - M_2M_1^{-1}J_{\scriptscriptstyle \gamma} \gamma
.$ Denoting $M=J_{\scriptscriptstyle \delta}M_2M_1^{-1}- M_2M_1^{-1}J_{\scriptscriptstyle \gamma},$ we get
\bdima
(\delta-M_2M_1^{-1} \gamma)^{\cdot}=J_{\scriptscriptstyle \delta} (\delta-M_2M_1^{-1} \gamma)+ M\gamma.
\edima

Define $\nu(t)\in \R^2$  as:
\beq \label{eqnDeltaNu}
\nu(t)=\delta(t) - M_2M_1^{-1} \gamma(t).
\eeq
We get that  $\nu $ satisfies
\beq
\dot \nu= \left [ \begin{array}{cr}0& 1\\-1& 0 \end{array} \right ]\nu+ M \gamma, \; \mbox{ and }
\lim_{t\to \infty} (\delta (t)-\nu(t)) =M_2M_1^{-1} \lim_{t\to \infty} \gamma (t)=0.
\eeq
The limit $\gamma\to 0$ is a consequence of the asymptotic behavior of the reduced system: all limit points are equilibrium points, thus have zero $\gamma $ components. Using variation of parameters, we obtain that
\beq
\label{variationparameters}
\nu(t) = e^{J_{\scriptscriptstyle \delta}t}\left ( \nu(0)+ \int_0^t e^{-J_{\scriptscriptstyle \delta}s}M \gamma(s)ds \right ),
\mbox{ with }  J_{\scriptscriptstyle \delta}= \left [ \begin{array}{cr}0& 1\\-1& 0 \end{array} \right ].
\eeq
Since $  e^{tJ_{\scriptscriptstyle \delta}}= \left [ \begin{array}{lr} \cos t & \sin t \\- \sin t & \cos t \end{array} \right ], $ the term $ e^{J_{\scriptscriptstyle \delta}t}\nu(0)$ contributes a small clockwise rotation (using complex notation, it equals
$e^{-it} \nu(0)$). To establish the stability and asymptotics on $\mathcal M_{n+1}$
we only need to investigate the convergence and the bounds for the integral
$\int _0 ^t \; \left [ \begin{array}{lr} \cos s & -\sin s \\\sin s & \cos s \end{array} \right ]
M \gamma(s) ds $
as $t\to \infty$, or equivalently, that of the integrals $\int _0 ^t e^{is} \gamma_j(s) \; ds$ for $j=1,2$ and points in $\mathcal M_{n+1} \cap \{\mce >0\}.$

We use the $\dot \gamma_j$ estimate of Theorem \ref{ThmApproximationTheorem} to eventually express $\gamma_j$ as the difference of two decreasing functions. We have  $\displaystyle \frac{d\gamma}{dt} = \mce \mco (|Z|_r)=\mce \mco (\sqrt{\mcd _L})$.  The last equality follows from Lemma \ref{lemmaDLBounds}. Let $C>0$ be such that
\begin{equation} \label{eqnBoundDirGamma}
 \left |\frac{d\gamma}{dt} \right |  \leq C \mce \sqrt{\mcd _L}\; \mbox{ for all sufficiently small } Z.
\end{equation}
Introduce (auxiliary) functions
$
B_j(t)= \sqrt {\mcd_L (t)} -\sqrt{ \mcd_{L, \infty}}+ \frac{a}{4C}\gamma_j (t),
$
where $\mcd_{L, \infty}= \lim _{t\to \infty} \mcd_L (t)$. The limit exists since the function $\mcd_L $ is decreasing, per (\ref{EqnDifEqnDispersion}), and bounded  below by zero. Then
\beq \label{gamma_as_difference}
\gamma_j (t)= \frac{4C}{\mathrm a} B_j(t) - \frac{4C}{\mathrm a} ( \sqrt {\mcd_L (t)} -\sqrt{ \mcd_{L, \infty}}).
\eeq
The functions in the above difference have zero limit as $t$ approaches infinity, and the subtrahend is decreasing. We claim that $B_j$ are decreasing functions as well. Using (\ref{EqnDifEqnDispersion}),
\begin{equation}
\begin{split}
\frac{dB_j}{dt} & = -\frac{1}{2\sqrt {\mcd_L (t)}} 2 \mathrm a \; \mce(Z)  \mcd_L(Z) ( 1 + \mathcal O(|Z|) )
+\frac{\mathrm a}{4C}\frac{d\gamma_j}{dt} \\
&\leq  -\mathrm a\; \mce(Z) \sqrt {\mcd_L(Z) } ( 1 + \mathcal O(|Z|) + \frac{\mathrm a}{4C} C \mathcal E(Z) \sqrt{D_L(Z)} \\
& =
\mathrm a \; \mce(Z) \sqrt {\mcd_L(Z) }\left(\frac{1}{4}+\mathcal O(|Z|)-1 \right)\leq 0 \mbox{ for small } Z.
\end{split}
\end{equation}
Thus the functions $B_j(t)$ and $( \sqrt{ \mcd_L (t)}-\sqrt{ \mcd_{L,\infty}})$ are decreasing towards zero.   It follows from the Dirichlet test that their improper integrals
$  \int _0 ^t e^{is} B_j(s) ds $ and $ \int _0 ^t e^{is} ( \sqrt{ \mcd_L (s)}-\sqrt{ \mcd_{L, \infty}} )\;  d s$
converge as $t\to\infty$ and their respective ranges are bounded by $4B_j(0)$ and
 $4\left(\sqrt{ \mcd_L (0)}-\sqrt{ \mcd_{L,\infty}} \right).$ Using (\ref{gamma_as_difference}),
 we conclude that for $j=1,2$,
$\int _0 ^t e^{is} \gamma_j(s) ds $ converge when $t\to \infty$ and remain within distance
$K \max(|\gamma(0)|,\sqrt{D_L(0)})$ from the origin. Set
\bdima
\delta_{\star} = \nu(0)+ \lim_{ t \to \infty } \int_0^t e^{-J_{\scriptscriptstyle \delta}s}M \gamma(s)\; ds.
\edima
From the variation of parameters equation (\ref{variationparameters})
we get $\lim _{ t \to \infty } \left (\nu(t)-e^{-it } \delta_{\star} \right ) =0.$
Using  $\delta=  \nu+ M_2 M_1^{-1} \gamma $, (\ref{eqnDeltaNu}), conclude that  for some constant $K'>0$, $\delta(t)$ remains within a distance of $K' \max(|\delta(0)|,|\gamma(t)|,\sqrt{D_L(0)})$ from the origin. The stability of the full system near the origin follows since $\gamma(t)$ was proven to remain small. Moreover, $\lim _{ t \to \infty } \left (\delta(t)-e^{-it } \delta_{\star} \right ) =0.$
This proves that statement of the theorem for the $\delta$ components of the flow $\phi_t$. The convergence to an equilibrium point $Q_\star$ of the first $4n-2$ components was established by the reduced-system analysis.
\end{proof}
Note that the counterclockwise rotation $e^{-it } \delta_{\star} $ obtained in the $(Z, \dots, \delta)$ coordinates (and in the rotating frame $X_k, Y_k$) corresponds to translating the center of mass of the swarm by $\delta_\star$ in the original, physical coordinates.

{\it Stability of every rotating state follows form the degenerate and non-degenerate cases.} The fixed points of the rotating-frame system $X_k, Y_k$, are stable.  Moreover, starting in a neighborhood of a rotating state, the system will necessarily converge to a rotating state.
In the $\epsilon-\delta$ formulation of stability, the value of $\delta$ depends on the starting rotating state configuration, whereas Theorem \ref{TheoremMainResultIntro} claims it does not. This is resolved via a compactness argument.
%
\begin{remark}\label{UniformStabilityOdd}
Ring states centered at the origin  correspond to fixed points
from the compact set  $\mathcal F$ in the rotating-frame of Section \ref{SectionRingStateChangeOfCoordinates}.
For each  $P_0 =\{ [X_{0,k}+i Y_{0,k}, 0+i0] \}\in  \mathcal F $  and $\epsilon >0,$  there exists $\delta = \delta (P_0, \epsilon ) >0 $ such that if an initial condition $Q_0 =\{[X_{ k}+i Y_{ k}, \dot X_{ k}+i \dot Y_{ k}]\} $ is in the ball $B(P_0,\delta)$ of radius $\delta $ from $P_0$ then the solution $\phi_t (Q_0)$ of (\ref{eqnXYEquilibriumSystem})  satisfies
$ |\phi_t (Q_0) - P_0| <\epsilon $ for all $t>0.$

Let $\epsilon >0$. For each point $P\in\mathcal F$ we construct $\delta (P, \epsilon )$ as above and, by compactness of $\mathcal F$, cover $\mathcal F$ with finitely many balls $B(P_m, \delta( P_m, \epsilon)).$ Let $\delta=\frac{1}{2} \min \delta( P_m, \epsilon).$
We get that for any $P_0 \in \mathcal F$, if $|Q - P_0| < \delta$ then $ |\phi_t (Q) - P_0| <2\epsilon $ for all $t>0.$
The translation invariance of System (\ref{eqnMainModel}) ensures that similar estimates hold for ring states centered anywhere.
\end{remark}

{\it Remark \ref{UniformStabilityOdd}, the results of Sections \ref{SectionNonDegenerateRingState}  and \ref{SectionDegenerateRingState} complete the proof of stability of rotating states, Theorem \ref{TheoremMainResultIntro}. } Moreover, if the swarm consists of an odd number of particles, there exists $\eta >0$ such that small perturbations of rotating states converge $e^{-\eta t}$ fast to a nearby rotating state. If the number of agents is even, the rate of  convergence to degenerate rotating states is $1/\sqrt t$.
%
%


\section*{Appendix}
Link to Supplement with Mathematica code will be included here.
%




The views expressed in this article are those of the authors and do not reflect the official policy or position of the United States Naval Academy, Department of the Navy, Department of Defense, or the U.S. Government

\end{document}